\documentclass{article}
\usepackage{amsmath, amssymb, amsfonts, epsfig, graphicx, tcolorbox, enumitem, stmaryrd, graphicx,tikz}
\usepackage{multicol, float, lipsum, blindtext, amsthm, caption, subcaption,upgreek, comment}

\newtheorem{theorem}{Theorem}[section]
\newtheorem{corollary}[theorem]{Corollary}
\newtheorem{definition}[theorem]{Definition}
\newtheorem{proposition}[theorem]{Proposition}

\title{Isoperimetric 3- and 4-bubble results on $\mathbb{R}$ with density $|x|$}
\author{Evan Alexander, Emily Burns, John Ross, Jesse Stovall, Zariah Whyte}
\begin{document}

\maketitle

\section{Introduction}

%\begin{itemize}
%    \item Brief history of  isoperimetric problem
%    \item Introduce density problems
%    \item Work done on the real number line specifically the single/double bubble in Huang et al; the log convex work; and the triple bubble
%    \item state main theorem
%    \item organization of this paper
%\end{itemize}

The classic isoperimetric problem asks one to find, out of all simple closed curves in the plane with fixed perimeter $P$, the curve that encloses the maximal amount of area $A$. Equivalently, the iso-area problem asks one to find the simple-closed curve with minimal perimeter out of all simple closed curves with a fixed enclosed area. These two problems, who share a solution of a circle of appropriate size, have been known since antiquity. Classically, a solution to an isoperimeteric problem is called an isoperimetric region or a  ``bubble'' due to how soap films naturally form isoperimetric surfaces in $\mathbb{R}^3$.

Isoperimetric problems can be generalized in a number of ways, including the introduction of multiple regions to optimize, looking at closed hypersurfaces (as analogues of simple closed curves) within $\mathbb{R}^n$ or other manifolds; and introducing density functions that affect how perimeter and area are measured. Although studying isoperimetry with density functions is a more recent development, much has already been said. By an argument in \cite{MorganPratelli13}, it is known that a perimeter-minimizing $n$-bubble solution will exist on $\mathbb{R}^m$ when using a density function that radially increases to infinity. On $\mathbb{R}^1$, early isoperimetric results for a single bubble were found in \cite{BayleCaneteMorganRosales06}. Additionally, both \cite{ChambersBongiovanniETAL18} and \cite{HuangMorgan19} have found single- and double-bubble solutions on $\mathbb{R}^1$ with density. In \cite{HuangMorgan19}, the single- and double-bubble solutions were found for density $|x|^p$, extending the single-bubble result from \cite{BoyerBrownChambersLovingTammen16} in the 1-dimensional case. In \cite{ChambersBongiovanniETAL18},
single- and double-bubble results were found for density functions that were log-convex. In recent work, \cite{Sothanaphan20} extended these results to identify possible triple-bubbles.

Our work will examine the multi-bubble isoperimetric problem on $\mathbb{R}^1$ with a prescribed density function of $|x|$. Our main results, Theorem \ref{Thm:TripleBubble} and Theorem \ref{Thm:FourBubble}, will provide the least-perimeter way to enclose and separate $n$ regions of specified masses for $n=3, 4$. We will also record the results for $n=1,2$ first identified in \cite{HuangMorgan19} (proven for completeness in this paper). The results for the first four isoperimetric problems are displayed in Figure \ref{fig:FirstFourBubbles}. Although we will wait to formally state the result, we note that the masses alternate positive/negative across the origin as they increase in weighted mass. This leads to a natural conjecture that the pattern will continue for larger $n$-bubbles; this conjecture is addressed and positively proven in forthcoming work \cite{Ross21}.

\begin{figure}[H]
    \centering
    \includegraphics[scale=0.75]{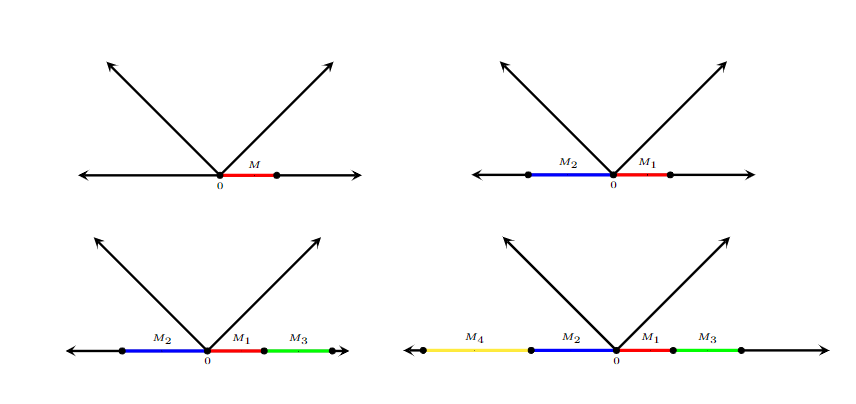}
    \caption{Solutions for the 1-, 2-, 3-, and 4-bubble isoperimetric problem on $\mathbb{R}$ with density $|x|$. For these masses, note that $M_1 \leq M_2 \leq M_3 \leq M_4$}
    \label{fig:FirstFourBubbles}
\end{figure}

Although working in such a low dimension reduces the complexity of our regions, there are still complications we must deal with. For example, a priori, each region may  consist of multiple intervals which may or may not share endpoints with other regions. In \cite{ChambersBongiovanniETAL18}, it was found that the 2-bubble had either two or three intervals, depending on the actual sizes (not just relative sizes) of the two regions in question. An early step in our proof will be to show that we can shrink our total perimeter by consolidating each region into, at most, two intervals. This will reduce our $n$-bubble candidates to, at most, $2n$ total intervals. From there, we examine a number of tools we can use to further consolidate our number of intervals. These tools are enough to fully prove the 3-bubble. For the 4-bubble, we those tools still leave us with 5 intervals, and so an explicit comparison is needed. A crucial element of this is a set of ``optimal arrangment'' propositions that tell us the perimeter-minimizing way to order 4, 5, or 6 adjacent regions. These propositions require many cases and are included as an appendix. % Finally, we optimize our ordering of the $N$ regions, and identify that the origin is best left as one of the endpoints of the central region. These final steps are handled in \textcolor{red}{Some proposition to be identified below}Triple and Quadruple

\subsection{Acknowledgements}

This work was completed as part of Southwestern University's SCOPE program, whose aim is to encourage and fund student-faculty collaborations. All authors are grateful for the support of Southwestern University and the SCOPE program.

\section{Preliminaries}

In this section we include some basic definitions and early results.

\begin{definition}
A \textbf{density} function on $\mathbb{R}$ is simply a nonnegative function $f$.
\end{definition}
%When we want to draw attention to a specific density function $f$ on $\mathbb{R}$, we will often refer to the pair $(\mathbb{R}, f)$.

\begin{definition} Given an interval $[a,b]$ and a density $f$, the \textbf{weighted mass} of the interval with respect to the density is $\int_a^b f$.  The \textbf{weighted perimeter} with respect to the density is defined to be $f(a) + f(b)$.
\end{definition}
Note that weighted mass will often simply be called mass (and occasionally, by abuse of language, be called area or volume), and weighted perimeter will often be called perimeter. When looking at an interval $[a,b]$, we will exclusively use the word \textbf{length} to refer to the quantity $b-a$.

\begin{definition}
A \textbf{region} on $\mathbb{R}$ is a collection of disjoint intervals $[a_i, b_i]$ with  total mass $M$ and perimeter $P$ calculated as
    \begin{align}
        M &= \displaystyle\sum_i \int_{a_i}^{b_i} f\\
        P &= \displaystyle\sum_i f(a_i) + f(b_i)
    \end{align}
\end{definition}
We remark that, under most natural density functions (in particular: density functions that have a positive lower bound), regions with finite mass or finite perimeter must necessarily be made up of only finitely many intervals. With other density functions (like $f(x) = |x|$, for example), it is possible to have a region of finite mass consisting of infinitely many intervals. In such a scenario, a point of zero density will act as an accumulation point for the interval endpoints.
\begin{definition}
A region $R$ with mass $M$ and perimeter $P$ is said to be \textbf{isoperimetric} (also referred to as a \textbf{bubble}) if, out of all possible regions with mass $M$, $R$ has the least perimeter.
\end{definition}

We can also consider multiple regions on the same number line. In such a situation, each region consists of (possibly multiple disjoint) intervals. Two intervals from different regions are either completely disjoint, or meet at a single endpoint. In addition to measuring the mass and perimeter of each region, we can measure the \textbf{total (weighted) perimeter} by summing the perimeter of each of the regions. In doing so, an endpoint where two intervals meet is only counted once. 

\begin{definition}
A configuration of n regions $R_1, \dots, R_n$ with masses $M_1, \dots, M_n$ is said to be \textbf{isoperimetric} (also referred to as an \textbf{$n$-bubble}) if, out of all possible configurations of $n$ regions with these same masses, this configuration has the least total perimeter.
\end{definition}
2-bubbles and 3-bubbles are traditionally called a \textbf{double bubble} and \textbf{triple bubble}.

The major aim of this paper is to prove a particular configuration of intervals is isoperimetric in the case of the $3$-bubble and $4$-bubble for density $|x|$. To begin exploring isoperimetry in earnest, we begin by reducing the total number of intervals that will appear in a perimeter-minimizing configuration of $n$ regions. A priori, each region could consist of one or more intervals, and that there might be ``empty'' intervals that do not correspond to any particular region. However, we quickly eliminate this possibility by showing that perimeter is minimized by ``condensing'' our intervals near the origin.

\begin{proposition}
On $\{ x\in \mathbb{R} : x \geq 0 \}$, with an increasing density function $f$: suppose we have $n$ regions $R_i$ with masses $M_i$. Then one can create a new configuration of regions which simultaneously preserves mass and decreases perimeter, so that each region consists of a single interval and there are no empty intervals.
\label{prop:ConsolidateOnR+}
\end{proposition}

\begin{proof}
%Suppose we have an arbitrary configuration on $n$ regions, each consisting of (possibly multiple) intervals on the non-negative axis $x \geq 0$. WLOG, we identify $R_1$ to have a maximum value (and rightmost endpoint) of $b_1$, $R_2$ to have a maximum value of $b_2$, and so on, named so that $b_1 < b_2 < \dots < b_n$. As seen in the image \textcolor{red}{(to be created)}, we can create a new configuration in which regions consist of a single interval, each of which are adjacent and ordered so that region $R_i$ sits to the left of region $R_j$ if $i < j$. Furthermore, this can clearly be done in such a manner that the mass of each region is preserved.

Suppose we have an arbitrary configuration on $n$ regions, each consisting of (possibly multiple) intervals on the non-negative axis $x \geq 0$. WLOG, we identify $R_1$ to have a maximum value (and rightmost endpoint) of $b_1$, $R_2$ to have a maximum value of $b_2$, and so on, named so that $b_1 < b_2 < \dots < b_n$. As seen in Figure \ref{Figure3Consolidate}, we can create a new configuration in which regions consist of a single interval, each of which are adjacent and ordered so that region $R_i$ sits to the left of region $R_j$ if $i < j$. Furthermore, this can clearly be done in such a manner that preserves the mass of each region.

%%%%%%%%%%%%%%%%%%%%%%%%%%%%%%%%%%%%%%%%
%
%  Figure 3: Consolidation
%
%%%%%%%%%%%%%%%%%%%%%%%%%%%%%%%%%%%%%%%%
     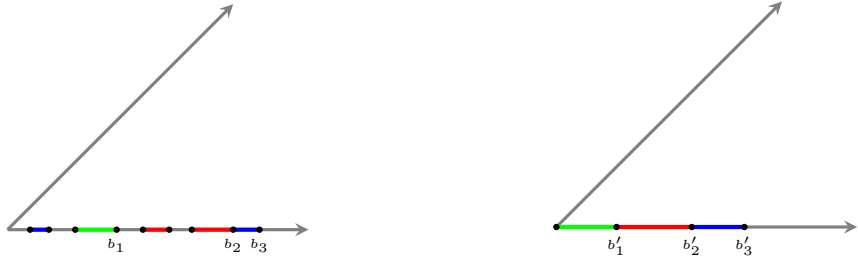
\begin{figure}[H]
    \centering
    \begin{subfigure}[h]{0.4\textwidth}    
    \begin{tikzpicture}
    %x axis and |x|:
    \draw[gray, -stealth,very thick](0,0) -- (3,3);
    \draw[gray, -stealth,very thick] (0,0) -- (4,0);

    % green intervals:
    \filldraw[green,ultra thick] (.9,0) -- (1.45,0);

    % blue intervals:
    \filldraw[blue,ultra thick] (.3,0) -- (.55,0);
    \filldraw[blue,ultra thick] (3,0) -- (3.35,0);

    % red intervals:
    \filldraw[red,ultra thick] (2.15,0) -- (1.8,0);
    \filldraw[red,ultra thick] (2.45,0) -- (3,0);

    % yellow intervals:
%    \filldraw[yellow,ultra thick] (3.6,0) -- (3.85,0);

    %POINTS: 
    \filldraw (.9,0) circle (1pt);
    \filldraw (1.45,0) circle (1pt) node[anchor=north] {\tiny $b_1$};
    \filldraw (.3,0) circle (1pt) ;
    \filldraw (.55,0) circle (1pt);
    \filldraw (2.15,0) circle (1pt);
    \filldraw (1.8,0) circle (1pt);
    \filldraw (3,0) circle (1pt) node[anchor=north] {\tiny $b_2$};
    \filldraw (3.35,0) circle (1pt) node[anchor=north] {\tiny $b_3$};
    \filldraw (2.45,0) circle (1pt);
%    \filldraw (3.6,0) circle (1pt);
%    \filldraw (3.85,0) circle (1pt);

    \end{tikzpicture}
	\end{subfigure}
 \hfill
     \begin{subfigure}[h]{0.4\textwidth}    
    \begin{tikzpicture}
    %x axis and |x|:
    \draw[gray, -stealth,very thick](0,0) -- (3,3);
    \draw[gray, -stealth,very thick] (0,0) -- (4,0);

    % green intervals:
    \filldraw[green,ultra thick] (0,0) -- (0.8,0);

    % blue intervals:
    \filldraw[blue,ultra thick] (1.8,0) -- (2.5,0);

    % red intervals:
    \filldraw[red,ultra thick] (0.8,0) -- (1.8,0);
%    \filldraw[red,ultra thick] (2.45,0) -- (3,0);

    % yellow intervals:
%    \filldraw[yellow,ultra thick] (3.6,0) -- (3.85,0);

    %POINTS: 
	\filldraw (0,0) circle (1pt);    
    \filldraw (0.8,0) circle (1pt) node[anchor=north] {\tiny $b_1'$};
    \filldraw (1.8,0) circle (1pt) node[anchor=north] {\tiny $b_2'$};
    \filldraw (2.5,0) circle (1pt) node[anchor=north] {\tiny $b_3'$};

    \end{tikzpicture}
    \end{subfigure}
    \caption{We see how an arbitrary arrangement (on the left) can be consolidated (shown the right) in a manner that lowers perimeter.}
	\label{Figure3Consolidate}
\end{figure}

Because of how we have constructed these regions, it is clear that the new rightmost endpoint $b_j'$ of the region $R_j$ satisfies $0 < b_j' \leq b_j$. Since $f$ is assumed to be increasing, we know $f(b_j') \leq f(b_j)$. This gives us the following inequalities for the new and original perimeter:
\begin{align}
P_{new} = \displaystyle \sum_{j=1}^n f(b_j') &\leq \displaystyle \sum_{j=1}^n f(b_j) \leq P_{orig} 
\end{align}
which completes the proof.
\end{proof}
We will collectively use the word \textbf{reconfiguration} to describe exercises in rearranging and consolidating intervals, when in such a way that preserves weighted masses.
\begin{corollary}
If a configuration of $n$ regions is isoperimetric, then the configuration consists of at most $2n$ adjacent intervals, with the origin contained in at least one interval (either in the interior or as an endpoint).
\label{Cor:CondenseToTheOrigin}
\end{corollary}

\begin{proof}
Proposition \ref{prop:ConsolidateOnR+} says that perimeter can be lowered by reconfiguring regions on each side of the origin, so that there are at most $n$ intervals on the positive and $n$ on the negative side. Furthermore, as our consolidation from before occurs at the origin, it is clear that the origin will either be between two different regions (in which case it's an endpoint of each), or in a single region (with some mass to be found on both the positive and negative side).
\end{proof}

\begin{definition}
A configuration of $n$ regions that consists of at most $2n$ adjacent intervals (with at most 2 intervals per region), with the origin contained either in an interior of an interval or at the endpoint of (at least one) interval, is said to be a \textbf{condensed configuration of $n$ regions}.
\end{definition}

Beyond simply knowing that our regions accumulate near the origin, we also learn that we can reorder them so that they ``grow'' from smallest to largest mass as we move further out from the origin. This is summarized in the following

\begin{proposition}[Transposition Lemma]
\label{prop:TranspositionLemma}
On $\{ x \geq 0 \}$, consider two intervals $R_i = [a_i, b_i]$ of mass $M_i$. Suppose each $R_i$ has mass $M_i$, with $M_2 < M_1$. Finally, suppose $a_2 = b_1$. Then we would reduce total perimeter by switching the relative positions of $R_1$ and $R_2$.% weighted perimeter is minimized when the interval for $R_1$ lies to the left of the interval for $R_2$.
\end{proposition}

\begin{proof}
This is seen immediately in Figure \ref{Figure4Transpose}. Keeping the leftmost and rightmost endpoints fixed, and transposing the two intervals if necessary, we see that the perimeter is lowered when the interval with less mass is placed closer to the origin (as it moves the shared endpoint closer to the origin).
\end{proof}

%%%%%%%%%%%%%%%%%%%%%%%%%%%%%%%%%%%%%%%%
%
%  Figure 4: Transposition
%
%%%%%%%%%%%%%%%%%%%%%%%%%%%%%%%%%%%%%%%%

     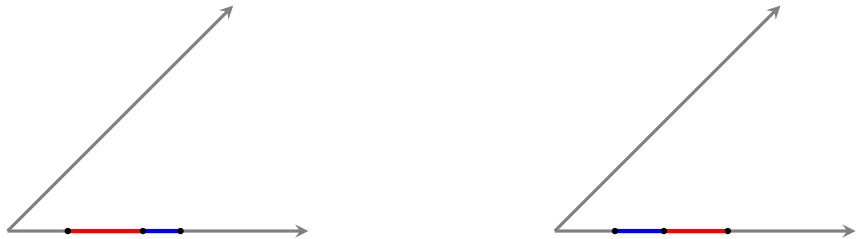
\begin{figure}[H]
    \centering %center the figure to be created
    \begin{subfigure}[h]{0.4\textwidth}    
    
    \begin{tikzpicture}
    %x axis and |x|:
    \draw[gray, -stealth,very thick](0,0) -- (3,3);
    \draw[gray, -stealth,very thick] (0,0) -- (4,0);

    % blue intervals:
    \filldraw[blue,ultra thick] (1.8,0) -- (2.3,0);

    % red intervals:
    \filldraw[red,ultra thick] (0.8,0) -- (1.8,0);
%    \filldraw[red,ultra thick] (2.45,0) -- (3,0);

    % yellow intervals:
%    \filldraw[yellow,ultra thick] (3.6,0) -- (3.85,0);

    %POINTS: 
    \filldraw (0.8,0) circle (1pt);
    \filldraw (1.8,0) circle (1pt);
    \filldraw (2.3,0) circle (1pt);

    \end{tikzpicture}
	\end{subfigure}
 \hfill
     \begin{subfigure}[h]{0.4\textwidth}    
    
    \begin{tikzpicture}
    %x axis and |x|:
    \draw[gray, -stealth,very thick](0,0) -- (3,3);
    \draw[gray, -stealth,very thick] (0,0) -- (4,0);

    % red intervals:
    \filldraw[red,ultra thick] (1.45,0) -- (2.3,0);

    % blueintervals:
    \filldraw[blue,ultra thick] (0.8,0) -- (1.45,0);
%    \filldraw[red,ultra thick] (2.45,0) -- (3,0);

    % yellow intervals:
%    \filldraw[yellow,ultra thick] (3.6,0) -- (3.85,0);

    %POINTS: 
    \filldraw (0.8,0) circle (1pt);
    \filldraw (1.45,0) circle (1pt);
    \filldraw (2.3,0) circle (1pt);

    \end{tikzpicture}
	\end{subfigure}
 
    \caption{Transposition of adjacent intervals: If we transpose $R_1$ and $R_2$ without changing their global location, the lower perimeter results from placing the smaller interval closer to the origin (as the only endpoint to move is the internal endpoint). Iterating the transposition lemma guarantees that, on one side of the interval, regions will be ordered (according to weighted area) from smallest to largest as we move away from the origin.}
    \label{Figure4Transpose}
\end{figure}

We conclude this section with an important  observation about comparing intervals of appropriate masses. Although the proposition is stated for the positive real number line, it is clear (due to the symmetry of $|x|$) that there are equivalent statements for comparing intervals on either side of the number line; in such a comparison, all that matters is which interval's ``inner'' endpoint is positioned closer to the origin.

\begin{proposition}[Length and Endpoint Inequalities]
\label{prop:LengthEndpointInequality}
On $\mathbb{R}$ with density $|x|$: Consider two intervals $R_1 = [a_1, b_1]$,  $R_2 = [a_2, b_2]$. Suppose that each of these intervals contains the same weighted mass (so $M_1 = M_2$), and suppose that $a_1 < a_2$. Then we get the following two inequalities: 
%    $$
 %   b_1 < b_2 \qquad \text{and} \qquad b_1 - a_1 > b_2 - a_2
  %  $$
    \begin{align}
        b_1 &< b_2\\
      b_1 - a_1 &> b_2 - a_2
    \end{align}
\end{proposition}

\begin{proof}
The first inequality is an immediate consequence of the two intervals enclosing the same mass. For the second inequality, note that since our density function is increasing as we move away from the origin, the  average density of $[a_2, b_2]$ will be larger than that of $[a_1, b_1]$: since they each contain the same amount of mass, this means the length of the interval $[a_2, b_2]$ will be smaller.
\end{proof}

\begin{corollary}
The first inequality above continues to hold true in the cases where $M_2 \geq M_1$. The second inequality also holds true in the cases where $M_2 \leq M_1$.
\end{corollary}

\section{Rearranging masses across the origin}

As the previous section shows, an $n$-bubble will have at most $2n$ intervals, and each region within an $n$-bubble will have at most two intervals (one on each side of the origin). In this section, we will develop strategies for lowering perimeter by reconfigurations that move intervals across the origin (and consolidating two intervals from the same region into a single interval). As a corollary to our results, we will provide concise proofs for the single and double bubble results for $\mathbb{R}$ with density $|x|$, previously proven in \cite{HuangMorgan19}.

\begin{proposition}[Mass Stealing Lemma]
\label{prop:OuterMassStealing}
 On $(\mathbb{R}, |x|)$, suppose we have a condensed configuration of $n$ regions. Suppose the outermost interval on the negative side of the origin comes from a region that also has an interval on the positive side of the origin. Then there exists a modified configuration of the same regions, but with less perimeter, achieved by moving all mass from this region to the negative side.
\end{proposition}

\begin{proof}
Suppose our $n$ regions are arranged as adjacent intervals with endpoints $-L_1, -L_2,..., -L_l$ on $\{x < 0\}$, and $R_1, R_2,..., R_m$ on $\{x > 0\}$. Furthermore, suppose the interval $[-L_l, -L_{l-1}]$ is part of a region that has additional mass in an interval on the right-hand side of the number line. Our strategy will be to move the mass from the RHS to adjoin it adjoin it to $[-L_l, -L_{l-1}]$ on the LHS, while sliding the outer intervals on the right closer to the origin to fill the empty space. 

The interval we move from the RHS is of the form $[R_{i-1}, R_i]$ for some $i$ between $1$ and $m$. The initial configuration will have a total perimeter of $$TP = L_l + L_{l-1} + ... + L_2 + L_1 + R_1 + R_2 + ... + R_{m-1} + R_m.$$  The new configuration will have total perimeter of $$TP = L_l' + L_{l-1} + ... + L_2 + L_1 + R_1 + R_2 ... + R_{i-1} + R_{i+1}' + ... + R_m'.$$
In this calculation, the endpoint $L_l$ has been shifted to the left, resulting in $L_l'$; for $j > i$, the endpoints $R_j$ have shifted to the left, resulting in $R_j'$; and the endpoint $R_i$ has disappeared.

\begin{figure}[H]
     \centering
     %\begin{subfigure}[h]{0.45\textwidth}
     %    \centering
     
    \begin{tikzpicture}
    %LINES:
    %|x| and x axis
    \draw[stealth-stealth,very thick] (-4,0) -- (4,0); 
    \draw[stealth-,very thick] (-1,1) -- (0,0);
    \draw[-stealth,very thick](0,0) -- (1,1);
    %2 intervals of one bubble
    \filldraw[red,ultra thick] (-2.7,0) -- (-2,0);
    \filldraw[red,ultra thick] (1.6,0) -- (2.05,0);

    %POINTS: 
    %(neg)
    \filldraw (-2.7,0) circle (2pt) node[anchor=north] {\tiny $-L_l$};
    \filldraw (-2,0) circle (2pt) node[anchor=north] {\tiny $-L_{l-1}$};
    \filldraw (-1.5,0) circle (.2pt) node[anchor=south] {. . .};
    \filldraw (-1,0) circle (2pt) node[anchor=north] {\tiny $-L_2$};
    \filldraw (-.5,0) circle (2pt) node[anchor=north] {\tiny $-L_1$};
    %(pos)
    \filldraw (.4,0) circle (2pt) node[anchor=north] {\tiny $R_1$};
    \filldraw (.8,0) circle (2pt) node[anchor=north] {\tiny $R_2$};
    \filldraw (1.2,0) circle (.2pt) node[anchor=south] {. . .};
    \filldraw (1.6,0) circle (2pt) node[anchor=north] {\tiny $R_{i-1}$};
    \filldraw (2.05,0) circle (2pt) node[anchor=north] {\tiny $R_i$};
    \filldraw (2.4,0) circle (.2pt) node[anchor=south] {. . .};
    \filldraw (2.8,0) circle (2pt) node[anchor=north] {\tiny $R_{m-1}$};
    \filldraw (3.5,0) circle (2pt) node[anchor=north] {\tiny $R_{m}$};

    \end{tikzpicture}
    \caption{Configuration before joining $[R_{i-1},R_i]$ with $[-L_l,-L_{l-1}]$}
    \label{fig:beforeRiSwitch}
    \end{figure}
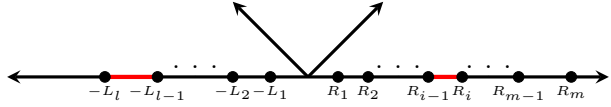
    % \end{subfigure}
    % \hfill
    % \begin{subfigure}[h]{0.45\textwidth}
    %     \centering

%\end{figure}
    
 %           In this initial configuration we have a total perimeter of $$TP = L_l + L_{l-1} + ... + L_2 + L_1 + R_1 + R_2 + ... + R_{i-1} + R_i + ... + R_{m-1} + R_m$$
  %          If we take $[R_{i-1}, R_i]$ and consolidate it with its matching interval of the same region, $[-L_l, -L_{l-1}]$, we will be left with this new configuration below which has total perimeter of $$TP = L_l' + L_{l-1} + ... + L_2 + L_1 +  + ... + R'_{m-1} + R'_m$$
            
            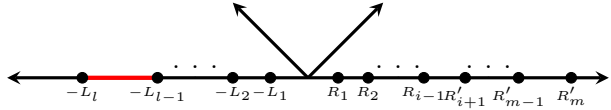
\begin{figure}[H]
    \centering
    \begin{tikzpicture}
    %LINES:
    %|x| and x axis
    \draw[stealth-stealth,very thick] (-4,0) -- (4,0); 
    \draw[stealth-,very thick] (-1,1) -- (0,0);
    \draw[-stealth,very thick](0,0) -- (1,1);
    %interval of one bubble
    \filldraw[red,ultra thick] (-3,0) -- (-2,0);
    
    %POINTS: 
    %(neg)
    \filldraw (-3,0) circle (2pt) node[anchor=north] {\tiny $-L_l$};
    \filldraw (-2,0) circle (2pt) node[anchor=north] {\tiny $-L_{l-1}$};
    \filldraw (-1.5,0) circle (.2pt) node[anchor=south] {. . .};
    \filldraw (-1,0) circle (2pt) node[anchor=north] {\tiny $-L_2$};
    \filldraw (-.5,0) circle (2pt) node[anchor=north] {\tiny $-L_1$};
    %(pos)
    \filldraw (.4,0) circle (2pt) node[anchor=north] {\tiny $R_1$};
    \filldraw (.8,0) circle (2pt) node[anchor=north] {\tiny $R_2$};
    \filldraw (1.2,0) circle (.2pt) node[anchor=south] {. . .};
    \filldraw (1.54,0) circle (2pt) node[anchor=north] {\tiny $R_{i-1}$};
    \filldraw (2.09,0) circle (2pt) node[anchor=north] {\tiny $R_{i+1}'$};
    \filldraw (2.4,0) circle (.2pt) node[anchor=south] {. . .};
    \filldraw (2.8,0) circle (2pt) node[anchor=north] {\tiny $R_{m-1}'$};
    \filldraw (3.5,0) circle (2pt) node[anchor=north] {\tiny $R_{m}'$};

    \end{tikzpicture}
    \caption{Configuration after joining results in new interval $[-L'_l,-L_{l-1}]$}
    \label{fig:afterRiSwitch}
%    \end{subfigure}
\end{figure}

There are two possibilities:

\begin{itemize}
    \item \textbf{Case 1: $R_{i-1} \leq L_{l}$}. In this case, we see that that $R_i - R_{i-1} \geq L_l' - L_l$ by Proposition \ref{prop:LengthEndpointInequality}. This means that $R_i + L_l > L_l'$. Furthermore, for each $j > i$, we have $R_j' < R_j$. This is enough to tell us that our new perimeter is smaller than our original perimeter.

    \item \textbf{Case 2: $R_{i-1} > L_{l}$}. In this case, since $L_l < R_{i-1}$, we get $L_l' < R_i$ (again by Proposition \ref{prop:LengthEndpointInequality}). Using reasoning similar to above, we see that the new perimeter will again be smaller than our original perimeter. %Therefore, after the re-consolidation, there will be less total perimeter.

%            Our new total perimeter is: $$TP = L_l' + L_{l-1} + ... + L_2 + L_1 + R_1 + R_2 + ... + R_{i-1} + R_{i+1}' + ... + R_{m-1}' + R_m'$$
            
            %If $L_l \geq R_{i-1}$ then $[R_{i-1}, R_i]$ is no smaller than $[L_l, L_l']$, and so $R_i - R_{i-1} \geq L_l' - L_l$. Then, certainly, $R_i > R_i - R_{i-1} \geq L_l' - L_l$, giving that $R_i + L_l > L_l'$. Furthermore, because intervals farther right of the $i^{th}$ interval slid down to $R_{i-1}$, we can also say that $R_{i+2}' < R_{i+2}$ and so on until $R_{m}' < R_{m}$. Therefore, there is less total perimeter.
            
            %\bigskip
            %\item \textbf{Case 3b)} \emph{Middle $i^{th}$ interval where $L_l < R_i$}:
            
            %In this final case, if$L_l < R_{i-1}$ then $L_l' < R_i$ by Lemma 2.1. Therefore, after the re-consolidation, there will be less total perimeter.

\end{itemize}

\end{proof}

The result we just proved does not depend on the relative size of the intervals involved (although, by Proposition \ref{prop:TranspositionLemma}, it can be made so that the outer interval contains the largest mass on its side of the origin). 

Proposition \ref{prop:OuterMassStealing} immediately gives us the result for the single and double bubbles. Previously proven in \cite{HuangMorgan19}, we include the results here for completeness.

\begin{theorem}[1- and 2-Bubble Theorem with Density $|x|$]
On $\mathbb{R}$ with density $|x|$: The $n$-bubble solution to the isoperimetric problem for $n=1$ is a single interval with one endpoint at the origin. For $n=2$, the solution is two intervals, each with an endpoint at the origin. 
\end{theorem}

\begin{proof}
\textbf{For the single bubble}: If we start with any arbitrary configuration for one region with mass $M$, we are starting with an arbitrary (finite or countably infinite) union of intervals (with the only possible accumulation point for the endpoints being the origin). We can first consolidate near the origin using Corollary \ref{Cor:CondenseToTheOrigin}, resulting in at most two intervals: one from $[-a, 0]$ and one from $[0,b]$. Then, using Proposition \ref{prop:OuterMassStealing}, this can be converted into a single interval at $[0,c]$. Since each action we applied reduced perimeter or kept it the same, we can conclude that an interval of the form $[0,c]$ is isoperimetric.

\textbf{For the double bubble}: If we start with an arbitrary configuration for two regions with masses $M_1$ and $M_2$, we can first consolidate to the origin, resulting in at most two intervals on each side. We can then apply Proposition \ref{prop:OuterMassStealing} twice (first to the negative side of the origin, then to the positive side) to convert each region to a single interval. Doing so cannot make the perimeter larger, and the resulting configuration places each region entirely to one side of the origin. Thus, we can conclude that a pair of intervals of the form $[-d, 0]$ and $[0, e]$ form an isoperimetric double-bubble.
\end{proof}

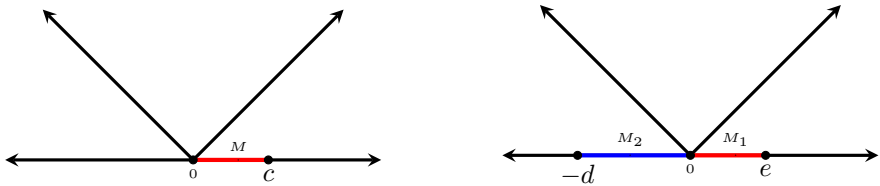
\begin{figure}[H]
     \centering
     \begin{subfigure}[h]{0.45\textwidth}
         \centering
         \begin{tikzpicture}
    %LINES:
    %x axis and |x|
    \draw[stealth-,very thick] (-2,2) -- (0,0);
    \draw[-stealth,very thick](0,0) -- (2,2);
    \draw[stealth-stealth,very thick] (-2.5,0) -- (2.5,0); 
    % green interval
    \filldraw[red,ultra thick] (0,0) -- (1,0);

    %POINTS: 
    %(neg)
    \filldraw (0,0) circle (1.5pt) node[anchor=north] {\tiny$0$};
    %(pos)
    \filldraw (1,0) circle (1.5pt) node[anchor=north] {$c$};
    
    %LABELS:
    \filldraw (.6,0) circle (.05pt) node[anchor=south] {\tiny$M$};

    \end{tikzpicture}
         %\caption{}
     \end{subfigure}
     \hfill
     \begin{subfigure}[h]{0.45\textwidth}
         \centering
         \begin{tikzpicture}
    %LINES:
    %x axis and |x|
    \draw[stealth-,very thick] (-2,2) -- (0,0);
    \draw[-stealth,very thick](0,0) -- (2,2);
    \draw[stealth-stealth,very thick] (-2.5,0) -- (2.5,0); 
    % green interval
    \filldraw[red,ultra thick] (0,0) -- (1,0);
    \filldraw[blue, ultra thick] (0,0) -- (-1.5,0);

    %POINTS: 
    %(neg)
    \filldraw (0,0) circle (1.5pt) node[anchor=north] {\tiny$0$};
    %(pos)
    \filldraw (1,0) circle (1.5pt) node[anchor=north] {$e$};
    \filldraw (-1.5,0) circle (1.5pt) node[anchor=north] {$-d$};
    
    %LABELS:
    \filldraw (.6,0) circle (.05pt) node[anchor=south] {\tiny$M_1$};
    \filldraw (-0.8,0) circle (.05pt) node[anchor=south] {\tiny$M_2$};
    \end{tikzpicture}
         %\caption{}
         \label{fig:medFinal}
     \end{subfigure}
      \caption{Solutions to the 1- and 2-bubble problem on $\mathbb{R}$ with density $|x|$}
\end{figure}

We conclude this section by remarking that, for configurations of $n$ regions with $n > 2$, we can still apply Proposition \ref{prop:OuterMassStealing} twice (first to the negative side, then to the positive side) to reduce the total number of intervals in our configuration. This gives us the following:

\begin{corollary}
An isoperimetric $n$-bubble is in condensed form, and has at most $2n-2$ adjacent intervals.
\end{corollary}

\section{Sliding masses and the first variation formula}

Up until now, we have reconfigured our regions by moving masses around in discrete steps and as complete intervals. In this section, we introduce the concept of continuously varying configurations via the first variation formula. We will formulate this for an arbitrary density function $f$. Consider an interval of $[a,x]$ with one fixed endpoint. We have already seen how the weighted mass $M$ and perimeter $P$ are measured as $\int_a^x f$ and $f(a) + f(x)$, respectively. An immediate consequence of this is that

\begin{align}
    \frac{dM}{dx} &= f(x)\\
    \frac{dP}{dx} &= f'(x).
\end{align}
Next, suppose that we let $x = x(t)$ vary as a function of time, and with a specified velocity $x'(t)$. This will allow us to calculate 

\begin{align}
    \frac{dM}{dt} &= f(x(t)) x'(t)\\
    \frac{dP}{dt} &= f'(x(t)) x'(t).
\end{align}
If we choose our speed to equal $1/f(x(t))$, these become

\begin{align}
    \frac{dM}{dt} &= 1\\
    \frac{dP}{dt} &= f'(x(t)) / f(x(t)) = \frac{d}{dx}\left[ \log(f(x(t))) \right].
\end{align}
We record this observation as follows:

\begin{proposition}[First Variation Formula]
\label{prop:FirstVariationFormula}
On $\mathbb{R}$ with density $f$: Suppose we take an interval $[a,b]$, and drag \emph{both} endpoints to the right at speed $1/f(x)$. Then the weighted mass of the interval will not change, while the perimeter's instantaneous rate of change will be $(\log(f))'(a) + (\log(f))'(b)$
\end{proposition}

\begin{corollary}
On $\mathbb{R}$ with density $f$: Suppose we have an isoperimetric $n$-bubble, comprised of a number of intervals with all endpoints of all intervals listed as $x_1, x_2, \dots, x_m$. Furthermore, assume $\log(f)$ is differentiable at each of the $x_i$. Then
\begin{equation}
\sum_{i=1}^m \log(f)'(x_i) = 0.
\end{equation}
\end{corollary}

\begin{proof} If our region is isoperimetric, it has minimal perimeter out of all other configurations that have regions with the same prescribed masses. This means, as we vary our configuration by dragging every endpoint to the right at speed $1/f$, our masses will stay the same and our perimeter must have been at a local/global minimum. Thus, its instantaneous rate of change should be 0.
\end{proof}

We will use the first variation formula to show that configurations of intervals that exhibit certain patterns cannot be isoperimertic. We begin by defining the patterns in question. %Next, we use ideas present in the first variation formula to argue that there cannot be ``alternating'' patterns in an isoperimetric region. To begin, we must clarify what we mean by ``alternating'' sets of intervals, as well as ``nested'' sets of intervals.

\begin{definition}
Suppose we have a condensed configuration of $n$ regions, and suppose that there are  two regions (call them $A$ and $B$) that are each composed of two intervals. Call these intervals $A^-$, $A^+$, $B^-$, and $B^+$  depending on their relative positions (so that $A^-$ and $B^-$ will sit to the left of $A^+$ and $B^+$, respectively). Suppose, up to renaming, that $A^-$ falls to the left of $B^-$. Then there are three possibilities for the relative ordering of these intervals:

    \begin{itemize}
        \item If their relative positions, ordered left to right, are $A^-$, $B^-$, $A^+$, $B^+$, then we say $A$, $B$ are \textbf{alternating}.
        \item If their relative ordering is $A^-$, $B^-$, $B^+$, $A^+$, then we say $A$, $B$ are \textbf{nested}.
        \item If their relative ordering is $A^-, A^+$, $B^-$, $B^+$, then we say that $A$, $B$ are \textbf{ordered}.
    \end{itemize}
\end{definition}

Our earlier work (specifically Corollary \ref{Cor:CondenseToTheOrigin}) guarantees that an ordered $A$ and $B$ will not occur in an isoperimetric configuration, since each side of the origin has at most one interval from each region. Our next lemma will guarantee that an isoperimetric region cannot have an alternating pattern either: if such a configuration exists, we can use the first variation formula to create a continuous movement that preserves the masses of our $n$ regions while reducing total perimeter. The core idea present -- simultaneous siphoning of area from each of the inner intervals in the alternating pattern -- was used in \cite{ChambersBongiovanniETAL18} to reduce the number of possible intervals in a 2-bubble.

\begin{proposition}[Simultaneous Mass Siphoning]
\label{Prop:SimultaneousSiphoning}
On $\mathbb{R}$ with density $|x|$: Suppose there exists a condensed configuration of $n$ regions. Furthermore, suppose there are two specific regions (identified as regions $A$ and $B$) that are of an alternating pattern (with ordering $A^-, B^-, A^+, B^+$). Then we can create a new configuration with lower perimeter by eliminating one of the interior intervals ($B^-$ or $A^+$) using a ``simultaneous siphoning'' process.
\end{proposition}

\begin{proof}
Take the right endpoint of $A^-$, the left endpoint of $B^-$, and all interval endpoints between these two. We drag these endpoints to the right at speed $1/|x|$. Simultaneously, we take the right endpoint of $A^+$, the left endpoint of $B^+$, and all interval endpoints between these two, and drag these endpoints to the left at speed $1/|x|$. This process is illustrated in Figure \ref{Figure6Alternating}. According to Proposition \ref{prop:FirstVariationFormula}, this process does not change the weighted masses of any region for which both endpoints are moving. Additionally, the masses for $A$ and $B$ are not changing: the mass siphoned into $A^-$ (and $B^+$) is identical to the mass siphoned away from $A^+$ (and $B^-$). Because our density function is $|x|$ and our endpoints are all moving towards the origin, this variation reduces total weighted perimeter. Since no moving endpoint is in danger of reaching the origin, this process can continue until either $B^-$ or $A^+$ has shrunk entirely to size 0. We see that perimeter was made to decrease throughout this process, and continued to decrease until one of the central intervals in the alternating pattern completely disappears.
\end{proof}

%%%%%%%%%%%%%%%%%%%%%%%%%%%%%%%%%%%%%%%%%%%%%%%%%%%%%
%
%  Figure 6: The Alternating Argument
%
%%%%%%%%%%%%%%%%%%%%%%%%%%%%%%%%%%%%%%%%%%%%%%%%%%%%%

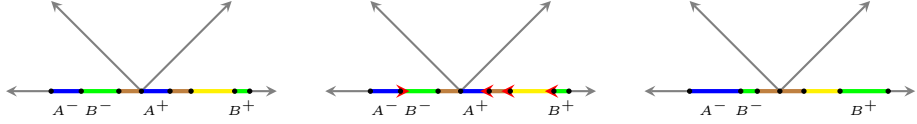
\begin{figure}[H]
    \centering
    \begin{subfigure}[h]{0.3\textwidth}    
    \begin{tikzpicture}[scale=0.4]
    %x axis and |x|:
    \draw[gray, stealth-,thick](-3,3) -- (0,0);
    \draw[gray, -stealth,thick](0,0) -- (3,3);
    \draw[gray, stealth-stealth,thick](-4.5,0) -- (4.5,0); 
    % green intervals: 2
    \filldraw[green,ultra thick] (-.75,0) -- (-2,0);
    \filldraw[green,ultra thick] (3.1,0) -- (3.6,0);

	% green labels: 2
	\node[anchor=north] at (-1.4,0) {\tiny $B^-$};
	\node[anchor=north] at (3.35,0) {\tiny $B^+$};
	
    % blue intervals: 2
    \filldraw[blue, ultra thick] (-3,0) -- (-2, 0);
    \filldraw[blue,ultra thick] (0,0) -- (.95,0);
    
	% blue labels: 2
	\node[anchor=north] at (-2.5,0) {\tiny $A^-$};
	\node[anchor=north] at (0.5,0) {\tiny $A^+$};
    
    % brown intervals: 2
    \filldraw[brown,ultra thick] (0,0) -- (-.75,0);
    \filldraw[brown,ultra thick] (.95,0) -- (1.65,0);
    % yellow intervals: 1
    \filldraw[yellow,ultra thick] (1.65,0) -- (3.1,0);
    
    %POINTS: 
    \filldraw (-.75,0) circle (2pt);
    \filldraw (-2,0) circle (2pt);
    \filldraw (-3,0) circle (2pt);
    \filldraw (0,0) circle (2pt);
    \filldraw (0.95,0) circle (2pt);
    \filldraw (1.65,0) circle (2pt);
    \filldraw (3.1,0) circle (2pt);
    \filldraw (3.6,0) circle (2pt);

    \end{tikzpicture}
%    \caption{Original picture, notice how the blue and green intervals are in the ``LRLR'' pattern.}

	\end{subfigure}
 	\hfill
    \begin{subfigure}[h]{0.3\textwidth}    
    \begin{tikzpicture}[scale=0.4]
    %x axis and |x|:
    \draw[gray, stealth-,thick](-3,3) -- (0,0);
    \draw[gray, -stealth,thick](0,0) -- (3,3);
    \draw[gray, stealth-stealth,thick](-4.5,0) -- (4.5,0); 
    % green intervals: 2
    \filldraw[green,ultra thick] (-.75,0) -- (-2,0);
    \filldraw[green,ultra thick] (3.1,0) -- (3.6,0);
    % blue intervals: 2
    \filldraw[blue, ultra thick] (-3,0) -- (-2, 0);
    \filldraw[blue,ultra thick] (0,0) -- (.95,0);
    % brown intervals: 2
    \filldraw[brown,ultra thick] (0,0) -- (-.75,0);
    \filldraw[brown,ultra thick] (.95,0) -- (1.65,0);
    % yellow intervals: 1
    \filldraw[yellow,ultra thick] (1.65,0) -- (3.1,0);

	% blue labels: 2
	\node[anchor=north] at (-2.5,0) {\tiny $A^-$};
	\node[anchor=north] at (0.5,0) {\tiny $A^+$};

	% green labels: 2
	\node[anchor=north] at (-1.4,0) {\tiny $B^-$};
	\node[anchor=north] at (3.35,0) {\tiny $B^+$};    
    
    %ARROWS:
    \draw[red, -stealth,very thick] (-2,0) -- (-1.7,0);
    \draw[red, stealth-,very thick] (2.8,0) -- (3.1,0);
    \draw[red, stealth-,very thick] (1.35, 0) -- (1.65,0);
    \draw[red, stealth-,very thick] (0.65, 0) -- (0.95,0);
    
    %POINTS: 
    \filldraw (-.75,0) circle (2pt);
    \filldraw (-2,0) circle (2pt);
    \filldraw (-3,0) circle (2pt);
    \filldraw (0,0) circle (2pt);
    \filldraw (0.95,0) circle (2pt);
    \filldraw (1.65,0) circle (2pt);
    \filldraw (3.1,0) circle (2pt);
    \filldraw (3.6,0) circle (2pt);

    \end{tikzpicture}
%    \caption{Dragging the endpoints between the  ``LR'' pairs on each side towards the origin. The negative side has the blue interval growing and the green interval shrinking; the positive side has the blue interval shrinking and the green interval growing. By the 1VF, the rates of growth are all the same, so the masses do not change size. Furthermore, perimeter is decreasing.}
	\end{subfigure}
 	\hfill
    \begin{subfigure}[h]{0.3\textwidth}
    \begin{tikzpicture}[scale=0.4]
    %x axis and |x|:
    \draw[gray, stealth-,thick](-3,3) -- (0,0);
    \draw[gray, -stealth,thick](0,0) -- (3,3);
    \draw[gray, stealth-stealth,thick](-4.5,0) -- (4.5,0); 
    % green intervals: 2
    \filldraw[green,ultra thick] (-.75,0) -- (-1.3,0);
    \filldraw[green,ultra thick] (2.0,0) -- (3.6,0);
    % blue intervals: 2
    \filldraw[blue, ultra thick] (-3,0) -- (-1.3, 0);
%    \filldraw[blue,ultra thick] (0,0) -- (.95,0);
    % brown intervals: 2
    \filldraw[brown,ultra thick] (0,0) -- (-.75,0);
    \filldraw[brown,ultra thick] (0,0) -- (0.8,0);
    % yellow intervals: 1
    \filldraw[yellow,ultra thick] (0.8,0) -- (2.0,0);

	% blue labels: 1
	\node[anchor=north] at (-2.2,0) {\tiny $A^-$};

	% green labels: 2
	\node[anchor=north] at (-1,0) {\tiny $B^-$};
	\node[anchor=north] at (2.8,0) {\tiny $B^+$};
    
    %POINTS: 
    \filldraw (-.75,0) circle (2pt);
    \filldraw (-1.3,0) circle (2pt);
    \filldraw (-3,0) circle (2pt);
    \filldraw (0,0) circle (2pt);
    \filldraw (0.8,0) circle (2pt);
    \filldraw (2,0) circle (2pt);
    %\filldraw (3.1,0) circle (2pt);
    \filldraw (3.6,0) circle (2pt);

    \end{tikzpicture}
    \end{subfigure}
    \caption{Over these three images, we see the interior intervals $B^-$ and $A^+$ shrink while the exterior intervals grow. This is done in a manner that reduces total perimeter and leaves total mass of $A$ and $B$ unchanged. In the final image, $A^+$ has completely vanished. At this point, the sliding is done and the alternating pattern has been eliminated.}
    \label{Figure6Alternating}
\end{figure}

Next, using a similar argument as above, we show that two intervals cannot be nested \emph{directly} around the origin in an isoperimetric region.

\begin{proposition}
\label{Prop:InnermostSliding}
On $\mathbb{R}$ with density $|x|$: Suppose there exists a condensed configuration of $n$ regions (not necessarily isoperimetric). Furthermore, suppose two specific regions (identified as regions $A$ and $B$) are nested such that the origin sits in the interior of region $B$, and region $B$ is immediately surrounded by region $A$. Then we can lower perimeter by sliding region $B$ to the right.
\end{proposition}

\begin{proof}
We will use the ideas present in the first variation formula here. Suppose our regions of interest are identified as $A^- = [-a, -b]$, $B = B^- \cup B^+ = [-b, 0] \cup [0, c]$, and $A^+ = [c,d]$. We focus on the perimeter contributed by $a$, $b$, $c$, and $d$, noting that there might be other interval endpoints in our configuration that correspond to regions other than $A$ and $B$. 

Slide the endpoints $-b$ and $c$ to the right with speed $1/|x|$. The weighted mass $A^-$ will grow at exactly the same rate for which $A^+$ shrinks, leaving the total mass of region A unchanged. Similarly, the weighted mass of region B remains unchanged. We continue this until one of two things happens:

    \begin{itemize}
        \item The $A^+$ region completely disappears, and the endpoint $c$ collides with the endpoint $d$. This will occur if, initially, $A^+$ has less weighted mass than $B^-$. The result is a new configuration with perimeter points of $-a$, $-b'$, $0$, and $d$ (note that $c$ has been eliminated since it collided with $d$). Since $b' < b$, this new configuration has less total perimeter.
        \item The $B^-$ interval completely disappears, and the endpoint $b$ collides with the origin. This will occur if the weighted mass of $A^+$ is initially larger than the weighted mass of $B^-$. The result is a new configuration with perimeter points of $-a$, $0$, $c'$, and $d$. Although the $c$ endpoint has grown, we can guarantee that $c'$ contributes less total perimeter than $c+b$ did, as this is equivalent to the applying the 1-bubble result of \cite{HuangMorgan19} to the region $B$. Thus, this new configuration has less total perimeter.
    \end{itemize}
Regardless, we see that total perimeter has been reduced while maintaining the same weighted masses in $A$ and $B$, completing the proof.
\end{proof}

Our next proposition also works to adjust the position of the origin. Unlike our previous proposition, this does not require a siphoning argument, but rather allows for all interval endpoints to shift simultaneously. Note that this argument is specific for our $|x|$ density function.

\begin{proposition}[The position of the the origin]
\label{prop:LocatingTheOriginAtAnEndpoint}Suppose we have $n$ regions in a condensed configuration on $\mathbb{R}$ with density $|x|$, and suppose that the origin lands in the interior of one interval. Then, by sliding all intervals and regions to either the left or the right and adjusting the endpoints appropriately, we can lower perimeter and position the origin at an endpoint between two intervals.  %(alternatively: that one region $A$ is split into two intervals, one region can be written as the union of two intervals, one non-positive and one non-negative. Then there exists a configuration $n$ regions with the same masses as before, but with lower perimeter than before, achieved by combining the region identified above into a single interval on one side of the origin.
\end{proposition}

\begin{proof}
Given a condensed configuration of regions, suppose the origin sits in the interior of a region $A$. By the nature of condensed configurations, this means $A$ can be split into two intervals directly adjacent to the origin, $A = A^- \cup A^+ = [-a_1, 0] \cup [0, a_2]$. Denote all other endpoints of other regions as $x_i$'s. %Assume both intervals are not adjacent to the origin (so it is not the case that $r_2 = 0 = r_3$); if this is the case, we will handle this in \textcolor{red}{the next proposition}. 
We proceed by taking all endpoints (except the origin) and sliding them to the right at speed $1/|x|$. According to the first variation formula for enclosed mass, this results in movement that keeps the weighted mass of each region constant (with $A$'s mass staying constant because $A^-$ loses mass at the same rate $A^+$ gains mass). By the first variation formula for perimeter, we know the total perimeter changes at rates of

\begin{align}
\frac{dP}{dt} &= \sum f'(x(t)) / f(x(t))\\
\frac{d^2P}{dt^2} &= \sum \frac{f(x(t)) f''(x(t)) - \left[f'(x(t))\right]^2}{\left[f(x(t))\right]^3}
\end{align}
where $f$ is the density function, and the  the summation is taken over all perimeter points that are moving. When $f(x) = |x|$, this simplifies to
\begin{align}
\frac{dP}{dt} &= \sum 1/x \\
\frac{d^2P}{dt^2} &= \sum  -1/|x|^3.
\end{align}
Notably, we can observe that the second derivative of perimeter is negative for $f(x) = |x|$. So the perimeter function $P$, when dragging every perimeter point to the right, is concave down. There are therefore two possibilities:
	\begin{enumerate}
	\item If perimeter is decreasing (so $P' < 0$), it will continue to decrease. We can therefore continue in this manner until the endpoint $a_1$ has been dragged to collide with the origin.
	\item If perimeter is increasing, we instead drag every one of our targeted points to the \emph{left} with speed $1/|x|$. This will keep the masses the same, will change the sign of $P'$, and will leave $P''$ negative. The result is that perimeter will decrease, and will continue to decrease until the endpoint $a_2$ collides with the origin.
	\end{enumerate}

Thus, we see that we can drag endpoints in a manner that decreases perimeter until either $A^-$ or $A^+$ disappears. Either way, we are left with the entirety of region $A$ on one side of the origin.
\end{proof}

At this point, we are ready to prove the triple bubble theorem.

\begin{theorem}[3-Bubble Theorem]
\label{Thm:TripleBubble}
On $\mathbb{R}$ with density $|x|$: The $n$-bubble solution to the isoperimetric problem for $n=3$ is a condensed configuration of three intervals. If the three regions have masses $M_1 \leq M_2 \leq M_3$, then the regions of mass $M_1$ and $M_2$ are adjacent to the origin, with $M_3$ positioned to be adjacent to $M_1$.
\end{theorem}

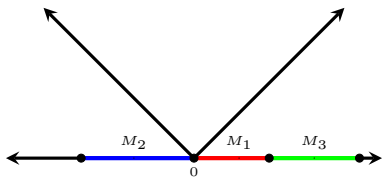
\begin{figure}[H]
     \centering
         \centering
         \begin{tikzpicture}
    %LINES:
    %x axis and |x|
    \draw[stealth-,very thick] (-2,2) -- (0,0);
    \draw[-stealth,very thick](0,0) -- (2,2);
    \draw[stealth-stealth,very thick] (-2.5,0) -- (2.5,0); 
    % green interval
    \filldraw[red,ultra thick] (0,0) -- (1,0);
    \filldraw[blue, ultra thick] (0,0) -- (-1.5,0);
    \filldraw[green, ultra thick] (1,0) -- (2.2, 0);

    %POINTS: 
    %(neg)
    \filldraw (0,0) circle (1.5pt) node[anchor=north] {\tiny$0$};
    %(pos)
    \filldraw (1,0) circle (1.5pt);% node[anchor=north] {$e$};
    \filldraw (-1.5,0) circle (1.5pt);% node[anchor=north] {$-d$};
    \filldraw (2.2,0) circle (1.5pt);
    
    %LABELS:
    \filldraw (.6,0) circle (.05pt) node[anchor=south] {\tiny$M_1$};
    \filldraw (-0.8,0) circle (.05pt) node[anchor=south] {\tiny$M_2$};
    \filldraw (1.6,0) circle (.05pt) node[anchor=south] {\tiny$M_3$};
    \end{tikzpicture}
      \caption{Solution to the triple-bubble problem with density $|x|$}
\end{figure}

\begin{proof}
By condensing our configuration and applying Proposition \ref{prop:OuterMassStealing} twice, we know our configuration has exactly three adjacent intervals. By applying Proposition \ref{prop:LocatingTheOriginAtAnEndpoint}, we know the middle interval cannot contain the origin in its interior. Therefore, we have exactly three intervals, with the origin acting as an endpoint shared by two of them. A direct computation tells us the correct ordering is with $M_1$ and $M_3$ on the same side of the origin, with $M_2$ on the opposite side. Proposition \ref{prop:TranspositionLemma} confirms the ordering of $M_1$ and $M_3$. This completes the proof.
\end{proof}

\section{Proof of the 4-bubble theorem}

In this section, we identify the isoperimetric configuration of 4 distinct regions on the real number line with density $|x|$. From our earlier work, we know that that such a configuration is condensed into  at most 6 intervals (first by condensing into 8 intervals, and then by ``mass stealing'' to condense the outer interval on each side of the origin). We will reduce further in the final proof, but we begin by identifying a few results that will aid us down the stretch. We start with a ``4-bubble lemma'' which tells us isoperimetric 4-bubbles with exactly 4 intervals must have two intervals on each side of the origin.

\begin{theorem}[4-Bubble Lemma]
On $\mathbb{R}$ with density $|x|$: If a configuration of 4 regions is isoperimetric, and if each region consists of a single interval, then there must be exactly two intervals on each side of the origin.
\end{theorem}

\begin{proof}
Using Proposition \ref{prop:LocatingTheOriginAtAnEndpoint}, we know that the origin will not appear in the interior of one of the intervals. Thus, we simply need to eliminate the possibility that more than two intervals appear on one side of the origin. If 4 intervals are on one side of the origin, we can simply move the outermost interval to the other side of the origin. This will clearly result in a configuration with less perimeter. We are left exploring the configuration with 3 intervals on one side of the origin, and 1 interval on the other. In this case, we identify the interval of smallest mass as $S$.

\textbf{If $S$ is on the side with 1 interval}, then it is there alone. Following along with Figure \ref{fig:4bubblefringecase0}, we move the outermost interval across to the other side of the origin. Our original figure had total perimeter $a + b + c + d$, compared to our new perimeter of $b + c + d + e'$. Because $-a$ was further from the origin than $e'$, this new configuration will have less total perimeter.

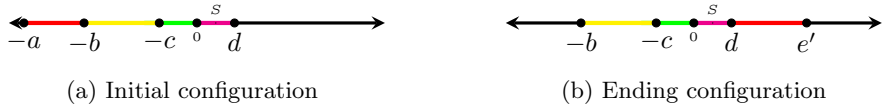
\begin{figure}[H]
     \centering
     \begin{subfigure}[h]{0.45\textwidth}
         \centering
         \begin{tikzpicture}
    %LINES:
    %x axis and |x|
    %\draw[stealth-,very thick] (-2,2) -- (0,0);
    %\draw[-stealth,very thick](0,0) -- (2,2);
    \draw[stealth-stealth,very thick] (-2.5,0) -- (2.5,0); 
    % green interval
    \filldraw[magenta,ultra thick] (0,0) -- (0.5,0);
    % pink intervals
    \filldraw[green,ultra thick] (0,0) -- (-0.7,0);
    % yellow interval
    \filldraw[yellow,ultra thick] (-.5,0) -- (-1.5,0);
    % red interval
    \filldraw[red,ultra thick] (-1.5,0) -- (-2.3,0);
    %POINTS: 
    %(neg)
    \filldraw (-2.3,0) circle (1.5pt) node[anchor=north] {$-a$};
    \filldraw (-1.5,0) circle (1.5pt) node[anchor=north] {$-b$};
    \filldraw (-0.5,0) circle (1.5pt) node[anchor=north] {$-c$};
    %(origin)
    \filldraw (0,0) circle (1.5pt) node[anchor=north] {\tiny$0$};
    %(pos)
    \filldraw (0.5,0) circle (1.5pt) node[anchor=north] {$d$};
    %\filldraw (1.9,0) circle (1.5pt) node[anchor=north] {$e'$};
    %LABELS:
    %\filldraw (-1.5,0) circle (.05pt) node[anchor=south] {\tiny$R_2$};
    \filldraw (.25,0) circle (.05pt) node[anchor=south] {\tiny$S$};
    %\filldraw (.35,0) circle (.05pt) node[anchor=south] {\tiny$R_1$};
    %\filldraw (1.2,0) circle (.05pt) node[anchor=south] {\tiny$R_3$};

    \end{tikzpicture}
         \caption{Initial configuration}
         \label{fig:4bubblefringecase0.1}
     \end{subfigure}
     \hfill
     \begin{subfigure}[h]{0.45\textwidth}
         \centering
         \begin{tikzpicture}
    %LINES:
    %x axis and |x|
    %\draw[stealth-,very thick] (-2,2) -- (0,0);
    %\draw[-stealth,very thick](0,0) -- (2,2);
    \draw[stealth-stealth,very thick] (-2.5,0) -- (2.5,0); 
    % green interval
    \filldraw[magenta,ultra thick] (0,0) -- (0.5,0);
    % pink intervals
    \filldraw[green,ultra thick] (0,0) -- (-0.7,0);
    % yellow interval
    \filldraw[yellow,ultra thick] (-.5,0) -- (-1.5,0);
    % red interval
    \filldraw[red,ultra thick] (0.5,0) -- (1.5,0);
    %POINTS: 
    %(neg)
    %\filldraw (-2.3,0) circle (1.5pt) node[anchor=north] {$-a$};
    \filldraw (-1.5,0) circle (1.5pt) node[anchor=north] {$-b$};
    \filldraw (-0.5,0) circle (1.5pt) node[anchor=north] {$-c$};
    %(origin)
    \filldraw (0,0) circle (1.5pt) node[anchor=north] {\tiny$0$};
    %(pos)
    \filldraw (0.5,0) circle (1.5pt) node[anchor=north] {$d$};
    \filldraw (1.5,0) circle (1.5pt) node[anchor=north] {$e'$};
    %LABELS:
    %\filldraw (-1.5,0) circle (.05pt) node[anchor=south] {\tiny$R_2$};
    \filldraw (.25,0) circle (.05pt) node[anchor=south] {\tiny$S$};
    %\filldraw (.35,0) circle (.05pt) node[anchor=south] {\tiny$R_1$};
    %\filldraw (1.2,0) circle (.05pt) node[anchor=south] {\tiny$R_3$};
    \end{tikzpicture}
         \caption{Ending configuration}
         \label{fig:4bubblefringecase0.2}
     \end{subfigure}
      \caption{Case 1: Moving the outermost interval over.}
        \label{fig:4bubblefringecase0}
\end{figure}

\textbf{If $S$ is on the side with 3 intervals}, we must do a bit more work. We name the lone interval on the other side of the origin $L$, and name the remaining two intervals $X_1$ and $X_2$ as in Figure \ref{fig:4bubblefringecase2}.

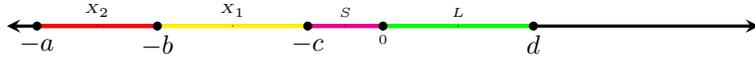
\begin{figure}[H]
\centering
        \begin{tikzpicture}
    %LINES:
    %x axis and |x|
    %\draw[stealth-,very thick] (-2,2) -- (0,0);
    %\draw[-stealth,very thick](0,0) -- (2,2);
    \draw[stealth-stealth,very thick] (-5,0) -- (5,0); 
    % green interval
    \filldraw[green,ultra thick] (0,0) -- (2,0);
    % pink intervals
    \filldraw[magenta,ultra thick] (0,0) -- (-1,0);
    % yellow interval
    \filldraw[yellow,ultra thick] (-1,0) -- (-3,0);
    % red interval
    \filldraw[red,ultra thick] (-3,0) -- (-4.6,0);
    %POINTS: 
    %(neg)
    \filldraw (-4.6,0) circle (1.5pt) node[anchor=north] {$-a$};
    \filldraw (-3,0) circle (1.5pt) node[anchor=north] {$-b$};
    \filldraw (-1,0) circle (1.5pt) node[anchor=north] {$-c$};
    %(origin)
    \filldraw (0,0) circle (1.5pt) node[anchor=north] {\tiny$0$};
    %(pos)
    \filldraw (2,0) circle (1.5pt) node[anchor=north] {$d$};
    %\filldraw (1.9,0) circle (1.5pt) node[anchor=north] {$e'$};
    %LABELS:
    \filldraw (-2,0) circle (.05pt) node[anchor=south] {\tiny$X_1$};
    \filldraw (-.5,0) circle (.05pt) node[anchor=south] {\tiny$S$};
    \filldraw (-3.8,0) circle (.05pt) node[anchor=south] {\tiny$X_2$};
    \filldraw (1,0) circle (.05pt) node[anchor=south] {\tiny$L$};

    \end{tikzpicture}
         \caption{Our naming convention for the rest of this proof.}
         \label{fig:4bubblefringecase2}
         
\end{figure}
Proposition \ref{prop:TranspositionLemma} guarantees that our masses will be ordered as $S \leq X_1 \leq X_2$. However, we do not know the relative size of $L$. We have two cases to consider:

If $X_1 \leq L$, then we will want to position our smallest region $S$ across the origin. This will result in endpoints $-a$ and $-b$ moving closer to the origin, with endpoint $d$ moving further away from the origin. The initial and final configurations are shown in Figure \ref{fig:4bubblefringecase1}. Our original perimeter was $a + b + c + d$, while our new perimeter is $a' + b' + c' + d'$. By assumption, $X_1 \leq L$, and so $d \geq b'$. This will mean that $[d, d']$ is narrower than (or the same length as) $[-b, -b']$ by Proposition \ref{prop:LengthEndpointInequality}. We can conclude that $d' - d \leq b - b'$, and therefore that $d' + b' \leq b + d$. Since $a' < a$, we get that our new configuration has less perimeter than our original one.

\begin{figure}[H]
     \centering
     \begin{subfigure}[h]{0.45\textwidth}
         \centering
         \begin{tikzpicture}
    %LINES:
    %x axis and |x|
    %\draw[stealth-,very thick] (-2,2) -- (0,0);
    %\draw[-stealth,very thick](0,0) -- (2,2);
    \draw[stealth-stealth,very thick] (-2.5,0) -- (2.5,0); 
    % green interval
    \filldraw[green,ultra thick] (0,0) -- (1,0);
    % pink intervals
    \filldraw[magenta,ultra thick] (0,0) -- (-0.5,0);
    % yellow interval
    \filldraw[yellow,ultra thick] (-.5,0) -- (-1.5,0);
    % red interval
    \filldraw[red,ultra thick] (-1.5,0) -- (-2.3,0);
    %POINTS: 
    %(neg)
    \filldraw (-2.3,0) circle (1.5pt) node[anchor=north] {$-a$};
    \filldraw (-1.5,0) circle (1.5pt) node[anchor=north] {$-b$};
    \filldraw (-0.5,0) circle (1.5pt) node[anchor=north] {$-c$};
    %(origin)
    \filldraw (0,0) circle (1.5pt) node[anchor=north] {\tiny$0$};
    %(pos)
    \filldraw (1,0) circle (1.5pt) node[anchor=north] {$d$};
    %\filldraw (1.9,0) circle (1.5pt) node[anchor=north] {$e'$};
    %LABELS:
    %\filldraw (-1.5,0) circle (.05pt) node[anchor=south] {\tiny$R_2$};
    \filldraw (-.25,0) circle (.05pt) node[anchor=south] {\tiny$S$};
    %\filldraw (.35,0) circle (.05pt) node[anchor=south] {\tiny$R_1$};
    %\filldraw (1.2,0) circle (.05pt) node[anchor=south] {\tiny$R_3$};

    \end{tikzpicture}
         \caption{Initial configuration}
         \label{fig:4bubblefringecase1.1}
     \end{subfigure}
     \hfill
     \begin{subfigure}[h]{0.45\textwidth}
         \centering
         \begin{tikzpicture}
   %LINES:
    %x axis and |x|
    %\draw[stealth-,very thick] (-2,2) -- (0,0);
    %\draw[-stealth,very thick](0,0) -- (2,2);
    \draw[stealth-stealth,very thick] (-2.5,0) -- (2.5,0); 
    % green interval
    \filldraw[green,ultra thick] (0,0) -- (1,0);
    % pink intervals
    \filldraw[magenta,ultra thick] (0,0) -- (-0.5,0);
    % yellow interval
    \filldraw[yellow,ultra thick] (-.5,0) -- (-1.5,0);
    % red interval
    \filldraw[red,ultra thick] (-1.5,0) -- (-2.3,0);
    %POINTS: 
    %(neg)
    \filldraw (-2.3,0) circle (1.5pt) node[anchor=north] {$-a'$};
    \filldraw (-1.5,0) circle (1.5pt) node[anchor=north] {$-b'$};
    \filldraw (-0.0,0) circle (1.5pt) node[anchor=north] {$c$};
    %(origin)
    \filldraw (-.5,0) circle (1.5pt) node[anchor=north] {\tiny$0$};
    %(pos)
    \filldraw (1,0) circle (1.5pt) node[anchor=north] {$d'$};
    %\filldraw (1.9,0) circle (1.5pt) node[anchor=north] {$e'$};
    %LABELS:
    %\filldraw (-1.5,0) circle (.05pt) node[anchor=south] {\tiny$R_2$};
    \filldraw (-.25,0) circle (.05pt) node[anchor=south] {\tiny$S$};
    %\filldraw (.35,0) circle (.05pt) node[anchor=south] {\tiny$R_1$};
    %\filldraw (1.2,0) circle (.05pt) node[anchor=south] {\tiny$R_3$};

    \end{tikzpicture}
         \caption{Ending configuration}
         \label{fig:4bubblefringecase1.2}
     \end{subfigure}
      \caption{Case 2: Moving the smallest region across the origin.}
        \label{fig:4bubblefringecase1}
\end{figure}
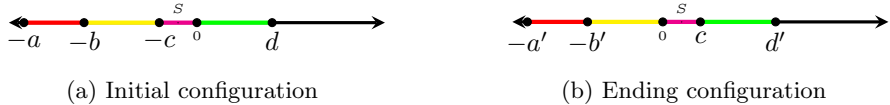

Finally, if $L < X_1$, We can move region $X_2$ over to the other side of the origin. We see this in Figure \ref{fig:4bubblefringecase3}. Our old perimeter was $a + b + c + d$, which we must compare to $b + c + d + e'$. However, since $L < x_1$, it is clear that $-b$ is further from the origin than $d$ is, and therefore we get that $-a$ is further from the origin than $e'$. Thus, this new configuration will have less total perimeter.

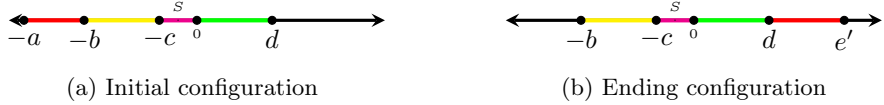
\begin{figure}[H]
     \centering
     \begin{subfigure}[h]{0.45\textwidth}
         \centering
         \begin{tikzpicture}
    %LINES:
    %x axis and |x|
    %\draw[stealth-,very thick] (-2,2) -- (0,0);
    %\draw[-stealth,very thick](0,0) -- (2,2);
    \draw[stealth-stealth,very thick] (-2.5,0) -- (2.5,0); 
    % green interval
    \filldraw[green,ultra thick] (0,0) -- (1,0);
    % pink intervals
    \filldraw[magenta,ultra thick] (0,0) -- (-0.5,0);
    % yellow interval
    \filldraw[yellow,ultra thick] (-.5,0) -- (-1.5,0);
    % red interval
    \filldraw[red,ultra thick] (-1.5,0) -- (-2.3,0);
    %POINTS: 
    %(neg)
    \filldraw (-2.3,0) circle (1.5pt) node[anchor=north] {$-a$};
    \filldraw (-1.5,0) circle (1.5pt) node[anchor=north] {$-b$};
    \filldraw (-0.5,0) circle (1.5pt) node[anchor=north] {$-c$};
    %(origin)
    \filldraw (0,0) circle (1.5pt) node[anchor=north] {\tiny$0$};
    %(pos)
    \filldraw (1,0) circle (1.5pt) node[anchor=north] {$d$};
    %\filldraw (1.9,0) circle (1.5pt) node[anchor=north] {$e'$};
    %LABELS:
    %\filldraw (-1.5,0) circle (.05pt) node[anchor=south] {\tiny$R_2$};
    \filldraw (-.25,0) circle (.05pt) node[anchor=south] {\tiny$S$};
    %\filldraw (.35,0) circle (.05pt) node[anchor=south] {\tiny$R_1$};
    %\filldraw (1.2,0) circle (.05pt) node[anchor=south] {\tiny$R_3$};

    \end{tikzpicture}
         \caption{Initial configuration}
         \label{fig:4bubblefringecase3.1}
     \end{subfigure}
     \hfill
     \begin{subfigure}[h]{0.45\textwidth}
         \centering
         \begin{tikzpicture}
     %LINES:
    %x axis and |x|
    %\draw[stealth-,very thick] (-2,2) -- (0,0);
    %\draw[-stealth,very thick](0,0) -- (2,2);
    \draw[stealth-stealth,very thick] (-2.5,0) -- (2.5,0); 
    % green interval
    \filldraw[green,ultra thick] (0,0) -- (1,0);
    % pink intervals
    \filldraw[magenta,ultra thick] (0,0) -- (-0.5,0);
    % yellow interval
    \filldraw[yellow,ultra thick] (-.5,0) -- (-1.5,0);
    % red interval
    \filldraw[red,ultra thick] (1,0) -- (2,0);
    %POINTS: 
    %(neg)
    %\filldraw (-2.3,0) circle (1.5pt) node[anchor=north] {$-a$};
    \filldraw (-1.5,0) circle (1.5pt) node[anchor=north] {$-b$};
    \filldraw (-0.5,0) circle (1.5pt) node[anchor=north] {$-c$};
    %(origin)
    \filldraw (0,0) circle (1.5pt) node[anchor=north] {\tiny$0$};
    %(pos)
    \filldraw (1,0) circle (1.5pt) node[anchor=north] {$d$};
    \filldraw (2,0) circle (1.5pt) node[anchor=north] {$e'$};
    %LABELS:
    %\filldraw (-1.5,0) circle (.05pt) node[anchor=south] {\tiny$R_2$};
    \filldraw (-.25,0) circle (.05pt) node[anchor=south] {\tiny$S$};
    %\filldraw (.35,0) circle (.05pt) node[anchor=south] {\tiny$R_1$};
    %\filldraw (1.2,0) circle (.05pt) node[anchor=south] {\tiny$R_3$};

    \end{tikzpicture}
         \caption{Ending configuration}
         \label{fig:4bubblefringecase3.2}
     \end{subfigure}
      \caption{Case 3: Moving the largest region over.}
        \label{fig:4bubblefringecase3}
\end{figure}
This exhausts all of our possible cases. We conclude that, if we know our 4-bubble has exactly four intervals, then it must have two intervals on each side of the origin.
\end{proof}
Next we state several results that display the appropriate way to order intervals of different masses if we know, a priori, that we have a fixed number of intervals on each side of the origin.

\begin{proposition}[Alternating 4-interval framework]
On $\mathbb{R}$ with density $|x|$: Suppose we have four regions, condensed so that each region consists of a single interval, and organized so that there are two intervals on each side of the origin. Then the least-perimeter way to meet these specifications is to have intervals alternate back and forth across the origin as they increase from smallest mass to largest. This means that the masses can be ordered, up to reflection, as $R_3 R_1 . R_2 R_4$ (where the masses of region $R_i$ satisfy $M_1 \leq M_2 \leq M_3 \leq M_4$). 
\end{proposition}

\begin{proposition}[Alternating 5-interval framework]
On $\mathbb{R}$ with density $|x|$: Suppose we have five regions, condensed so that each region consists of a single interval, and organized so that one side of the origin has two intervals while the other side has three. Then the least-perimeter way to meet these specifications is to order the masses, up to reflection, as $ R_5 R_3 R_1 . R_2 R_4$ (where the masses of region $R_i$ satisfy $M_1 \leq \dots \leq M_5$).
\end{proposition}

\begin{proposition}[Alternating 6-interval framework]
On $\mathbb{R}$ with density $|x|$: Suppose we have six regions, condensed so that each region consists of a single interval, and organized so that there are three intervals on each side of the origin. Then the least-perimeter way to meet these specifications is to have intervals alternate back and forth across the origin as they increase from smallest mass to largest. This means that the masses can be ordered, up to reflection, as $ R_5 R_3 R_1 . R_2 R_4 R_6$ (where the masses of region $R_i$ satisfy $M_1 \leq \dots \leq M_6$). 
\end{proposition}
The proofs of these propositions are left for the appendix. However, with these propositions, we are now able to prove the 4-bubble theorem.

\begin{theorem}[4-Bubble Theorem]
\label{Thm:FourBubble}
On $\mathbb{R}$ with density $|x|$: The $n$-bubble solution to the isoperimetric problem for $n=4$ is a condensed configuration of four regions, with each region consisting of a single interval, and with two intervals on each side of the origin. If the four regions have masses $M_1 \leq M_2 \leq M_3 \leq M_4$, then the regions of mass $M_1$ and $M_2$ are adjacent to the origin; $M_3$ is positioned to be adjacent to $M_1$; and $M_4$ is positioned to be adjacent to $M_2$.
\end{theorem}

\begin{figure}[H]
     \centering
         \centering
         \begin{tikzpicture}
    %LINES:
    %x axis and |x|
    \draw[stealth-,very thick] (-2,2) -- (0,0);
    \draw[-stealth,very thick](0,0) -- (2,2);
    \draw[stealth-stealth,very thick] (-3.75,0) -- (3.75,0); 
    % green interval
    \filldraw[red,ultra thick] (0,0) -- (1,0);
    \filldraw[blue, ultra thick] (0,0) -- (-1.5,0);
    \filldraw[green, ultra thick] (1,0) -- (2.2, 0);
    \filldraw[yellow, ultra thick] (-3.4,0) -- (-1.5,0);

    %POINTS: 
    %(neg)
    \filldraw (0,0) circle (1.5pt) node[anchor=north] {\tiny$0$};
    %(pos)
    \filldraw (1,0) circle (1.5pt);% node[anchor=north] {$e$};
    \filldraw (-1.5,0) circle (1.5pt);% node[anchor=north] {$-d$};
    \filldraw (2.2,0) circle (1.5pt);
    \filldraw (-3.4,0) circle (1.5pt);
    
    %LABELS:
    \filldraw (.6,0) circle (.05pt) node[anchor=south] {\tiny$M_1$};
    \filldraw (-0.8,0) circle (.05pt) node[anchor=south] {\tiny$M_2$};
    \filldraw (1.6,0) circle (.05pt) node[anchor=south] {\tiny$M_3$};
    \filldraw (-2.5, 0) circle (.05pt) node[anchor=south] {\tiny $M_4$};
    \end{tikzpicture}
      \caption{Solution to the 4-bubble problem with density $|x|$}
\end{figure}
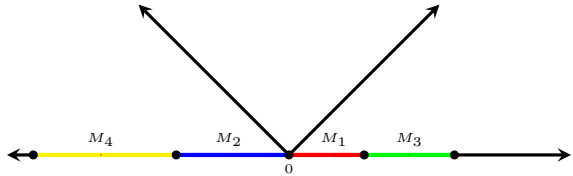

%\begin{tcolorbox}
%\textbf{Quadruple Bubble Argument:}  an isoperimetric configuration of four distinct regions with four masses is an alternating scheme from smallest to largest - i.e. without loss of generality, on the left side we have the largest and medium bubbles, and on the right we have the small and large bubble \textcolor{red}{INCLUDE SOME ``STANDARD'' language here}.
%\end{tcolorbox}

\begin{proof}
Based on Corollary \ref{Cor:CondenseToTheOrigin}, we know an isoperimetric 4-bubble can be made up of at most 8 intervals. Applying Proposition \ref{prop:OuterMassStealing} twice, we can reduce this maximum number to 6. Additionally, Proposition \ref{prop:OuterMassStealing} guarantees that the outermost interval on each side of the origin is the \emph{only} interval of its respective region. Thus, if we do have 6 total intervals, two of these regions are each made up of two intervals and are positioned on the interior of the configuration. 

With 6 intervals, there are two possible scenarios: the innermost intervals are either alternating, or they are nested directly adjacent to the origin. Based on how these inner intervals are configured, we can apply either Proposition \ref{Prop:SimultaneousSiphoning} or \ref{Prop:InnermostSliding} to eliminate one more interval. Finally, we can slide the entire picture (if necessary) using \ref{prop:LocatingTheOriginAtAnEndpoint} to give us a configuration with at most five distinct intervals, and with the origin between two of these intervals in the interior. 

Thus, an isoperimetric configuration has either 4 or 5 intervals in total. If it has 4 intervals, every region is made up of a single interval and Proposition \ref{Prop:4IntervalFramework} will prove the theorem for us. Thus, we only need to worry about a scenario in which we have five intervals. In such a case, we know that a single region (which we identify as $A = A^- \cup A^+$) is split into two intervals, one on each side of the origin. The other regions we will label as $R_i$, with mass $M_i$. Throughout this proof, we will abuse notation by conflating regions $R_i$ with their mass size $M_i$, and will use the understanding that $M_1 \leq M_2 \leq M_3$.  Because of the Proposition \ref{prop:OuterMassStealing}, we know neither $A^+$ nor $A^-$ is part of an outermost interval on either side. At the same time, Proposition \ref{prop:LocatingTheOriginAtAnEndpoint} allows us to conclude that $A^+$ and $A^-$ cannot both be directly adjacent to the origin. Therefore, our configuration (up to reflection) must look like Figure \ref{fig:initPositions}. Due to Proposition \ref{Prop:5IntervalFramework} in the appendix, we know that the relative sizes of these masses is $R_1 \leq A^- \leq A^+ \leq R_2 \leq R_3$.

%In a scenario with a full six intervals, this means that two regions consist of one interval (and these the outermost

%By Lemma 2.2, Mass Stealing Lemma, and the LRLR Proposition, we know that we can consolidate the four bubbles from an arbitrary number of intervals to the origin such that we have at most five intervals; this means that only one bubble is represented in two separate intervals. Without loss of generality this will be the Pink bubble in Figure \ref{fig:initPositions} below where one Pink interval has mass $A^-$ and the other has mass $A^+$. We will refer to the outermost intervals (Red and Yellow) as having masses $R_2$ and $R_3$, respectively. Let's assume that the least perimeter arrangement orders these intervals from smallest to largest mass, alternating outward from the origin. Let's also treat the origin as our smallest interval. Then we have that Green is our medium-sized, or mass $R_1$, interval and so $M \leq L_1 \leq L_2 \leq X_1 \leq X_2$ is true. \bigskip

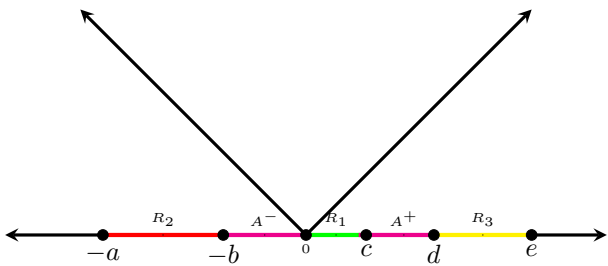
\begin{figure}[H]
    \centering
    \begin{tikzpicture}
    %LINES:
    %x axis and |x|
    \draw[stealth-,very thick] (-3,3) -- (0,0);
    \draw[-stealth,very thick](0,0) -- (3,3);
    \draw[stealth-stealth,very thick] (-4,0) -- (4,0); 
    
    % green interval
    \filldraw[green,ultra thick] (0,0) -- (.8,0);
    
    % pink intervals
    \filldraw[magenta,ultra thick] (0,0) -- (-1.1,0);
    \filldraw[magenta,ultra thick] (.7,0) -- (1.7,0);
    % yellow interval
    \filldraw[yellow,ultra thick] (1.7,0) -- (3,0);
    % red interval
    \filldraw[red,ultra thick] (-1.1,0) -- (-2.7,0);

    %POINTS: 
    %(neg)
    \filldraw (-2.7,0) circle (2pt) node[anchor=north] {$-a$};
    \filldraw (-1.1,0) circle (2pt) node[anchor=north] {$-b$};
    %(origin)
    \filldraw (0,0) circle (2pt) node[anchor=north] {\tiny$0$};
    %(pos)
    \filldraw (.8,0) circle (2pt) node[anchor=north] {$c$};
    \filldraw (1.7,0) circle (2pt) node[anchor=north] {$d$};
    \filldraw (3,0) circle (2pt) node[anchor=north] {$e$};
    
    %LABELS:
    \filldraw (-1.9,0) circle (.05pt) node[anchor=south] {\tiny$R_2$};%formerly X_1
    \filldraw (-.55,0) circle (.05pt) node[anchor=south] {\tiny$A^-$};%formerly L_1
    \filldraw (.4,0) circle (.05pt) node[anchor=south] {\tiny$R_1$};%formerly M
    \filldraw (1.3,0) circle (.05pt) node[anchor=south] {\tiny$A^+$};%formerly L_2
    \filldraw (2.35,0) circle (.05pt) node[anchor=south] {\tiny$R_3$};%formerly X_2

    \end{tikzpicture}
    \caption{Initial arbitrary $4$ bubble configuration of $5$ intervals}
    \label{fig:initPositions}
\end{figure}

To show this 5-interval configuration is not isoperimetric, we will want to show that we can reduce perimeter by creating a consolidated configuration (with $A^-$ and $A^+$ joined into a single interval which we identify as $A$). There are four possible scenarios we have to consider, depending on how the size of $A$ fits into the relative sizes of $R_1$, $R_2$, and $R_3$.

%SAY WHY EVENTUALLY%
%So, we will prove that in all three cases, alternating the order from smallest to largest is isoperimetric. \bigskip

\textbf{Case 1: $A \leq R_1$.} This is impossible, as it would mean $A^+$ and $A^-$ were each individually smaller than $R_1$, and therefore $R_1$ would not be directly adjacent to the origin.

\textbf{Case 2: $R_1 < A \leq R_2$.}  First, we will consolidate $A$ by bringing $A^+$ to the left side of the origin and adjoining it to $A^-$ . This intermediate step gives the new arrangement (a) with new endpoints $-a'$, $-b'$, and $e'$. The next step (b) is to switch the position $R_2$ and $R_3$. This gives us our final position with new endpoints $-a''$ and $e''$ and, therefore, the new total perimeter $a'' + b' + c + e''$. We will show this new perimeter is less than our starting perimeter.

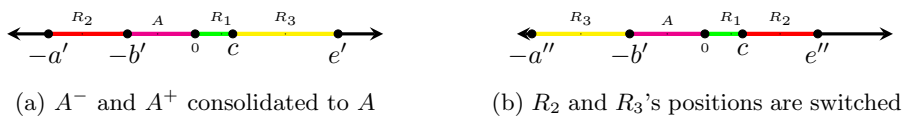
\begin{figure}[H]
     \centering
     \begin{subfigure}[h]{0.45\textwidth}
         \centering
         \begin{tikzpicture}
    %LINES:
    %x axis and |x|
    %\draw[stealth-,very thick] (-2,2) -- (0,0);
    %\draw[-stealth,very thick](0,0) -- (2,2);
    \draw[stealth-stealth,very thick] (-2.5,0) -- (2.5,0); 
    % green interval
    \filldraw[green,ultra thick] (0,0) -- (.5,0);
    % pink intervals
    \filldraw[magenta,ultra thick] (0,0) -- (-1,0);
    % yellow interval
    \filldraw[yellow,ultra thick] (.5,0) -- (1.9,0);
    % red interval
    \filldraw[red,ultra thick] (-.9,0) -- (-1.95,0);
    %POINTS: 
    %(neg)
    \filldraw (-1.95,0) circle (1.5pt) node[anchor=north] {$-a'$};
    \filldraw (-.9,0) circle (1.5pt) node[anchor=north] {$-b'$};
    %(origin)
    \filldraw (0,0) circle (1.5pt) node[anchor=north] {\tiny$0$};
    %(pos)
    \filldraw (.5,0) circle (1.5pt) node[anchor=north] {$c$};
    \filldraw (1.9,0) circle (1.5pt) node[anchor=north] {$e'$};
    %LABELS:
    \filldraw (-1.5,0) circle (.05pt) node[anchor=south] {\tiny$R_2$};
    \filldraw (-.5,0) circle (.05pt) node[anchor=south] {\tiny$A$};
    \filldraw (.35,0) circle (.05pt) node[anchor=south] {\tiny$R_1$};
    \filldraw (1.2,0) circle (.05pt) node[anchor=south] {\tiny$R_3$};

    \end{tikzpicture}
         \caption{$A^-$ and $A^+$ consolidated to $A$}
         \label{fig:medConsolidated}
     \end{subfigure}
     \hfill
     \begin{subfigure}[h]{0.45\textwidth}
         \centering
         \begin{tikzpicture}
     %LINES:
    %x axis and |x|
    %\draw[stealth-,very thick] (-2,2) -- (0,0);
    %\draw[-stealth,very thick](0,0) -- (2,2);
    \draw[stealth-stealth,very thick] (-2.5,0) -- (2.5,0); 
    % green interval
    \filldraw[green,ultra thick] (0,0) -- (.5,0);
    % pink intervals
    \filldraw[magenta,ultra thick] (0,0) -- (-1,0);
    % yellow interval
    \filldraw[yellow,ultra thick] (-1,0) -- (-2.3,0);
    % red interval
    \filldraw[red,ultra thick] (.5,0) -- (1.5,0);
    %POINTS: 
    %(neg)
    \filldraw (-2.3,0) circle (1.5pt) node[anchor=north] {$-a''$};
    \filldraw (-1,0) circle (1.5pt) node[anchor=north] {$-b'$};
    %(origin)
    \filldraw (0,0) circle (1.5pt) node[anchor=north] {\tiny$0$};
    %(pos)
    \filldraw (.5,0) circle (1.5pt) node[anchor=north] {$c$};
    \filldraw (1.5,0) circle (1.5pt) node[anchor=north] {$e''$};
    %LABELS:
    \filldraw (-1.65,0) circle (.05pt) node[anchor=south] {\tiny$R_3$};
    \filldraw (-.5,0) circle (.05pt) node[anchor=south] {\tiny$A$};
    \filldraw (.35,0) circle (.05pt) node[anchor=south] {\tiny$R_1$};
    \filldraw (1,0) circle (.05pt) node[anchor=south] {\tiny$R_2$};
    \end{tikzpicture}
         \caption{$R_2$ and $R_3$'s positions are switched}
         \label{fig:medFinal}
     \end{subfigure}
      \caption{Case 2: Initial rearrangements}
        \label{fig:initSteps1}
\end{figure}

We keep track of the changes that have been made with Figure \ref{fig:medHistory}, where we introduce $\updelta_x=R_3-R_2$ and $\updelta_{x_1}=R_2-A^+$. The appropriate masses have been identified in the figure.

\begin{figure}[H]
    \centering
    \begin{tikzpicture}
    %LINES:
    %x axis and |x|:
    \draw[stealth-,very thick] (-4.5,4.5) -- (0,0);
    \draw[-stealth,very thick](0,0) -- (4.5,4.5);
    \draw[stealth-stealth,very thick] (-4.5,0) -- (4.5,0); 
    % green interval:
    \filldraw[green,ultra thick] (0,0) -- (1,0);
    % pink intervals:
    \filldraw[magenta,ultra thick] (0,0) -- (-1.5,0);
    % yellow interval:
    \filldraw[yellow,ultra thick] (-1.5,0) -- (-3.8,0);
    % red interval:
    \filldraw[red,ultra thick] (1,0) -- (2.5,0);
    %mass dividers: (LH -> RH
    \draw[thick](-3.8,0) -- (-3.8,3.8);
    \draw[dotted,thick] (-3.1,0) -- (-3.1,3.1);
    \draw[dotted,thick] (-2.2,0) -- (-2.2,2.2);
    \draw[thick] (-1.5,0) -- (-1.5,1.5);
    \draw[dotted,thick] (-1,0) -- (-1,1);
    \draw[thick] (1,0) -- (1,1);
    \draw[dotted,thick] (1.9,0) -- (1.9,1.9);
    \draw[thick] (2.5,0) -- (2.5,2.5);
    \draw[dotted,thick] (3.2,0) -- (3.2,3.2);
    \draw[dotted,thick] (3.8,0) -- (3.8,3.8);

    %POINTS: 
    %(neg)
        %new:
    \filldraw [blue](-3.8,0) circle (1.5pt) node[anchor=north] {$-a''$};
    \filldraw [blue](-1.5,0) circle (1.5pt) node[anchor=north] {$-b'$};
        %old:
    \filldraw (-3.1,0) circle (1pt) node[anchor=north] {\scriptsize $-a$};
    \filldraw (-2.2,0) circle (1pt) node[anchor=north] {\scriptsize $-a'$};
    \filldraw (-1,0) circle (1pt) node[anchor=north] {\scriptsize $-b$};
    %(origin)
    \filldraw (0,0) circle (1.5pt) node[anchor=north] {\scriptsize};
    %(pos)
        %new:
    \filldraw [blue](1,0) circle (1.5pt) node[anchor=north] {$c$};
    \filldraw [blue](2.5,0) circle (1.5pt) node[anchor=north] {$e''$};
        %old:
    \filldraw (1.9,0) circle (1pt) node[anchor=north] {\scriptsize$d$};
    \filldraw (3.2,0) circle (1pt) node[anchor=north] {\scriptsize$e'$};
    \filldraw (3.8,0) circle (1pt) node[anchor=north] {\scriptsize$e$};
    
    %LABELS:
    \filldraw (-3.4,0) circle (.05pt) node[anchor=south] {\scriptsize$\updelta_x$};
    \filldraw (-2.65,0) circle (.05pt) node[anchor=south] {\tiny$R_3$};
    \filldraw (-1.8,0) circle (.05pt) node[anchor=south] {\scriptsize$\updelta_{x_1}$};
    \filldraw (2.2,0) circle (.05pt) node[anchor=south] {\scriptsize$\updelta_{x_1}$};
    \filldraw (2.85,0) circle (.05pt) node[anchor=south] {\scriptsize$\updelta_{x}$};
    \filldraw (3.5,0) circle (.05pt) node[anchor=south] {\tiny$A^+$};
    \filldraw (-1.2,0) circle (.05pt) node[anchor=south] {\tiny$A^+$};
    \filldraw (-.6,0) circle (.05pt) node[anchor=south] {\tiny$A^-$};
    \filldraw (.5,0) circle (.05pt) node[anchor=south] {\tiny$R_1$};
    \filldraw (1.5,0) circle (.05pt) node[anchor=south] {\tiny$A^+$};

    \end{tikzpicture}
    \caption{Case 2: Endpoints and masses labeled}
    \label{fig:medHistory}
\end{figure}
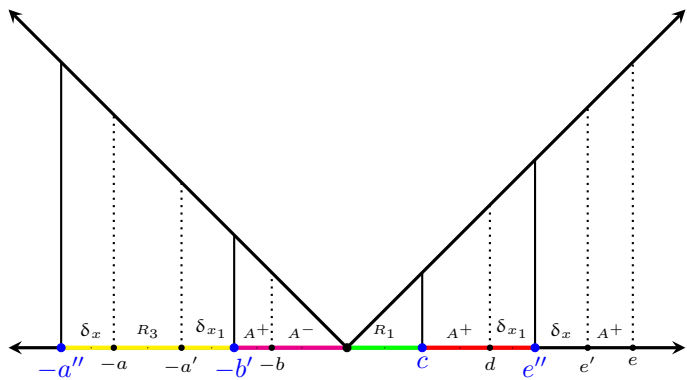

We want to show that $a''+b'+c+e'' \leq a+b+c+d+e$. To start, notice in Figure \ref{fig:medHistory}  that $[-b', -b]$ and $[c, d]$ both have mass of $A^+$. However, because $-b$ is farther away from the origin than $c$,  Proposition \ref{prop:LengthEndpointInequality} guarantees that $b'-b < d-c$. By some algebra, we have that $b' < b+d-c < b+d$.  %Furthermore, because both the old and new total perimeters have $c$, we can simply say $b' < b+d$. 
Additionally, notice that $[-a'', -a]$ has mass $\updelta_x$ and $[e'', e]$ has a greater mass of $(\updelta_x + A^+)$. Also note that $-a$ is farther from the origin than $e''$ because $R_3 + \updelta_{x_1} + A > \updelta_{x_1} + R_1$. Therefore, we are able to say $a''-a < e-e''$ (again by Proposition \ref{prop:LengthEndpointInequality}) and thus $a''+e'' < e+a$. Taking these inequalities together, we conclude that there is less total perimeter in this new arrangement.\bigskip

\textbf{Case 3: $R_2 < A \leq R_3$.}
First, we will again consolidate the region $A$. However, this time we will move $A^-$ to the positive side of the origin, adjoining it to $A^+$ and shifting $R_3$ to the right. Then, we move $R_3$ to the negative side of the origin (further away from the origin than $R_2$). Note that this will end with our intervals in our hypothesized isoperimetric configuration.

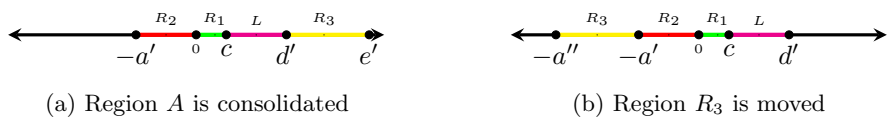
\begin{figure}[H]
     \centering
     \begin{subfigure}[h]{0.45\textwidth}
         \centering
         \begin{tikzpicture}
    %LINES:
    %x axis and |x|
    %\draw[stealth-,very thick] (-2,2) -- (0,0);
    %\draw[-stealth,very thick](0,0) -- (2,2);
    \draw[stealth-stealth,very thick] (-2.5,0) -- (2.5,0); 
    % green interval
    \filldraw[green,ultra thick] (0,0) -- (.4,0);
    % pink intervals
    \filldraw[magenta,ultra thick] (.4,0) -- (1.2,0);
    % yellow interval
    \filldraw[yellow,ultra thick] (1.2,0) -- (2.3,0);
    % red interval
    \filldraw[red,ultra thick] (-.8,0) -- (0,0);
    %POINTS: 
    %(neg)
    \filldraw (-.8,0) circle (1.5pt) node[anchor=north] {$-a'$};
    %(origin)
    \filldraw (0,0) circle (1.5pt) node[anchor=north] {\tiny$0$};
    %(pos)
    \filldraw (.4,0) circle (1.5pt) node[anchor=north] {$c$};
    \filldraw (1.2,0) circle (1.5pt) node[anchor=north] {$d'$};
    \filldraw (2.3,0) circle (1.5pt) node[anchor=north] {$e'$};
    %LABELS:
    \filldraw (-.4,0) circle (.05pt) node[anchor=south] {\tiny$R_2$};
    \filldraw (.8,0) circle (.05pt) node[anchor=south] {\tiny$L$};
    \filldraw (.25,0) circle (.05pt) node[anchor=south] {\tiny$R_1$};
    \filldraw (1.7,0) circle (.05pt) node[anchor=south] {\tiny$R_3$};

    \end{tikzpicture}
         \caption{Region $A$ is consolidated}
         \label{fig:largeConsolidated}
     \end{subfigure}
     \hfill
     \begin{subfigure}[h]{0.45\textwidth}
         \centering
         \begin{tikzpicture}
     %LINES:
    %x axis and |x|
    %\draw[stealth-,very thick] (-2,2) -- (0,0);
    %\draw[-stealth,very thick](0,0) -- (2,2);
    \draw[stealth-stealth,very thick] (-2.5,0) -- (2.5,0); 
    % green interval
    \filldraw[green,ultra thick] (0,0) -- (.4,0);
    % pink intervals
    \filldraw[magenta,ultra thick] (.4,0) -- (1.2,0);
    % yellow interval
    \filldraw[yellow,ultra thick] (-1.9,0) -- (-.8,0);
    % red interval
    \filldraw[red,ultra thick] (-.8,0) -- (0,0);
    %POINTS: 
    %(neg)
    \filldraw (-1.9,0) circle (1.5pt) node[anchor=north] {$-a''$};
    \filldraw (-.8,0) circle (1.5pt) node[anchor=north] {$-a'$};
    %(origin)
    \filldraw (0,0) circle (1.5pt) node[anchor=north] {\tiny$0$};
    %(pos)
    \filldraw (.4,0) circle (1.5pt) node[anchor=north] {$c$};
    \filldraw (1.2,0) circle (1.5pt) node[anchor=north] {$d'$};
    %LABELS:
    \filldraw (-1.35,0) circle (.05pt) node[anchor=south] {\tiny$R_3$};
    \filldraw (-.4,0) circle (.05pt) node[anchor=south] {\tiny$R_2$};
    \filldraw (.25,0) circle (.05pt) node[anchor=south] {\tiny$R_1$};
    \filldraw (.8,0) circle (.05pt) node[anchor=south] {\tiny$L$};
    \end{tikzpicture}
         \caption{Region $R_3$ is moved}
         \label{fig:largeFinal}
     \end{subfigure}
      \caption{Case 3: Initial rearrangements}
        \label{fig:initSteps2}
\end{figure}

\bigskip

To keep track of our old and new endpoints, we have the following figure. Note that the new total perimeter will consist of $a''+a'+c+d'$. Also, note that we are defining new masses $\updelta_{x_2} = R_3-A^-$ and $\updelta_{x_3} = R_2-A^-$.

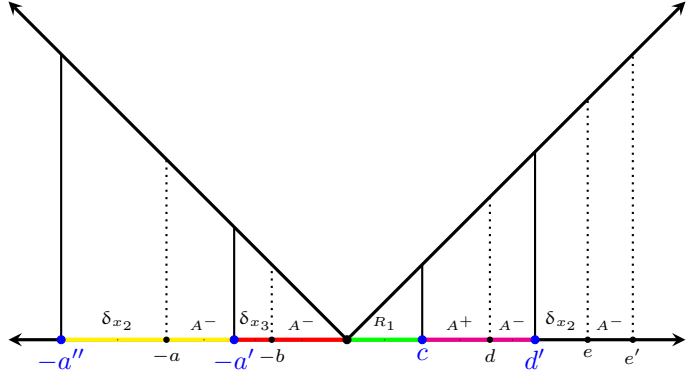
\begin{figure}[H]
    \centering
    \begin{tikzpicture}
    %LINES:
    %x axis and |x|:
    \draw[stealth-,very thick] (-4.5,4.5) -- (0,0);
    \draw[-stealth,very thick](0,0) -- (4.5,4.5);
    \draw[stealth-stealth,very thick] (-4.5,0) -- (4.5,0); 
    % green interval:
    \filldraw[green,ultra thick] (0,0) -- (1,0);
    % red interval:
    \filldraw[red,ultra thick] (0,0) -- (-1.5,0);
    % yellow interval:
    \filldraw[yellow,ultra thick] (-1.5,0) -- (-3.8,0);
    % pink interval:
    \filldraw[magenta,ultra thick] (1,0) -- (2.5,0);
    %mass dividers: (LH -> RH
    \draw[thick](-3.8,0) -- (-3.8,3.8);
    \draw[dotted,thick] (-2.4,0) -- (-2.4,2.4);
    \draw[thick] (-1.5,0) -- (-1.5,1.5);
    \draw[dotted,thick] (-1,0) -- (-1,1);
    \draw[thick] (1,0) -- (1,1);
    \draw[dotted,thick] (1.9,0) -- (1.9,1.9);
    \draw[thick] (2.5,0) -- (2.5,2.5);
    \draw[dotted,thick] (3.2,0) -- (3.2,3.2);
    \draw[dotted,thick] (3.8,0) -- (3.8,3.8);

    %POINTS: 
    %(neg)
        %new:
    \filldraw [blue](-3.8,0) circle (1.5pt) node[anchor=north] {$-a''$};
    \filldraw [blue](-1.5,0) circle (1.5pt) node[anchor=north] {$-a'$};
        %old:
    \filldraw (-2.4,0) circle (1pt) node[anchor=north] {\scriptsize $-a$};
    \filldraw (-1,0) circle (1pt) node[anchor=north] {\scriptsize $-b$};
    %(origin)
    \filldraw (0,0) circle (1.5pt) node[anchor=north] {\scriptsize};
    %(pos)
        %new:
    \filldraw [blue](1,0) circle (1.5pt) node[anchor=north] {$c$};
    \filldraw [blue](2.5,0) circle (1.5pt) node[anchor=north] {$d'$};
        %old:
    \filldraw (1.9,0) circle (1pt) node[anchor=north] {\scriptsize$d$};
    \filldraw (3.2,0) circle (1pt) node[anchor=north] {\scriptsize$e$};
    \filldraw (3.8,0) circle (1pt) node[anchor=north] {\scriptsize$e'$};
    
    %LABELS:
    
    \filldraw (-3.05,0) circle (.05pt) node[anchor=south] {\scriptsize$\updelta_{x_2}$};
    \filldraw (-1.9,0) circle (.05pt) node[anchor=south] {\tiny$A^-$};
    \filldraw (-1.22,0) circle (.05pt) node[anchor=south] {\scriptsize$\updelta_{x_3}$};
    \filldraw (-.6,0) circle (.05pt) node[anchor=south] {\tiny$A^-$};
    \filldraw (.5,0) circle (.05pt) node[anchor=south] {\tiny$R_1$};
    \filldraw (1.5,0) circle (.05pt) node[anchor=south] {\tiny$A^+$};
    \filldraw (2.2,0) circle (.05pt) node[anchor=south] {\tiny$A^-$};
    \filldraw (2.85,0) circle (.05pt) node[anchor=south] {\scriptsize$\updelta_{x_2}$};
    \filldraw (3.5,0) circle (.05pt) node[anchor=south] {\tiny$A^-$};

    \end{tikzpicture}
    \caption{Case 3: Masses and endpoints labeled}
    \label{fig:largeHistory}
\end{figure}

We want to examine the relationship between the new and old perimeter points. First, note that because $A \geq R_2$, we have the following two inequalities.

%\begin{equation} 
%\label{eqn:star1}
%R_1+A \geq R_2 \qquad \text{and} \qquad R_1+A^+ \geq R_2-A^-. 
%\end{equation} %we'll call this statement ``\emph{star}.'' 

\begin{align}
    R_1+A &\geq R_2\\
    R_1+A^+ &\geq R_2-A^- \label{eqn:star1}
\end{align}

There are two sub-cases to consider:

\begin{enumerate}
    \item $R_1+A^+ \leq R_2$:
    
    $[-a'',-a]$ and $[d',e]$ have mass $R_3-A^-$. However, in this case, $a' > d$, so $a > d'$. This implies that $[-a'',-a]$ is narrower than $[d',e]$ by Proposition \ref{prop:LengthEndpointInequality}, giving us $a''- a \leq e-d'$, and therefore $a''+d' \leq e+a$. By \eqref{eqn:star1}, $[0,d]$ contains more mass than $[-a',-b]$. In addition, $[0,d]$ has an inner endpoint that is closer to the origin than $[-a',-b]$, meaning $[-a',-b]$ is narrower than $[0,d]$. This allows us to conclude that $a'-b < d-0$ and therefore, $a'<b+d$.
    
    \item $R_1+A^+ > R_2$: 
    %BE CAREFUL OF INEQUALITY SIGNS HERE. CLARIFY. IF ABOVE CASE IS GEQ THEN SHOULDN'T THIS JUST BE > ?
    
    $[-a,-a']$ and $[d,d']$ both have mass $A^-$. In this case $d>a'$, so $[-a,-a']$ is not as narrow as $[d,d']$. Again using Proposition \ref{prop:LengthEndpointInequality}, we can say that $d'-d \leq a-a'$, which gives that $a'+d' \leq a+d$. Additionally, $[a'',-b]$ has mass no greater than $[0,e]$ (which is true since $R_2-A^- \leq R_1 + A^+$ by \eqref{eqn:star1}), giving us that $[a'',-b]$ is narrower. This allows us to conclude that $a''-b \leq e-0$, and therefore that $a''< b+e$.
\end{enumerate}

In each sub-case, our new endpoints contribute less perimeter than before, yielding an isoperimetric arrangement of four intervals.
\bigskip

\textbf{Case 4: $A > R_3$}. As before, we consolidate $A$ into a single interval. This time, we do so on the negative side of the origin. Then, we transpose the $A$ and the $R_2$ interval. The final arrangement has a new total perimeter of $a'+b''+c+e'$, which we compare to the old perimeter of $a+b+c+d+e$.

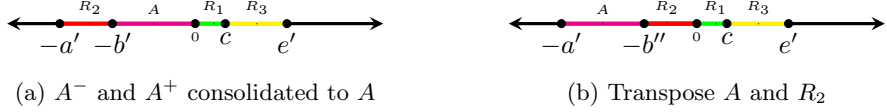
\begin{figure}[H]
     \centering
     \begin{subfigure}[h]{0.45\textwidth}
         \centering
         \begin{tikzpicture}
    %LINES:
    %x axis and |x|
    %\draw[stealth-,very thick] (-2,2) -- (0,0);
    %\draw[-stealth,very thick](0,0) -- (2,2);
    \draw[stealth-stealth,very thick] (-2.5,0) -- (2.5,0); 
    % green interval
    \filldraw[green,ultra thick] (0,0) -- (.4,0);
    % pink intervals
    \filldraw[magenta,ultra thick] (-1.1,0) -- (0,0);
    % yellow interval
    \filldraw[yellow,ultra thick] (.4,0) -- (1.22,0);
    % red interval
    \filldraw[red,ultra thick] (-1.1,0) -- (-1.8,0);
    %POINTS: 
    %(neg)
    \filldraw (-1.8,0) circle (1.5pt) node[anchor=north] {$-a'$};
    \filldraw (-1.1,0) circle (1.5pt) node[anchor=north] {$-b'$};
    %(origin)
    \filldraw (0,0) circle (1.5pt) node[anchor=north] {\tiny$0$};
    %(pos)
    \filldraw (.4,0) circle (1.5pt) node[anchor=north] {$c$};
    \filldraw (1.22,0) circle (1.5pt) node[anchor=north] {$e'$};
    
    %LABELS:
    \filldraw (-1.4,0) circle (.05pt) node[anchor=south] {\tiny$R_2$};
    \filldraw (-.55,0) circle (.05pt) node[anchor=south] {\tiny$A$};
    \filldraw (.25,0) circle (.05pt) node[anchor=south] {\tiny$R_1$};
    \filldraw (.82,0) circle (.05pt) node[anchor=south] {\tiny$R_3$};
    \end{tikzpicture}
         \caption{$A^-$ and $A^+$ consolidated to $A$ }
         \label{fig:largestConsolidated}
     \end{subfigure}
     \hfill
     \begin{subfigure}[h]{0.45\textwidth}
         \centering
         \begin{tikzpicture}
    %LINES:
    %x axis and |x|
    %\draw[stealth-,very thick] (-2,2) -- (0,0);
    %\draw[-stealth,very thick](0,0) -- (2,2);
    \draw[stealth-stealth,very thick] (-2.5,0) -- (2.5,0); 
    % green interval
    \filldraw[green,ultra thick] (0,0) -- (.4,0);
    % pink intervals
    \filldraw[magenta,ultra thick] (-1.8,0) -- (0,0);
    % yellow interval
    \filldraw[yellow,ultra thick] (.4,0) -- (1.22,0);
    % red interval
    \filldraw[red,ultra thick] (-.7,0) -- (0,0);
    %POINTS: 
    %(neg)
    \filldraw (-1.8,0) circle (1.5pt) node[anchor=north] {$-a'$};
    \filldraw (-.7,0) circle (1.5pt) node[anchor=north] {$-b''$};
    %(origin)
    \filldraw (0,0) circle (1.5pt) node[anchor=north] {\tiny$0$};
    %(pos)
    \filldraw (.4,0) circle (1.5pt) node[anchor=north] {$c$};
    \filldraw (1.22,0) circle (1.5pt) node[anchor=north] {$e'$};
    
    %LABELS:
    \filldraw (-1.25,0) circle (.05pt) node[anchor=south] {\tiny$A$};
    \filldraw (-.4,0) circle (.05pt) node[anchor=south] {\tiny$R_2$};
    \filldraw (.25,0) circle (.05pt) node[anchor=south] {\tiny$R_1$};
    \filldraw (.82,0) circle (.05pt) node[anchor=south] {\tiny$R_3$};
    \end{tikzpicture}
         \caption{Transpose $A$ and $R_2$}
         \label{fig:largestFinal}
     \end{subfigure}
      \caption{Case 4: Initial rearrangements}
        \label{fig:initSteps3}
\end{figure}

Like before, we can keep track of all of the changes we have made with Figure \ref{fig:largestHistory}, which has all endpoints and masses identified. As part of this image, we continue to let $\updelta_{x_3} = R_2 - A^-$ introduce mass $\updelta_{x_4} = R_3-A^+$

\begin{figure}[H]
    \centering
    \begin{tikzpicture}
     %LINES:
    %x axis and |x|:
    \draw[stealth-,very thick] (-4.5,4.5) -- (0,0);
    \draw[-stealth,very thick](0,0) -- (4.5,4.5);
    \draw[stealth-stealth,very thick] (-4.5,0) -- (4.5,0); 
    % green interval:
    \filldraw[green,ultra thick] (0,0) -- (1,0);
    % red intervals:
    \filldraw[red,ultra thick] (0,0) -- (-1.5,0);
    % pink interval:
    \filldraw[magenta,ultra thick] (-1.5,0) -- (-3.8,0);
    % yellow interval:
    \filldraw[yellow,ultra thick] (1,0) -- (2.5,0);
    %mass dividers: (LH -> RH
    \draw[thick](-3.8,0) -- (-3.8,3.8);
    \draw[dotted,thick] (-3.1,0) -- (-3.1,3.1);
    \draw[dotted,thick] (-2.5,0) -- (-2.5,2.5);
    \draw[thick] (-1.5,0) -- (-1.5,1.5);
    \draw[dotted,thick] (-1,0) -- (-1,1);
    \draw[thick] (1,0) -- (1,1);
    \draw[dotted,thick] (1.9,0) -- (1.9,1.9);
    \draw[thick] (2.5,0) -- (2.5,2.5);
    \draw[dotted,thick] (3.2,0) -- (3.2,3.2);

    %POINTS: 
    %(neg)
        %new:
    \filldraw [blue](-3.8,0) circle (1.5pt) node[anchor=north] {$-a'$};
    \filldraw [blue](-1.5,0) circle (1.5pt) node[anchor=north] {$-b''$};
        %old:
    \filldraw (-3.1,0) circle (1pt) node[anchor=north] {\scriptsize $-a$};
    \filldraw (-2.5,0) circle (1pt) node[anchor=north] {\scriptsize $-b'$};
    \filldraw (-1,0) circle (1pt) node[anchor=north] {\scriptsize $-b$};
    %(origin)
    \filldraw (0,0) circle (1.5pt) node[anchor=north] {\scriptsize};
    %(pos)
        %new:
    \filldraw [blue](1,0) circle (1.5pt) node[anchor=north] {$c$};
    \filldraw [blue](2.5,0) circle (1.5pt) node[anchor=north] {$e'$};
        %old:
    \filldraw (1.9,0) circle (1pt) node[anchor=north] {\scriptsize$d$};
    \filldraw (3.2,0) circle (1pt) node[anchor=north] {\scriptsize$e$};
    
    %LABELS:
    \filldraw (-3.4,0) circle (.05pt) node[anchor=south] {\tiny$A^+$};
    \filldraw (-2.75,0) circle (.05pt) node[anchor=south] {\scriptsize$\updelta_{x_1}$};
    \filldraw (-2,0) circle (.05pt) node[anchor=south] {\tiny$A-R_2$};
    \filldraw (2.2,0) circle (.05pt) node[anchor=south] {\tiny{$\updelta_{x_4}$}};
    \filldraw (2.85,0) circle (.05pt) node[anchor=south] {\tiny$A^+$};
    \filldraw (-1.2,0) circle (.05pt) node[anchor=south] {\scriptsize$\updelta_{x_3}$};
    \filldraw (-.6,0) circle (.05pt) node[anchor=south] {\tiny$A^-$};
    \filldraw (.5,0) circle (.05pt) node[anchor=south] {\tiny$R_1$};
    \filldraw (1.5,0) circle (.05pt) node[anchor=south] {\tiny$A^+$};

    \end{tikzpicture}
    \caption{Case 4: Masses and endpoints labeled}
    \label{fig:largestHistory}
\end{figure}
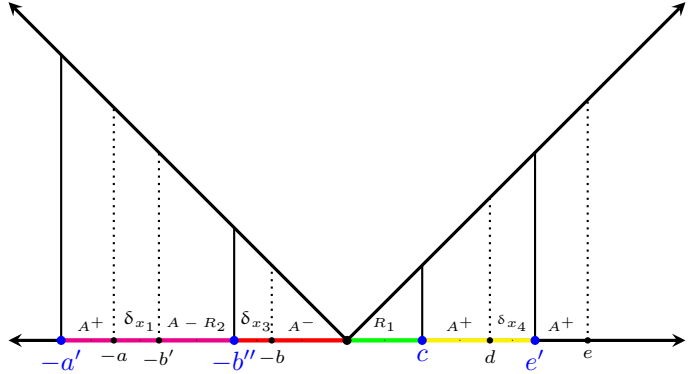

There are two sub-cases to consider:
\begin{enumerate}
    \item $R_2+A^- \geq R_1+R_3$:
    
    The interval $[-a',-a]$ has mass $A^+$  and is separated from the origin by an interval of mass $R_2+A^-$. $[e',e]$ has mass $A^+$ as well, but is separated from the origin by an interval of mass $R_1 + R_3$ away from the origin. We conclude $a'-a \leq e-e'$ and therefore that $a'+e' \leq e+a$. 
    
    Next, we want to show that $b'' < b+d$. Now, because it is true that $A \geq R_3 \geq R_2 \geq R_1$, we see that %by some algebra we can say $(L_1+L_2)-M \geq X_1-M$, meaning that 
    $A^- + A^+ > R_2-R_1$ is certainly true. Rearranging, we can have $R_1+A^+ > R_2 - A^-$. This is important because $[-b'',-b]$ has mass $R_2-A^-$ and $[0,d]$ has mass $R_1+A^+$. The interval $[-b'',-b]$ has a smaller mass and is farther away from the origin, so by Proposition \ref{prop:LengthEndpointInequality} we have that $b''-b < d-0$, and therefore $b''<d+b$. 
    
    Thus, there is less total perimeter in this new configuration. %than before in the case that $X_1+L_1 \geq X_2+M$.
    
    \item $R_2+ A^- < R_1+R_3$:
    
    The interval $[d,e']$ has mass $R_3-A^+$ and is separated from the origin by an interval of mass $R_1+A^+$. At the same time, $[-a,-b'']$ has mass $A^-$ and is separated from the origin by an interval of mass $R_2$. We know that $A \geq R_3 \geq R_2 \geq R_1$, and therefore $A^- \geq R_3-A^+$. Thus, the interval $[-a,-b'']$ has less mass than $[d,e']$. Additionally, $b''$ is closer to the origin than $d$. Applying Proposition \ref{prop:LengthEndpointInequality} once more gives us $a < e'$ and that $a-b'' > e'-d$, and therefore that $a+d > b''+e'$. 
    
    Lastly, we want to examine the endpoint $a'$. We begin by noting that the interval $[-a',-b]$ has a mass of $R_2+A^+$ while $[0,e]$ has mass $R_1+R_3+A^+$. By the relative size of our masses, we know $R_2+A^+ \leq R_1+R_3+A^+$. So, not only does $[0,e]$ have the larger mass than $[-a', -b]$, but it is also closer to the origin. Therefore, we can say that $a'-b \leq e-0$ and thus $a' \leq e+b$. 
    
    Taking these earlier inequalities together, we see that in this case the new configuration has less total perimeter.%So, in the case that $X_1+L_1 < X_2+M$, the alternating scheme gives less total perimeter than before.
    \end{enumerate}
    
    We observe that, in any possible case or sub-case, our configuration of 4 intervals uses less perimeter than any possible candidate with five intervals. Therefore, our proposed 4-interval solution is the perimeter-minimizing way to arrange four regions on $\mathbb{R}$ with density $|x|$.
\end{proof}

\section{Appendix: The ordering propositions}

In this appendix, we collect theorems that identify appropriate orderings of intervals. Specifically, we establish the ordering of $n$-interval frameworks, for $n=4,5,6$, where the masses come from distinct intervals and have masses $M_1 \leq \dots \leq M_n$. It turns out that the $n=4$ and $n=5$ cases can be seen as corollaries of the $n=6$ case, in the specific situation where one (or two) of the intervals have size 0. Thus, in this appendix we prove the 6-interval framework and state the 4- and 5-interval propositions as corollaries.

\begin{proposition}[Alternating 4-interval framework]
\label{Prop:4IntervalFramework}
On $\mathbb{R}$ with density $|x|$: Suppose we have four regions, condensed so that each region consists of a single interval, and organized so that there are two intervals on each side of the origin. Then the least-perimeter way to meet these specifications is to have intervals alternate back and forth across the origin as they increase from smallest mass to largest. This means that the masses can be ordered, up to reflection, as $R_3 R_1 . R_2 R_4$ (where the masses of region $R_i$ satisfy $M_1 \leq M_2 \leq M_3 \leq M_4$). 
\end{proposition}

\begin{proposition}[Alternating 5-interval framework]
\label{Prop:5IntervalFramework}
On $\mathbb{R}$ with density $|x|$: Suppose we have five regions, condensed so that each region consists of a single interval, and organized so that there are two intervals on one side of the origin and three on the other side. Then the least-perimeter way to meet these specifications is to have intervals alternate back and forth across the origin as they increase from smallest mass to largest. This means that the masses can be ordered, up to reflection, as $R_5 R_3 R_1 . R_2 R_4$ (where the masses of region $R_i$ satisfy $M_1 \leq M_2 \leq M_3 \leq M_4 \leq M_5$). 
\end{proposition}

As said above, these are corollaries of the  6-interval framework proposition, which we state and prove below.

%It turns out that the Alternating 4-interval framework proposition is an example of the alternating 6-interval framework proposition, which is detailed below. The adjustments can be made by setting the smallest two intervals equal to mass 0. As such, we will only prove the alternating 6-interval framework here.

\begin{proposition}[Alternating 6-interval framework]
\label{Prop:6IntervalFramework}
On $\mathbb{R}$ with density $|x|$: Suppose we have six regions, condensed so that each region consists of a single interval, and organized so that there are three intervals on each side of the origin. Then the least-perimeter way to meet these specifications is to have intervals alternate back and forth across the origin as they increase from smallest mass to largest. This means that the masses can be ordered, up to reflection, as $ R_5 R_3 R_1 . R_2 R_4 R_6$ (where the masses of region $R_i$ satisfy $M_1 \leq \dots \leq M_6$). 
\end{proposition}

\begin{proof}[Proof of Proposition 6.3]

%We will identify our six regions by their masses as $S$, $M$, $L_1$, $L_2$, $X_1$, and $X_2$. Assume that $S$ is the region with smallest mass, and therefore that $b < c$. Additionally, we can assume that $S \leq L_1 \leq X_1$ and that $M \leq L_2 \leq X_2$. However, we do not know the relative sizes of regions on opposite sides of the origin.

This is a proof of many cases. Proposition \ref{prop:TranspositionLemma} tells us that each side of the origin will be ordered from smallest to largest mass (as we move away from the origin). We will identify the regions of smallest mass and second-smallest mass as ``$S$'' and ``$M$'' throughout this proof.

WLOG, we can assume that $S$, appears directly adjacent to the origin on the negative side. The major dividing line in our cases is whether or not $M$, appears on the same side of the origin as $S$. In the case where $M$ appears on the opposite side of the origin, we will label our other regions as having masses $L_1$, $L_2$, $X_1$, and $X_2$, and consider the relative sizes of the $L_i$ and $X_i$. We can use Figure \ref{fig:6IntervalReferencePic1} as a reference.

\begin{figure}[H]
\centering
\begin{tikzpicture}
    %AXIS AND |x|:
    \draw[stealth-,very thick] (-2.2,2.2) -- (0,0);
    \draw[-stealth,very thick](0,0) -- (2.2,2.2);
    \draw[stealth-stealth,very thick] (-5,0) -- (5,0);
    %POINTS: (top->down = left->right)
    \filldraw (-3.7,0) circle (1pt) node[anchor=north] {\tiny $-a_2$};
    \filldraw (-1.9,0) circle (1pt) node[anchor=north] {\tiny$-a_1$};
    \filldraw (-.8,0) circle (1pt) node[anchor=north] {\tiny$-b$};
    \filldraw (0,0) circle (1pt) node[anchor=north] {\tiny$0$};
    \filldraw (1,0) circle (1pt) node[anchor=north] {\tiny$c$};
    \filldraw (2.2,0) circle (1pt) node[anchor=north] {\tiny$d_1$};
    \filldraw (4,0) circle (1pt) node[anchor=north] {\tiny$d_2$};
    %VERTICAL LINES:
    %\draw (-1.7,0) -- (-1.7,1.7);
    %%\draw (-.9,0) -- (-.9,.9);
    %\draw (-.3,0) -- (-.3,.3);
    %\draw (.5,0) -- (.5,.5);
    %\draw (1.2,0) -- (1.2,1.2);
    %\draw (2,0) -- (2,2);
    %MASSES:
    \filldraw [blue](-2.8,0) circle (.05pt) node[anchor=south] {\tiny $X_1$};
    \filldraw [blue](-1.3,0) circle (.05pt) node[anchor=south] {\tiny$L_1$};
    \node[anchor=south] at (-.4,0) {\textcolor{blue}{\tiny$S$}};
    \node[anchor=south] at (.6,0) {\textcolor{blue}{\tiny$M$}};
    \filldraw [blue](1.6,0) circle (.05pt) node[anchor=south] {\tiny$L_2$};
    \filldraw [blue](3.1,0) circle (.05pt) node[anchor=south] {\tiny$X_2$};
    \end{tikzpicture}
    
    \caption{The figure we use as an initial configuration when the second-smallest mass, $M$, is on the opposite side of the origin as $S$.}
    \label{fig:6IntervalReferencePic1}
\end{figure}
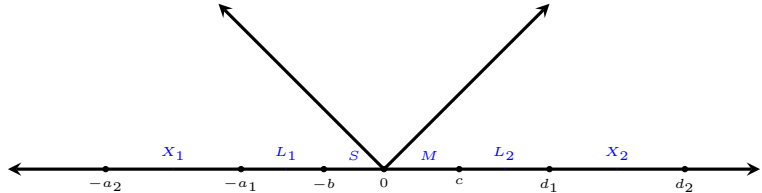
We can assume (by Proposition \ref{prop:TranspositionLemma}) that $L_1 < X_1$ and $L_2 < X_2$, but there are still many possible cases to consider. %We want to show that the optimal configuration is $L_1 \leq L_2 \leq X_1 \leq X_2$, and we prove this by cases:
    \begin{itemize}
        \item Case 0: $S \leq M \leq L_1 \leq L_2 \leq X_1 \leq X_2$. This is the case we are claiming is optimal.
        
        \item Case 1: $S \leq M \leq L_1 \leq L_2 \leq X_2 \leq X_1$. In this case, we can further reduce perimeter by flipping the positions of the $X_i$'s. The process is shown in the image below.
        
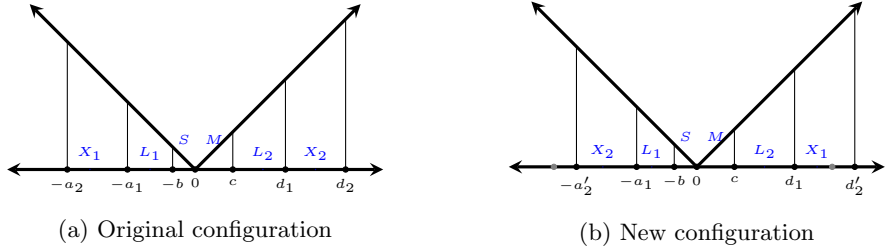
\begin{figure}[H]
    \centering
    \begin{subfigure}[h]{0.45\textwidth}
         \centering
         \begin{tikzpicture}
    %AXIS AND |x|:
    \draw[stealth-,very thick] (-2.2,2.2) -- (0,0);
    \draw[-stealth,very thick](0,0) -- (2.2,2.2);
    \draw[stealth-stealth,very thick] (-2.5,0) -- (2.5,0);
    %POINTS: (top->down = left->right)
    \filldraw (-1.7,0) circle (1pt) node[anchor=north] {\tiny $-a_2$};
    \filldraw (-.9,0) circle (1pt) node[anchor=north] {\tiny$-a_1$};
    \filldraw (-.3,0) circle (1pt) node[anchor=north] {\tiny$-b$};
    \filldraw (0,0) circle (1pt) node[anchor=north] {\tiny$0$};
    \filldraw (.5,0) circle (1pt) node[anchor=north] {\tiny$c$};
    \filldraw (1.2,0) circle (1pt) node[anchor=north] {\tiny$d_1$};
    \filldraw (2,0) circle (1pt) node[anchor=north] {\tiny$d_2$};
    %VERTICAL LINES:
    \draw (-1.7,0) -- (-1.7,1.7);
    \draw (-.9,0) -- (-.9,.9);
    \draw (-.3,0) -- (-.3,.3);
    \draw (.5,0) -- (.5,.5);
    \draw (1.2,0) -- (1.2,1.2);
    \draw (2,0) -- (2,2);
    %MASSES:
    \filldraw [blue](-1.4,0) circle (.05pt) node[anchor=south] {\tiny $X_1$};
    \filldraw [blue](-.6,0) circle (.05pt) node[anchor=south] {\tiny$L_1$};
    \node at (-.15,.4) {\textcolor{blue}{\tiny$S$}};
    \node at (.25,.4) {\textcolor{blue}{\tiny$M$}};
    \filldraw [blue](.9,0) circle (.05pt) node[anchor=south] {\tiny$L_2$};
    \filldraw [blue](1.6,0) circle (.05pt) node[anchor=south] {\tiny$X_2$};
    \end{tikzpicture}
    
    \caption{Original configuration}
    \label{fig:5.3.1.1}
    \end{subfigure}
    \hfill
    \begin{subfigure}[h]{0.45\textwidth}
    \centering
    
    \begin{tikzpicture}
     %AXIS AND |x|:
    \draw[stealth-,very thick] (-2.2,2.2) -- (0,0);
    \draw[-stealth,very thick](0,0) -- (2.2,2.2);
    \draw[stealth-stealth,very thick] (-2.5,0) -- (2.5,0);
    %POINTS: (top->down = left->right)
    \filldraw [gray](-1.9,0) circle (1pt);
    \filldraw (-1.6,0) circle (1pt) node[anchor=north] {\tiny $-a'_2$};
    %\filldraw [gray](-1.1,0) circle (1pt);
    \filldraw (-.8,0) circle (1pt) node[anchor=north] {\tiny$-a_1$};
    \filldraw (-.3,0) circle (1pt) node[anchor=north] {\tiny$-b$};
    \filldraw (0,0) circle (1pt) node[anchor=north] {\tiny$0$};
    \filldraw (.5,0) circle (1pt) node[anchor=north] {\tiny$c$};
    %\filldraw [gray](1,0) circle (1pt);   
    \filldraw (1.3,0) circle (1pt) node[anchor=north] {\tiny$d_1$};
    \filldraw [gray](1.8,0) circle (1pt);
    \filldraw (2.1,0) circle (1pt) node[anchor=north] {\tiny$d'_2$};
    %VERTICAL LINES:
    \draw (-1.6,0) -- (-1.6,1.6);
    \draw (-.8,0) -- (-.8,.8);
    \draw (-.3,0) -- (-.3,.3);
    \draw (.5,0) -- (.5,.5);
    \draw (1.3,0) -- (1.3,1.3);
    \draw (2.1,0) -- (2.1,2.1);
    %MASSES:
    \filldraw [blue](-1.25,0) circle (.05pt) node[anchor=south] {\tiny $X_2$};
    \filldraw [blue](-.6,0) circle (.05pt) node[anchor=south] {\tiny$L_1$};
    \node at (-.15,.4) {\textcolor{blue}{\tiny$S$}};
    \node at (.25,.4) {\textcolor{blue}{\tiny$M$}};
    \filldraw [blue](.9,0) circle (.05pt) node[anchor=south] {\tiny$L_2$};
    \filldraw [blue](1.6,0) circle (.05pt) node[anchor=south] {\tiny$X_1$};
    \end{tikzpicture}
    \caption{New configuration}
    \label{fig:5.3.2.2}
    \end{subfigure}
    \caption{Case 1: $X_i$'s have flipped their positions.}
    \label{fig:5.3.1.0}
\end{figure}

\begin{figure}[H]
    \centering
    \begin{tikzpicture}
    %LINES:
    %x axis and |x|:
    \draw[stealth-,very thick] (-4,4) -- (0,0);
    \draw[-stealth,very thick](0,0) -- (4,4);
    \draw[stealth-stealth,very thick] (-4.5,0) -- (4.5,0); 
    %mass dividers: (LH -> RH)
    \draw [dotted,thick] (-3.8,0) -- (-3.8,3.8);
    \draw (-3.1,0) -- (-3.1,3.1);
    \draw (-2.3,0) -- (-2.3,2.3);
    %\draw (-1.5,0) -- (-1.5,1.5);
    \draw (-1,0) -- (-1,1);
    \draw (1,0) -- (1,1);
    \draw (1.6,0) -- (1.6,1.6);
    %\draw (2,0) -- (2,2);
    \draw [dotted,thick] (2.6,0) -- (2.6,2.6);
    \draw (3.2,0) -- (3.2,3.2);
    %POINTS: 
    \filldraw (-3.8,0) circle (1pt) node[anchor=north] {\scriptsize$-a_2$};
    \filldraw (-3.1,0) circle (1.5pt) node[anchor=north] {\scriptsize $-a'_2$};
    \filldraw (-2.3,0) circle (1pt) node[anchor=north] {\scriptsize $-a_1$};
    %\filldraw (-1.5,0) circle (1.5pt) node[anchor=north] {\scriptsize$-a'_1$};
    \filldraw (-1,0) circle (1.5pt) node[anchor=north] {\scriptsize $-b$};
    \filldraw (0,0) circle (1pt) node[anchor=north] {\scriptsize $0$};
    \filldraw (1,0) circle (1.5pt) node[anchor=north] {\scriptsize$c$};
    \filldraw (1.6,0) circle (1pt) node[anchor=north] {\scriptsize$d_1$};
    %\filldraw (2,0) circle (1.5pt) node[anchor=north] {\scriptsize$d'_1$};
    \filldraw (2.6,0) circle (1pt) node[anchor=north] {\scriptsize$d_2$};
    \filldraw (3.2,0) circle (1.5pt) node[anchor=north] {\scriptsize$d'_2$};
    %LABELS:
    %BELOW: turning a node label 90 degrees to fit in the column
    %\node[rotate=90] at (-3.4,1){\scriptsize$\updelta_X+\updelta_L$};
    \filldraw (-3.4,0) circle (0.5pt) node[anchor=south] {\scriptsize$\delta$};
    \filldraw (-2.75,0) circle (0.5pt) node[anchor=south] {\scriptsize$X_2$};
    %\node[rotate=90] at (-2.75,1){\tiny$X_2$};
    \filldraw (-1.4,0) circle (.05pt) node[anchor=south] {\tiny$L_1$};
    %\filldraw (-1.9,0) circle (.05pt) node[anchor=south] {\scriptsize$\updelta_L$};
    \filldraw (-.6,0) circle (.05pt) node[anchor=south] {\tiny$S$};
    \filldraw (.5,0) circle (.05pt) node[anchor=south] {\tiny$M$};
    \filldraw (1.3,0) circle (.05pt) node[anchor=south] {\tiny$L_2$};
    \filldraw (2,0) circle (.05pt) node[anchor=south] {\tiny$X_2$};
    \filldraw (2.8,0) circle (.05pt) node[anchor=south] {\tiny$\delta$};
    %\node[rotate=90] at (1.8,.8){\scriptsize$\updelta_L$};
    %\node[rotate=90] at (2.3,1){\scriptsize$X_2-\updelta_L$};
    %\node[rotate=90] at (2.85,1) {\scriptsize$\updelta_X+\updelta_L$};
    \end{tikzpicture}
    \caption{Case 1: all endpoints and masses identified.}
    \label{fig:5.3.1.3}
\end{figure}
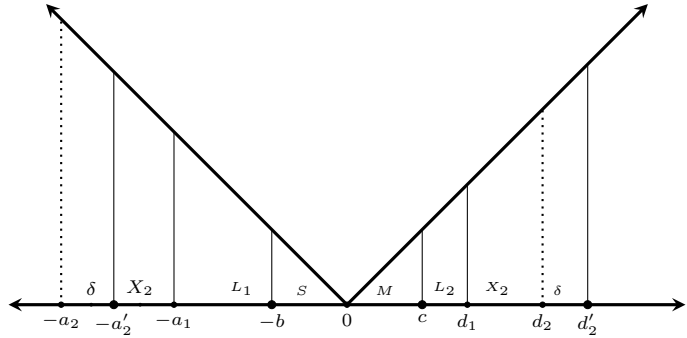

Specifically, we can define $\delta =X_1 - X_2$. Note that the original configuration has perimeter $a_2 + a_1 + b + c + d_1 + d_2$, while the modified configuration has perimeter $a_2' + a_1 + b + c + d_1 + d_2'$. Because $S \leq M$ and $L_1 \leq L_2$, see that the mass in $[-a_2',0]$ is less than the mass in $[0, d_2]$. This tells us that $a_2' < d_2$. This, in turn, implies that the interval $[-a_2, -a_2']$ is not as narrow as the interval $[d_2, d_2']$ since both contain the same mass $\delta$. We conclude that $d_2' - d_2 < a_2 - a_2'$, and therefore that $a_2' + d_2' < a_2 + d_2$. Thus, our new configuration has less total perimeter than our old configuration. This returns us to a framework that is identical to Case 0.

        \item Case 2: $S \leq M \leq L_2 \leq L_1 \leq X_2 \leq X_1$. In this case, we can further reduce perimeter by flipping the positions of the $X$'s and $L$'s simultaneously. The process is shown in the picture below. 
% DOUBLE GRAPH FIGURE
    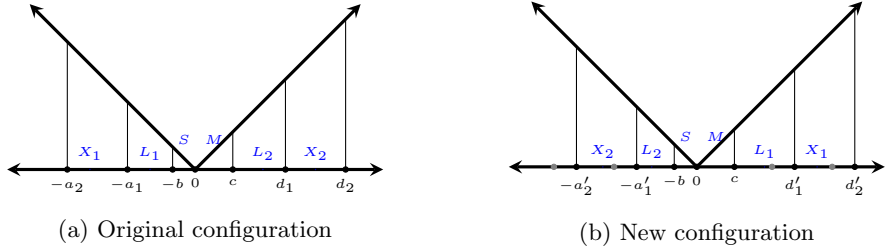
\begin{figure}[H]
    \centering
    \begin{subfigure}[h]{0.45\textwidth}
         \centering
         \begin{tikzpicture}
    %AXIS AND |x|:
    \draw[stealth-,very thick] (-2.2,2.2) -- (0,0);
    \draw[-stealth,very thick](0,0) -- (2.2,2.2);
    \draw[stealth-stealth,very thick] (-2.5,0) -- (2.5,0);
    %POINTS: (top->down = left->right)
    \filldraw (-1.7,0) circle (1pt) node[anchor=north] {\tiny $-a_2$};
    \filldraw (-.9,0) circle (1pt) node[anchor=north] {\tiny$-a_1$};
    \filldraw (-.3,0) circle (1pt) node[anchor=north] {\tiny$-b$};
    \filldraw (0,0) circle (1pt) node[anchor=north] {\tiny$0$};
    \filldraw (.5,0) circle (1pt) node[anchor=north] {\tiny$c$};
    \filldraw (1.2,0) circle (1pt) node[anchor=north] {\tiny$d_1$};
    \filldraw (2,0) circle (1pt) node[anchor=north] {\tiny$d_2$};
    %VERTICAL LINES:
    \draw (-1.7,0) -- (-1.7,1.7);
    \draw (-.9,0) -- (-.9,.9);
    \draw (-.3,0) -- (-.3,.3);
    \draw (.5,0) -- (.5,.5);
    \draw (1.2,0) -- (1.2,1.2);
    \draw (2,0) -- (2,2);
    %MASSES:
    \filldraw [blue](-1.4,0) circle (.05pt) node[anchor=south] {\tiny $X_1$};
    \filldraw [blue](-.6,0) circle (.05pt) node[anchor=south] {\tiny$L_1$};
    \node at (-.15,.4) {\textcolor{blue}{\tiny$S$}};
    \node at (.25,.4) {\textcolor{blue}{\tiny$M$}};
    \filldraw [blue](.9,0) circle (.05pt) node[anchor=south] {\tiny$L_2$};
    \filldraw [blue](1.6,0) circle (.05pt) node[anchor=south] {\tiny$X_2$};
    \end{tikzpicture}
    
    \caption{Original configuration}
    \label{fig:5.3.2.1}
    \end{subfigure}
    \hfill
    \begin{subfigure}[h]{0.45\textwidth}
    \centering
    
    \begin{tikzpicture}
     %AXIS AND |x|:
    \draw[stealth-,very thick] (-2.2,2.2) -- (0,0);
    \draw[-stealth,very thick](0,0) -- (2.2,2.2);
    \draw[stealth-stealth,very thick] (-2.5,0) -- (2.5,0);
    %POINTS: (top->down = left->right)
    \filldraw [gray](-1.9,0) circle (1pt);
    \filldraw (-1.6,0) circle (1pt) node[anchor=north] {\tiny $-a'_2$};
    \filldraw [gray](-1.1,0) circle (1pt);
    \filldraw (-.8,0) circle (1pt) node[anchor=north] {\tiny$-a'_1$};
    \filldraw (-.3,0) circle (1pt) node[anchor=north] {\tiny$-b$};
    \filldraw (0,0) circle (1pt) node[anchor=north] {\tiny$0$};
    \filldraw (.5,0) circle (1pt) node[anchor=north] {\tiny$c$};
    \filldraw [gray](1,0) circle (1pt);   
    \filldraw (1.3,0) circle (1pt) node[anchor=north] {\tiny$d'_1$};
    \filldraw [gray](1.8,0) circle (1pt);
    \filldraw (2.1,0) circle (1pt) node[anchor=north] {\tiny$d'_2$};
    %VERTICAL LINES:
    \draw (-1.6,0) -- (-1.6,1.6);
    \draw (-.8,0) -- (-.8,.8);
    \draw (-.3,0) -- (-.3,.3);
    \draw (.5,0) -- (.5,.5);
    \draw (1.3,0) -- (1.3,1.3);
    \draw (2.1,0) -- (2.1,2.1);
    %MASSES:
    \filldraw [blue](-1.25,0) circle (.05pt) node[anchor=south] {\tiny $X_2$};
    \filldraw [blue](-.6,0) circle (.05pt) node[anchor=south] {\tiny$L_2$};
    \node at (-.15,.4) {\textcolor{blue}{\tiny$S$}};
    \node at (.25,.4) {\textcolor{blue}{\tiny$M$}};
    \filldraw [blue](.9,0) circle (.05pt) node[anchor=south] {\tiny$L_1$};
    \filldraw [blue](1.6,0) circle (.05pt) node[anchor=south] {\tiny$X_1$};
    \end{tikzpicture}
    \caption{New configuration}
    \label{fig:5.3.2.2}
    \end{subfigure}
    \caption{Case 2: $X$'s and $L$'s have flipped position.}
    \label{fig:5.3.2.0}
\end{figure}

\begin{figure}[H]
    \centering
    \begin{tikzpicture}
    %LINES:
    %x axis and |x|:
    \draw[stealth-,very thick] (-4,4) -- (0,0);
    \draw[-stealth,very thick](0,0) -- (4,4);
    \draw[stealth-stealth,very thick] (-4.5,0) -- (4.5,0); 
    %mass dividers: (LH -> RH)
    \draw [dotted,thick] (-3.8,0) -- (-3.8,3.8);
    \draw (-3.1,0) -- (-3.1,3.1);
    \draw [dotted,thick] (-2.3,0) -- (-2.3,2.3);
    \draw (-1.5,0) -- (-1.5,1.5);
    \draw (-1,0) -- (-1,1);
    \draw (1,0) -- (1,1);
    \draw [dotted,thick] (1.6,0) -- (1.6,1.6);
    \draw (2,0) -- (2,2);
    \draw [dotted,thick] (2.6,0) -- (2.6,2.6);
    \draw (3.2,0) -- (3.2,3.2);
    %POINTS: 
    \filldraw (-3.8,0) circle (1pt) node[anchor=north] {\scriptsize$-a_2$};
    \filldraw (-3.1,0) circle (1.5pt) node[anchor=north] {\scriptsize $-a'_2$};
    \filldraw (-2.3,0) circle (1pt) node[anchor=north] {\scriptsize $-a_1$};
    \filldraw (-1.5,0) circle (1.5pt) node[anchor=north] {\scriptsize$-a'_1$};
    \filldraw (-1,0) circle (1.5pt) node[anchor=north] {\scriptsize $-b$};
    \filldraw (0,0) circle (1pt) node[anchor=north] {\scriptsize $0$};
    \filldraw (1,0) circle (1.5pt) node[anchor=north] {\scriptsize$c$};
    \filldraw (1.6,0) circle (1pt) node[anchor=north] {\scriptsize$d_1$};
    \filldraw (2,0) circle (1.5pt) node[anchor=north] {\scriptsize$d'_1$};
    \filldraw (2.6,0) circle (1pt) node[anchor=north] {\scriptsize$d_2$};
    \filldraw (3.2,0) circle (1.5pt) node[anchor=north] {\scriptsize$d'_2$};
    %LABELS:
    %BELOW: turning a node label 90 degrees to fit in the column
    \node[rotate=90] at (-3.4,1){\scriptsize$\updelta_X+\updelta_L$};
    \node[rotate=90] at (-2.75,1){\tiny{$R_3$}- \scriptsize{$\updelta_L$}};
    \filldraw (-1.2,0) circle (.05pt) node[anchor=south] {\tiny$L_2$};
    \filldraw (-1.9,0) circle (.05pt) node[anchor=south] {\scriptsize$\updelta_L$};
    \filldraw (-.6,0) circle (.05pt) node[anchor=south] {\tiny$S$};
    \filldraw (.5,0) circle (.05pt) node[anchor=south] {\tiny$M$};
    \filldraw (1.3,0) circle (.05pt) node[anchor=south] {\tiny$L_2$};
    \node[rotate=90] at (1.8,.8){\scriptsize$\updelta_L$};
    \node[rotate=90] at (2.3,1){\scriptsize$X_2-\updelta_L$};
    \node[rotate=90] at (2.85,1) {\scriptsize$\updelta_X+\updelta_L$};
    \end{tikzpicture}
    \caption{Case 2: all endpoints and masses identified.}
    \label{fig:5.3.2.3}
\end{figure}
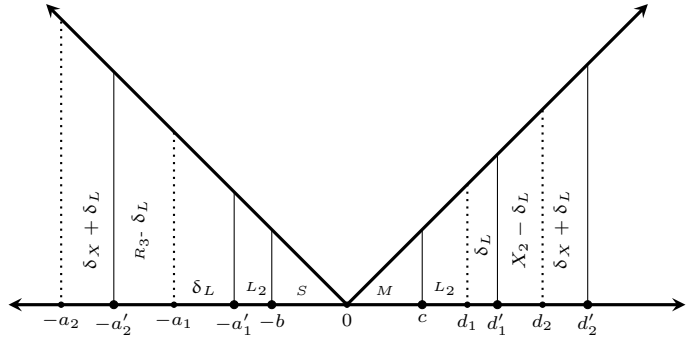
        
        Specifically, we can define $\delta_L = L_1 - L_2$ and $\delta_X = X_1 - X_2$. Note that the original configuration has perimeter $a_2 + a_1 + b + c + d_1 + d_2$, while the modified configuration has perimeter $a_2' + a_1' + b + c + d_1' + d_2'$. The second figure shows all endpoints and masses calculated. We can see that, because $S \leq M$, we get that $a_1' \leq d1$ and therefore the interval $[-a_1, -a_1']$ is no narrower than the interval $[d_1, d_1']$. In other words, $d_1' - d_1 \leq a_1 - a_1'$, and so $a_1' + d_1'\leq a_1 + d_1$. A similar argument tells us that $a_2' + d_2'\leq a_2 + d_2$, and we can conclude that this reconfiguration did not increase the total perimeter. This returns us to a framework that is identical to Case 0.
        
        \item Case 3: $S \leq M \leq L_2 \leq L_1 \leq X_1 \leq X_2$. In this case, we can switch the roles of the $L$'s, leaving the $X$'s in their current position.The process is shown in the picture below.
% DOUBLE GRAPH FIGURE      
    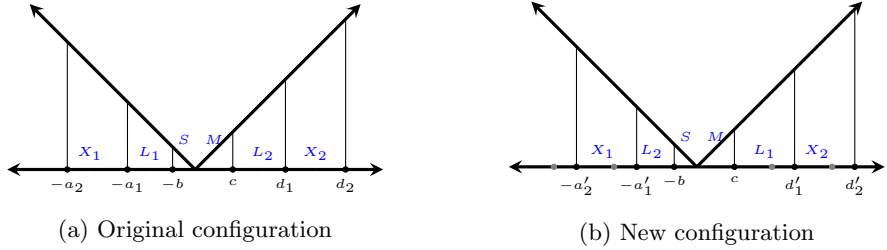
\begin{figure}[H]
    \centering
    \begin{subfigure}[h]{0.45\textwidth} %scaling the sub-figure
    \centering
    \begin{tikzpicture}
         
    %AXIS AND |x|:
    \draw[stealth-,very thick] (-2.2,2.2) -- (0,0);
    \draw[-stealth,very thick](0,0) -- (2.2,2.2);
    \draw[stealth-stealth,very thick] (-2.5,0) -- (2.5,0);
    
    %POINTS: (top->down = left->right)
    \filldraw (-1.7,0) circle (1pt) node[anchor=north] {\tiny $-a_2$};
    \filldraw (-.9,0) circle (1pt) node[anchor=north] {\tiny$-a_1$};
    \filldraw (-.3,0) circle (1pt) node[anchor=north] {\tiny$-b$};
    \filldraw (.5,0) circle (1pt) node[anchor=north] {\tiny$c$};
    \filldraw (1.2,0) circle (1pt) node[anchor=north] {\tiny$d_1$};
    \filldraw (2,0) circle (1pt) node[anchor=north] {\tiny$d_2$};
    
    %VERTICAL LINES:
    \draw (-1.7,0) -- (-1.7,1.7);
    \draw (-.9,0) -- (-.9,.9);
    \draw (-.3,0) -- (-.3,.3);
    \draw (.5,0) -- (.5,.5);
    \draw (1.2,0) -- (1.2,1.2);
    \draw (2,0) -- (2,2);
    
    %MASSES:
    \filldraw [blue](-1.4,0) circle (.05pt) node[anchor=south] {\tiny $X_1$};
    \filldraw [blue](-.6,0) circle (.05pt) node[anchor=south] {\tiny$L_1$};
    \node at (-.15,.4) {\textcolor{blue}{\tiny$S$}};
    \node at (.25,.4) {\textcolor{blue}{\tiny$M$}};
    \filldraw [blue](.9,0) circle (.05pt) node[anchor=south] {\tiny$L_2$};
    \filldraw [blue](1.6,0) circle (.05pt) node[anchor=south] {\tiny$X_2$};
    
    \end{tikzpicture}
    \caption{Original configuration}
    \end{subfigure}
    \hfill
    \begin{subfigure}[h]{0.45\textwidth}
    \centering
    \begin{tikzpicture}
    
     %AXIS AND |x|:
    \draw[stealth-,very thick] (-2.2,2.2) -- (0,0);
    \draw[-stealth,very thick](0,0) -- (2.2,2.2);
    \draw[stealth-stealth,very thick] (-2.5,0) -- (2.5,0);
    
    %POINTS: (top->down = left->right)
    \filldraw [gray](-1.9,0) circle (1pt);
    \filldraw (-1.6,0) circle (1pt) node[anchor=north] {\tiny $-a'_2$};
    \filldraw [gray](-1.1,0) circle (1pt);
    \filldraw (-.8,0) circle (1pt) node[anchor=north] {\tiny$-a'_1$};
    \filldraw (-.3,0) circle (1pt) node[anchor=north] {\tiny$-b$};
    \filldraw (.5,0) circle (1pt) node[anchor=north] {\tiny$c$};
    \filldraw [gray](1,0) circle (1pt);   
    \filldraw (1.3,0) circle (1pt) node[anchor=north] {\tiny$d'_1$};
    \filldraw [gray](1.8,0) circle (1pt);
    \filldraw (2.1,0) circle (1pt) node[anchor=north] {\tiny$d'_2$};
    
    %VERTICAL LINES:
    \draw (-1.6,0) -- (-1.6,1.6);
    \draw (-.8,0) -- (-.8,.8);
    \draw (-.3,0) -- (-.3,.3);
    \draw (.5,0) -- (.5,.5);
    \draw (1.3,0) -- (1.3,1.3);
    \draw (2.1,0) -- (2.1,2.1);
    
    %MASSES:
    \filldraw [blue](-1.25,0) circle (.05pt) node[anchor=south] {\tiny $X_1$};
    \filldraw [blue](-.6,0) circle (.05pt) node[anchor=south] {\tiny$L_2$};
    \node at (-.15,.4) {\textcolor{blue}{\tiny$S$}};
    \node at (.25,.4) {\textcolor{blue}{\tiny$M$}};
    \filldraw [blue](.9,0) circle (.05pt) node[anchor=south] {\tiny$L_1$};
    \filldraw [blue](1.6,0) circle (.05pt) node[anchor=south] {\tiny$X_2$};
    
    \end{tikzpicture}
    \caption{New configuration}
    \label{fig:5.3.3.2}
    \end{subfigure}
    \caption{Case 3: Flip the $L$'s, leaving the $X$'s where they are.}
    \label{fig:5.3.3.0}
    \end{figure}
    
%TOTAL MASSES FIGURE  
    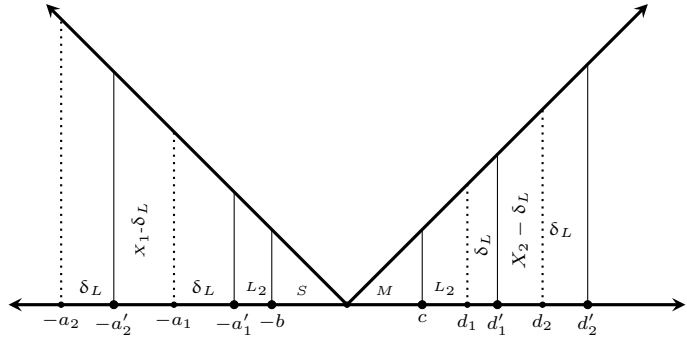
\begin{figure}[H]
    \centering
    \begin{tikzpicture}
    
    %LINES (x axis and |x|):
    \draw[stealth-,very thick] (-4,4) -- (0,0);
    \draw[-stealth,very thick](0,0) -- (4,4);
    \draw[stealth-stealth,very thick] (-4.5,0) -- (4.5,0); 
    
    %MASS DIVIDERS: (LH -> RH)
    \draw [dotted,thick] (-3.8,0) -- (-3.8,3.8);
    \draw (-3.1,0) -- (-3.1,3.1);
    \draw [dotted,thick] (-2.3,0) -- (-2.3,2.3);
    \draw (-1.5,0) -- (-1.5,1.5);
    \draw (-1,0) -- (-1,1);
    \draw (1,0) -- (1,1);
    \draw [dotted,thick] (1.6,0) -- (1.6,1.6);
    \draw (2,0) -- (2,2);
    \draw [dotted,thick] (2.6,0) -- (2.6,2.6);
    \draw (3.2,0) -- (3.2,3.2);
    
    %POINTS: (LH -> RH)
    \filldraw (-3.8,0) circle (1pt) node[anchor=north] {\scriptsize$-a_2$};
    \filldraw (-3.1,0) circle (1.5pt) node[anchor=north] {\scriptsize $-a'_2$};
    \filldraw (-2.3,0) circle (1pt) node[anchor=north] {\scriptsize $-a_1$};
    \filldraw (-1.5,0) circle (1.5pt) node[anchor=north] {\scriptsize$-a'_1$};
    \filldraw (-1,0) circle (1.5pt) node[anchor=north] {\scriptsize $-b$};
    \filldraw (0,0) circle (1pt);
    \filldraw (1,0) circle (1.5pt) node[anchor=north] {\scriptsize$c$};
    \filldraw (1.6,0) circle (1pt) node[anchor=north] {\scriptsize$d_1$};
    \filldraw (2,0) circle (1.5pt) node[anchor=north] {\scriptsize$d'_1$};
    \filldraw (2.6,0) circle (1pt) node[anchor=north] {\scriptsize$d_2$};
    \filldraw (3.2,0) circle (1.5pt) node[anchor=north] {\scriptsize$d'_2$};
    
    %LABELS: (LH -> RH)
    %(BELOW: turning a node label 90 degrees to fit in the column)
    \filldraw (-3.4,0) circle (.05pt) node[anchor=south] {\scriptsize$\updelta_L$};
    \node[rotate=90] at (-2.75,1){\tiny{$X_1$}-\scriptsize{$\updelta_L$}};
    \filldraw (-1.9,0) circle (.05pt) node[anchor=south] {\scriptsize$\updelta_L$};
    \filldraw (-1.2,0) circle (.05pt) node[anchor=south] {\tiny$L_2$};
    \filldraw (-.6,0) circle (.05pt) node[anchor=south] {\tiny$S$};
    \filldraw (.5,0) circle (.05pt) node[anchor=south] {\tiny$M$};
    \filldraw (1.3,0) circle (.05pt) node[anchor=south] {\tiny$L_2$};
    \node[rotate=90] at (1.8,.8){\scriptsize$\updelta_L$};
    \node[rotate=90] at (2.3,1){\scriptsize$X_2-\updelta_L$};
    \node[rotate=0] at (2.85,1) {\scriptsize$\updelta_L$}; %didn't do \filldraw here and set rotate to 0 just because it was easier to edit
    
    \end{tikzpicture}
    \caption{Case 3: All endpoints and masses identified.}
    \label{fig:5.3.3.3}
    \end{figure}
        
        Specifically, we can define $\delta_L = L_1 - L_2$. The initial configuration has perimeter $a_2 + a_1 + b + c + d_1 + d_2$, while the modified configuration has perimeter $a_2' + a_1' + b + c + d_1' + d_2'$. In the figure above, we see all endpoints and masses calculated. Note that similar reasoning to that used in Case 2 shows us that $d_1' - d_1 \leq a_1 - a_1'$, and also that $d_2' - d_2 \leq a_2 - a_2'$. Taken together, this gives us that the new modified configuration will never have more perimeter than the old configuration. Applying this rearrangement returns us to a framework that is identical to Case 0.
        
        \item Case 4: $S \leq M \leq L_1 \leq X_1 \leq L_2 \leq  X_2$. In this case we can transpose the $L_2$ and the $X_1$ regions, leaving the others in their current spots. The process is shown in the picture below.
        % DOUBLE GRAPH FIGURE
    \begin{figure}[H]
    \centering
    \begin{subfigure}[h]{0.45\textwidth}
         \centering
         \begin{tikzpicture}
    %AXIS AND |x|:
    \draw[stealth-,very thick] (-2.2,2.2) -- (0,0);
    \draw[-stealth,very thick](0,0) -- (2.2,2.2);
    \draw[stealth-stealth,very thick] (-2.5,0) -- (2.5,0);
    %POINTS: (top->down = left->right)
    \filldraw (-1.7,0) circle (1pt) node[anchor=north] {\tiny $-a_2$};
    \filldraw (-.9,0) circle (1pt) node[anchor=north] {\tiny$-a_1$};
    \filldraw (-.3,0) circle (1pt) node[anchor=north] {\tiny$-b$};
    \filldraw (0,0) circle (1pt) node[anchor=north] {\tiny$0$};
    \filldraw (.5,0) circle (1pt) node[anchor=north] {\tiny$c$};
    \filldraw (1.2,0) circle (1pt) node[anchor=north] {\tiny$d_1$};
    \filldraw (2,0) circle (1pt) node[anchor=north] {\tiny$d_2$};
    %VERTICAL LINES:
    \draw (-1.7,0) -- (-1.7,1.7);
    \draw (-.9,0) -- (-.9,.9);
    \draw (-.3,0) -- (-.3,.3);
    \draw (.5,0) -- (.5,.5);
    \draw (1.2,0) -- (1.2,1.2);
    \draw (2,0) -- (2,2);
    %MASSES:
    \filldraw [blue](-1.4,0) circle (.05pt) node[anchor=south] {\tiny $X_1$};
    \filldraw [blue](-.6,0) circle (.05pt) node[anchor=south] {\tiny$L_1$};
    \node at (-.15,.4) {\textcolor{blue}{\tiny$S$}};
    \node at (.25,.4) {\textcolor{blue}{\tiny$M$}};
    \filldraw [blue](.9,0) circle (.05pt) node[anchor=south] {\tiny$L_2$};
    \filldraw [blue](1.6,0) circle (.05pt) node[anchor=south] {\tiny$X_2$};
    \end{tikzpicture}
    
    \caption{Original configuration}
    \label{fig:5.3.4.1}
    \end{subfigure}
    \hfill
    \begin{subfigure}[h]{0.45\textwidth}
    \centering
    
    \begin{tikzpicture}
     %AXIS AND |x|:
    \draw[stealth-,very thick] (-2.2,2.2) -- (0,0);
    \draw[-stealth,very thick](0,0) -- (2.2,2.2);
    \draw[stealth-stealth,very thick] (-2.5,0) -- (2.5,0);
    %POINTS: (top->down = left->right)
    \filldraw (-1.6,0) circle (1pt) node[anchor=north] {\tiny $-a'_2$};
    \filldraw [gray](-1.4,0) circle (1pt);
    \filldraw (-.8,0) circle (1pt) node[anchor=north] {\tiny$-a'_1$};
    \filldraw (-.3,0) circle (1pt) node[anchor=north] {\tiny$-b$};
    \filldraw (0,0) circle (1pt) node[anchor=north] {\tiny$0$};
    \filldraw (.5,0) circle (1pt) node[anchor=north] {\tiny$c$};
    \filldraw [gray](1.45,0) circle (1pt);   
    \filldraw (1.3,0) circle (1pt) node[anchor=north] {\tiny$d'_1$};
    \filldraw [gray](2.25,0) circle (1pt);
    \filldraw (2.1,0) circle (1pt) node[anchor=north] {\tiny$d'_2$};
    %VERTICAL LINES:
    \draw (-1.6,0) -- (-1.6,1.6);
    \draw (-.8,0) -- (-.8,.8);
    \draw (-.3,0) -- (-.3,.3);
    \draw (.5,0) -- (.5,.5);
    \draw (1.3,0) -- (1.3,1.3);
    \draw (2.1,0) -- (2.1,2.1);
    %MASSES:
    \filldraw [blue](-1.25,0) circle (.05pt) node[anchor=south] {\tiny $L_2$};
    \filldraw [blue](-.6,0) circle (.05pt) node[anchor=south] {\tiny$L_1$};
    \node at (-.15,.4) {\textcolor{blue}{\tiny$S$}};
    \node at (.25,.4) {\textcolor{blue}{\tiny$M$}};
    \filldraw [blue](.9,0) circle (.05pt) node[anchor=south] {\tiny$X_1$};
    \filldraw [blue](1.6,0) circle (.05pt) node[anchor=south] {\tiny$X_2$};
    \end{tikzpicture}
    \caption{New configuration}
    \label{fig:5.3.4.2}
    \end{subfigure}
    \caption{Case 4: transpose the locations of $X_1$ and $L_2$.}
    \label{fig:5.3.4.0}
\end{figure}
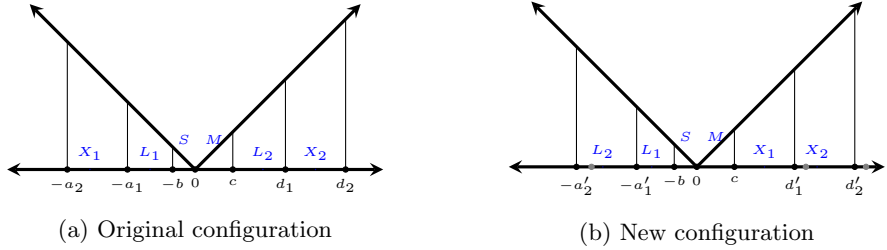
  
 % \begin{center}
 %           \includegraphics[scale=.5]{MainPaper/IntervalFlippingCase4.png}
 %       \end{center}
        
\begin{figure}[H]
    \centering
    \begin{tikzpicture}
    %LINES:
    %x axis and |x|:
    \draw[stealth-,very thick] (-3.5,3.5) -- (0,0);
    \draw[-stealth,very thick](0,0) -- (3.5,3.5);
    \draw[stealth-stealth,very thick] (-4.5,0) -- (4.5,0); 
    %mass dividers: (LH -> RH)
    \draw (-2.8,0) -- (-2.8,2.8);
    \draw [dotted,thick] (-2.3,0) -- (-2.3,2.3);
    \draw (-1.5,0) -- (-1.5,1.5);
    \draw (-1,0) -- (-1,1);
    \draw (1,0) -- (1,1);
    \draw (1.6,0) -- (1.6,1.6);
    \draw [dotted,thick] (2,0) -- (2,2);
    \draw (2.6,0) -- (2.6,2.6);
    \draw [dotted,thick](3.2,0) -- (3.2,3.2);
    %POINTS: 
    \filldraw (-2.8,0) circle (1.5pt) node[anchor=north] {\scriptsize $-a'_2$};
    \filldraw (-2.3,0) circle (1pt) node[anchor=north] {\scriptsize $-a_2$};
    \filldraw (-1.5,0) circle (1.5pt) node[anchor=north] {\scriptsize$-a_1$};
    \filldraw (-1,0) circle (1.5pt) node[anchor=north] {\scriptsize $-b$};
    \filldraw (0,0) circle (1pt) node[anchor=north] {\scriptsize $0$};
    \filldraw (1,0) circle (1.5pt) node[anchor=north] {\scriptsize$c$};
    \filldraw (1.6,0) circle (1pt) node[anchor=north] {\scriptsize$d'_1$};
    \filldraw (2,0) circle (1.5pt) node[anchor=north] {\scriptsize$d_1$};
    \filldraw (2.6,0) circle (1pt) node[anchor=north] {\scriptsize$d'_2$};
    \filldraw (3.2,0) circle (1.5pt) node[anchor=north] {\scriptsize$d_2$};
    %LABELS:
    %BELOW: turning a node label 90 degrees to fit in the column
    \filldraw (-2.55,0) circle (.05pt) node[anchor=south] {\scriptsize$\updelta$};
    \filldraw (-1.9,0) circle (.05pt) node[anchor=south] {\tiny$X_1$};
    \filldraw (-1.2,0) circle (.05pt) node[anchor=south] {\tiny$L_1$};
    \filldraw (-.6,0) circle (.05pt) node[anchor=south] {\tiny$S$};
    \filldraw (.5,0) circle (.05pt) node[anchor=south] {\tiny$M$};
    \filldraw (1.3,0) circle (.05pt) node[anchor=south] {\tiny$X_1$};
    \filldraw (1.8,0) circle (.05pt) node[anchor=south] {\scriptsize$\updelta$};
    \node[rotate=90] at (2.3,1){\scriptsize$X_2-\updelta$};
    \node[rotate=0] at (2.85,1) {\scriptsize$\updelta$};
    \end{tikzpicture}
    \caption{Case 4: all endpoints and masses identified.}
    \label{fig:5.3.4.3}
\end{figure}
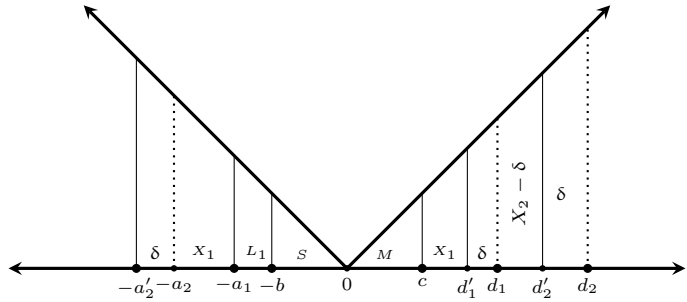

        Specifically, we can define $\delta = L_2 - X_1$. Note that $X_1 + L_1 + S \geq X_1 + M$, and therefore the interval $[-a_2', -a_2]$ cannot be wider than the interval $[d_1', d_1]$. This tells us that $a_2' - a_2 \leq d_1 - d_1'$, which implies that $a_2' + d_1' \leq a_2 + d_1$. Similarly, we can immediately see that $d_2' \leq d_2$. Taken together, this means that the perimeter of the new configuration is necessarily not greater than the perimeter of the original configuration. Applying this rearrangement returns us to a framework that is identical to Case 0.
        
        \item Case 5: $S \leq M \leq L_2 \leq X_2 \leq L_1 \leq X_1$. In this case, we can transpose the positions of $L_1$ and $L_2$, and simultaneously transpose the positions of $X_1$ and $X_2$. The result is depicted in the picture below.
         % DOUBLE GRAPH FIGURE
    \begin{figure}[H]
    \centering
    \begin{subfigure}[h]{0.45\textwidth}
         \centering
         \begin{tikzpicture}
    %AXIS AND |x|:
    \draw[stealth-,very thick] (-2.2,2.2) -- (0,0);
    \draw[-stealth,very thick](0,0) -- (2.2,2.2);
    \draw[stealth-stealth,very thick] (-2.5,0) -- (2.5,0);
    %POINTS: (top->down = left->right)
    \filldraw (-1.7,0) circle (1pt) node[anchor=north] {\tiny $-a_2$};
    \filldraw (-.9,0) circle (1pt) node[anchor=north] {\tiny$-a_1$};
    \filldraw (-.3,0) circle (1pt) node[anchor=north] {\tiny$-b$};
    \filldraw (0,0) circle (1pt) node[anchor=north] {\tiny$0$};
    \filldraw (.5,0) circle (1pt) node[anchor=north] {\tiny$c$};
    \filldraw (1.2,0) circle (1pt) node[anchor=north] {\tiny$d_1$};
    \filldraw (2,0) circle (1pt) node[anchor=north] {\tiny$d_2$};
    %VERTICAL LINES:
    \draw (-1.7,0) -- (-1.7,1.7);
    \draw (-.9,0) -- (-.9,.9);
    \draw (-.3,0) -- (-.3,.3);
    \draw (.5,0) -- (.5,.5);
    \draw (1.2,0) -- (1.2,1.2);
    \draw (2,0) -- (2,2);
    %MASSES:
    \filldraw [blue](-1.4,0) circle (.05pt) node[anchor=south] {\tiny $X_1$};
    \filldraw [blue](-.6,0) circle (.05pt) node[anchor=south] {\tiny$L_1$};
    \node at (-.15,.4) {\textcolor{blue}{\tiny$S$}};
    \node at (.25,.4) {\textcolor{blue}{\tiny$M$}};
    \filldraw [blue](.9,0) circle (.05pt) node[anchor=south] {\tiny$L_2$};
    \filldraw [blue](1.6,0) circle (.05pt) node[anchor=south] {\tiny$X_2$};
    \end{tikzpicture}
    
    \caption{Initial configuration}
    \label{fig:5.3.5.1}
    \end{subfigure}
    \hfill
    \begin{subfigure}[h]{0.45\textwidth}
    \centering
    
    \begin{tikzpicture}
     %AXIS AND |x|:
    \draw[stealth-,very thick] (-2.2,2.2) -- (0,0);
    \draw[-stealth,very thick](0,0) -- (2.2,2.2);
    \draw[stealth-stealth,very thick] (-2.5,0) -- (2.5,0);
    %POINTS: (top->down = left->right)
    \filldraw [gray](-1.6,0) circle (1pt) node[anchor=north] {\tiny $-a'_2$};
    \filldraw (-1.4,0) circle (1pt);
    \filldraw [gray](-1,0) circle (1pt);
    \filldraw (-.8,0) circle (1pt) node[anchor=north] {\tiny$-a'_1$};
    \filldraw (-.3,0) circle (1pt) node[anchor=north] {\tiny$-b$};
    \filldraw (0,0) circle (1pt) node[anchor=north] {\tiny$0$};
    \filldraw (.5,0) circle (1pt) node[anchor=north] {\tiny$c$};
    \filldraw [gray](1.1,0) circle (1pt);   
    \filldraw (1.3,0) circle (1pt) node[anchor=north] {\tiny$d'_1$};
    \filldraw [gray](1.9,0) circle (1pt);
    \filldraw (2.1,0) circle (1pt) node[anchor=north] {\tiny$d'_2$};
    %VERTICAL LINES:
    \draw (-1.4,0) -- (-1.4,1.4);
    \draw (-.8,0) -- (-.8,.8);
    \draw (-.3,0) -- (-.3,.3);
    \draw (.5,0) -- (.5,.5);
    \draw (1.3,0) -- (1.3,1.3);
    \draw (2.1,0) -- (2.1,2.1);
    %MASSES:
    \filldraw [blue](-1.1,0) circle (.05pt) node[anchor=south] {\tiny $X_2$};
    \filldraw [blue](-.6,0) circle (.05pt) node[anchor=south] {\tiny$L_2$};
    \node at (-.15,.4) {\textcolor{blue}{\tiny$S$}};
    \node at (.25,.4) {\textcolor{blue}{\tiny$M$}};
    \filldraw [blue](.9,0) circle (.05pt) node[anchor=south] {\tiny$L_1$};
    \filldraw [blue](1.6,0) circle (.05pt) node[anchor=south] {\tiny$X_1$};
    \end{tikzpicture}
    \caption{New configuration}
    \label{fig:5.3.5.2}
    \end{subfigure}
    \caption{Case 5: transpose the locations of the $L_i$ and the $X_i$ simultaneously.}
    \label{fig:5.3.5.0}
\end{figure}
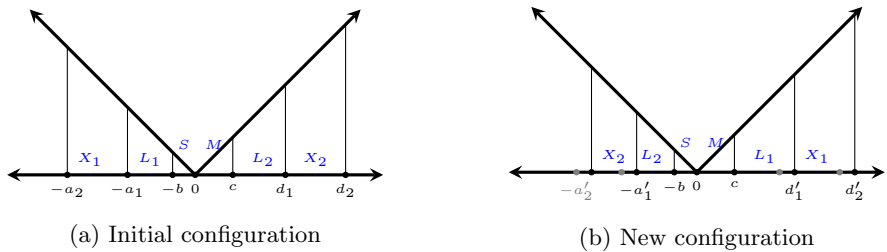

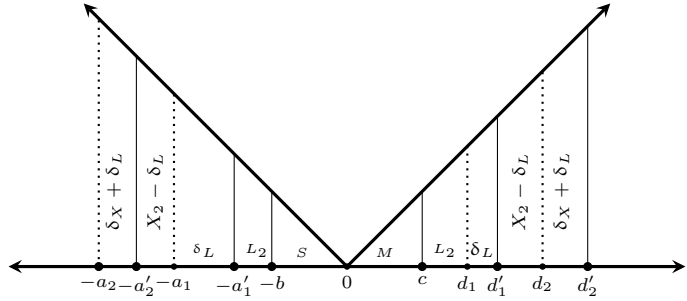
\begin{figure}[H]
    \centering
    \begin{tikzpicture}
    %LINES:
    %x axis and |x|:
    \draw[stealth-,very thick] (-3.5,3.5) -- (0,0);
    \draw[-stealth,very thick](0,0) -- (3.5,3.5);
    \draw[stealth-stealth,very thick] (-4.5,0) -- (4.5,0); 
    
    %mass dividers: (LH -> RH)
    \draw [dotted,thick](-3.3,0) -- (-3.3, 3.3);
    \draw (-2.8,0) -- (-2.8,2.8);
    \draw [dotted,thick] (-2.3,0) -- (-2.3,2.3);
    \draw (-1.5,0) -- (-1.5,1.5);
    \draw (-1,0) -- (-1,1);
    \draw (1,0) -- (1,1);
    \draw [dotted,thick](1.6,0) -- (1.6,1.6);
    \draw  (2,0) -- (2,2);
    \draw [dotted,thick](2.6,0) -- (2.6,2.6);
    \draw (3.2,0) -- (3.2,3.2);
    
    %POINTS: 
    \filldraw (-3.3,0) circle (1.5pt) node[anchor=north] {\scriptsize $-a_2$};
    \filldraw (-2.8,0) circle (1.5pt) node[anchor=north] {\scriptsize $-a'_2$};
    \filldraw (-2.3,0) circle (1pt) node[anchor=north] {\scriptsize $-a_1$};
    \filldraw (-1.5,0) circle (1.5pt) node[anchor=north] {\scriptsize$-a'_1$};
    \filldraw (-1,0) circle (1.5pt) node[anchor=north] {\scriptsize $-b$};
    \filldraw (0,0) circle (1pt) node[anchor=north] {\scriptsize $0$};
    \filldraw (1,0) circle (1.5pt) node[anchor=north] {\scriptsize$c$};
    \filldraw (1.6,0) circle (1pt) node[anchor=north] {\scriptsize$d_1$};
    \filldraw (2,0) circle (1.5pt) node[anchor=north] {\scriptsize$d'_1$};
    \filldraw (2.6,0) circle (1pt) node[anchor=north] {\scriptsize$d_2$};
    \filldraw (3.2,0) circle (1.5pt) node[anchor=north] {\scriptsize$d'_2$};
    %LABELS:
    %BELOW: turning a node label 90 degrees to fit in the column
    \node[rotate=90] at (-3.1,1){\scriptsize$\updelta_X + \updelta_L$};
    \node[rotate=90] at (-2.55,1){\scriptsize$X_2-\updelta_L$};
    %\filldraw (-2.55,0) circle (.05pt) node[anchor=south] {\scriptsize$X_2 - \updelta_L$};
    \filldraw (-1.9,0) circle (.05pt) node[anchor=south] {\tiny$\updelta_L$};
    \filldraw (-1.2,0) circle (.05pt) node[anchor=south] {\tiny$L_2$};
    \filldraw (-.6,0) circle (.05pt) node[anchor=south] {\tiny$S$};
    \filldraw (.5,0) circle (.05pt) node[anchor=south] {\tiny$M$};
    \filldraw (1.3,0) circle (.05pt) node[anchor=south] {\tiny$L_2$};
    \filldraw (1.8,0) circle (.05pt) node[anchor=south] {\scriptsize$\updelta_L$};
    \node[rotate=90] at (2.3,1){\scriptsize$X_2-\updelta_L$};
    \node[rotate=90] at (2.85,1) {\scriptsize$\updelta_X + \updelta_L$};
    \end{tikzpicture}
    \caption{Case 5: all endpoints and masses identified.}
    \label{fig:5.3.5.3}
\end{figure}        
        By setting $\updelta_X = X_1 - X_2$ and $\updelta_L = L_1 - L_2$ and reasoning as in previous cases, we can observe that this move will lower overall perimeter. The rearrangment will leave us in a framework that is identical to Case 4.
    \end{itemize}

Next, we must consider what happens if the second-smallest mass is on the same side of the origin as the smallest mass. We claim that this is not optimal, and that we can rearrange our masses to move the second-smallest mass across the perimeter, lowering perimeter in the process. To begin, let's rename our regions as follows:

\begin{figure}[H]
    \centering
    \begin{tikzpicture}
    %LINES:
    %x axis and |x|:
    \draw[stealth-,very thick] (-3.5,3.5) -- (0,0);
    \draw[-stealth,very thick](0,0) -- (3.5,3.5);
    \draw[stealth-stealth,very thick] (-4.5,0) -- (4.5,0); 
    
    %mass dividers: (LH -> RH)
    %\draw [dotted,thick](-3.3,0) -- (-3.3, 3.3);
    \draw (-2.8,0) -- (-2.8,2.8);
    %\draw [dotted,thick] (-2.3,0) -- (-2.3,2.3);
    \draw (-1.5,0) -- (-1.5,1.5);
    \draw (-1,0) -- (-1,1);
    \draw (1,0) -- (1,1);
    %\draw [dotted,thick](1.6,0) -- (1.6,1.6);
    \draw  (2,0) -- (2,2);
    %\draw [dotted,thick](2.6,0) -- (2.6,2.6);
    \draw (3.2,0) -- (3.2,3.2);
    
    %POINTS: 
    \filldraw (-2.8,0) circle (1.5pt) node[anchor=north] {\scriptsize $-a$};
    \filldraw (-1.5,0) circle (1.5pt) node[anchor=north] {\scriptsize$-b$};
    \filldraw (-1,0) circle (1.5pt) node[anchor=north] {\scriptsize $-c$};
    \filldraw (0,0) circle (1pt) node[anchor=north] {\scriptsize $0$};
    \filldraw (1,0) circle (1.5pt) node[anchor=north] {\scriptsize$d$};
    \filldraw (2,0) circle (1.5pt) node[anchor=north] {\scriptsize$e$};
    \filldraw (3.2,0) circle (1.5pt) node[anchor=north] {\scriptsize$f$};
    
    %LABELS:
    %BELOW: turning a node label 90 degrees to fit in the column
    %\node[rotate=90] at (-3.1,1){\scriptsize$\updelta_X + \updelta_L$};
    %\node[rotate=90] at (-2.55,1){\scriptsize$X_2-\updelta_L$};
    %\filldraw (-2.55,0) circle (.05pt) node[anchor=south] {\scriptsize$X_2 - \updelta_L$};
    \filldraw (-2.1,0) circle (.05pt) node[anchor=south] {\tiny$L$};
    \filldraw (-1.2,0) circle (.05pt) node[anchor=south] {\tiny$M$};
    \filldraw (-.6,0) circle (.05pt) node[anchor=south] {\tiny$S$};
    \filldraw (.5,0) circle (.05pt) node[anchor=south] {\tiny$R_1$};
    \filldraw (1.5,0) circle (.05pt) node[anchor=south] {\tiny$R_2$};
    \filldraw (2.4,0) circle (.05pt) node[anchor=south] {\tiny$R_3$};
    %\node[rotate=90] at (2.3,1){\scriptsize$X_2-\updelta_L$};
    %\node[rotate=90] at (2.85,1) {\scriptsize$\updelta_X + \updelta_L$};
    \end{tikzpicture}
    \caption{New cases, in which the second smallest mass ($M$) is on the same side of the origin as the smallest mass ($S$).}
    \label{fig:6}
\end{figure}
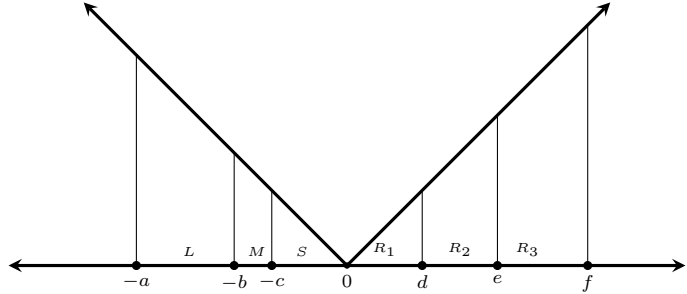    

We already know that $S \leq M \leq L$ and that $R_1 \leq R_2 \leq R_3$, and that $M < R_1$. The only question that remains is what the size of $L$ is relative to the $R_i$. There are four cases:

\begin{itemize}
    \item Case 6: $L \leq R_1 \leq R_2 \leq R_3$. In this case we can transpose $L$ and $R_1$, leaving the others in their current spots. The process is shown in the picture below;
    
    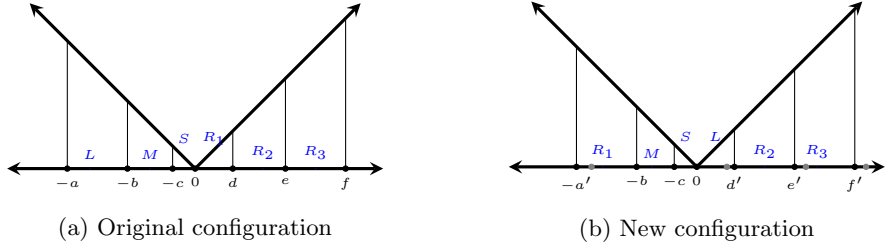
\begin{figure}[H]
    \centering
    \begin{subfigure}[h]{0.45\textwidth}
         \centering
         \begin{tikzpicture}
    %AXIS AND |x|:
    \draw[stealth-,very thick] (-2.2,2.2) -- (0,0);
    \draw[-stealth,very thick](0,0) -- (2.2,2.2);
    \draw[stealth-stealth,very thick] (-2.5,0) -- (2.5,0);
    %POINTS: (top->down = left->right)
    \filldraw (-1.7,0) circle (1pt) node[anchor=north] {\tiny $-a$};
    \filldraw (-.9,0) circle (1pt) node[anchor=north] {\tiny$-b$};
    \filldraw (-.3,0) circle (1pt) node[anchor=north] {\tiny$-c$};
    \filldraw (0,0) circle (1pt) node[anchor=north] {\tiny$0$};
    \filldraw (.5,0) circle (1pt) node[anchor=north] {\tiny$d$};
    \filldraw (1.2,0) circle (1pt) node[anchor=north] {\tiny$e$};
    \filldraw (2,0) circle (1pt) node[anchor=north] {\tiny$f$};
    %VERTICAL LINES:
    \draw (-1.7,0) -- (-1.7,1.7);
    \draw (-.9,0) -- (-.9,.9);
    \draw (-.3,0) -- (-.3,.3);
    \draw (.5,0) -- (.5,.5);
    \draw (1.2,0) -- (1.2,1.2);
    \draw (2,0) -- (2,2);
    %MASSES:
    \filldraw [blue](-1.4,0) circle (.05pt) node[anchor=south] {\tiny $L$};
    \filldraw [blue](-.6,0) circle (.05pt) node[anchor=south] {\tiny$M$};
    \node at (-.15,.4) {\textcolor{blue}{\tiny$S$}};
    \node at (.25,.4) {\textcolor{blue}{\tiny$R_1$}};
    \filldraw [blue](.9,0) circle (.05pt) node[anchor=south] {\tiny$R_2$};
    \filldraw [blue](1.6,0) circle (.05pt) node[anchor=south] {\tiny$R_3$};
    \end{tikzpicture}
    
    \caption{Original configuration}
    \label{fig:5.3.6.1}
    \end{subfigure}
    \hfill
    \begin{subfigure}[h]{0.45\textwidth}
    \centering
    
    \begin{tikzpicture}
     %AXIS AND |x|:
    \draw[stealth-,very thick] (-2.2,2.2) -- (0,0);
    \draw[-stealth,very thick](0,0) -- (2.2,2.2);
    \draw[stealth-stealth,very thick] (-2.5,0) -- (2.5,0);
    %POINTS: (top->down = left->right)
    \filldraw (-1.6,0) circle (1pt) node[anchor=north] {\tiny $-a'$};
    \filldraw [gray](-1.4,0) circle (1pt);
    \filldraw (-.8,0) circle (1pt) node[anchor=north] {\tiny$-b$};
    \filldraw (-.3,0) circle (1pt) node[anchor=north] {\tiny$-c$};
    \filldraw (0,0) circle (1pt) node[anchor=north] {\tiny$0$};
    \filldraw [gray](.4,0) circle (1pt);
    \filldraw (.5,0) circle (1pt) node[anchor=north] {\tiny$d'$};
    \filldraw [gray](1.45,0) circle (1pt);   
    \filldraw (1.3,0) circle (1pt) node[anchor=north] {\tiny$e'$};
    \filldraw [gray](2.25,0) circle (1pt);
    \filldraw (2.1,0) circle (1pt) node[anchor=north] {\tiny$f'$};
    %VERTICAL LINES:
    \draw (-1.6,0) -- (-1.6,1.6);
    \draw (-.8,0) -- (-.8,.8);
    \draw (-.3,0) -- (-.3,.3);
    \draw (.5,0) -- (.5,.5);
    \draw (1.3,0) -- (1.3,1.3);
    \draw (2.1,0) -- (2.1,2.1);
    %MASSES:
    \filldraw [blue](-1.25,0) circle (.05pt) node[anchor=south] {\tiny $R_1$};
    \filldraw [blue](-.6,0) circle (.05pt) node[anchor=south] {\tiny$M$};
    \node at (-.15,.4) {\textcolor{blue}{\tiny$S$}};
    \node at (.25,.4) {\textcolor{blue}{\tiny$L$}};
    \filldraw [blue](.9,0) circle (.05pt) node[anchor=south] {\tiny$R_2$};
    \filldraw [blue](1.6,0) circle (.05pt) node[anchor=south] {\tiny$R_3$};
    \end{tikzpicture}
    \caption{New configuration}
    \label{fig:5.3.6.2}
    \end{subfigure}
    \caption{Case 6: transpose the locations of $R_1$ and $L$.}
    \label{fig:5.3.6.0}
\end{figure}
      
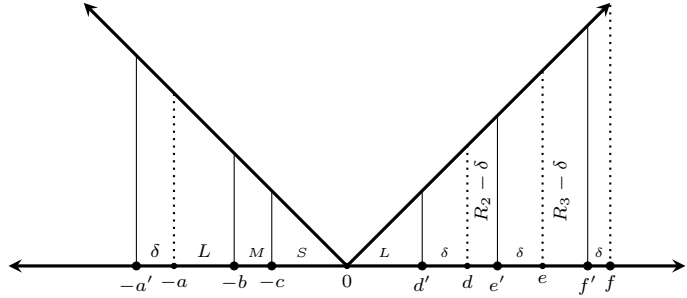
\begin{figure}[H]
    \centering
    \begin{tikzpicture}
    %LINES:
    %x axis and |x|:
    \draw[stealth-,very thick] (-3.5,3.5) -- (0,0);
    \draw[-stealth,very thick](0,0) -- (3.5,3.5);
    \draw[stealth-stealth,very thick] (-4.5,0) -- (4.5,0); 
    
    %mass dividers: (LH -> RH)
    %\draw [dotted,thick](-3.3,0) -- (-3.3, 3.3);
    \draw (-2.8,0) -- (-2.8,2.8);
    \draw [dotted,thick] (-2.3,0) -- (-2.3,2.3);
    \draw (-1.5,0) -- (-1.5,1.5);
    \draw (-1,0) -- (-1,1);
    \draw (1,0) -- (1,1);
    \draw [dotted,thick](1.6,0) -- (1.6,1.6);
    \draw  (2,0) -- (2,2);
    \draw [dotted,thick](2.6,0) -- (2.6,2.6);
    \draw (3.2,0) -- (3.2,3.2);
    \draw [dotted,thick](3.5,0) -- (3.5,3.5);
    %POINTS: 
    %\filldraw (-3.3,0) circle (1.5pt) node[anchor=north] {\scriptsize $-a'$};
    \filldraw (-2.8,0) circle (1.5pt) node[anchor=north] {\scriptsize $-a'$};
    \filldraw (-2.3,0) circle (1pt) node[anchor=north] {\scriptsize $-a$};
    \filldraw (-1.5,0) circle (1.5pt) node[anchor=north] {\scriptsize$-b$};
    \filldraw (-1,0) circle (1.5pt) node[anchor=north] {\scriptsize $-c$};
    \filldraw (0,0) circle (1pt) node[anchor=north] {\scriptsize $0$};
    \filldraw (1,0) circle (1.5pt) node[anchor=north] {\scriptsize$d'$};
    \filldraw (1.6,0) circle (1pt) node[anchor=north] {\scriptsize$d$};
    \filldraw (2,0) circle (1.5pt) node[anchor=north] {\scriptsize$e'$};
    \filldraw (2.6,0) circle (1pt) node[anchor=north] {\scriptsize$e$};
    \filldraw (3.2,0) circle (1.5pt) node[anchor=north] {\scriptsize$f'$};
    \filldraw (3.5,0) circle (1.5pt) node[anchor=north] {\scriptsize$f$};
    
    %LABELS:
    %BELOW: turning a node label 90 degrees to fit in the column
    %\node[rotate=90] at (-3.1,1){\scriptsize$\updelta_X + \updelta_L$};
    %\node[rotate=90] at (-2.55,1){\scriptsize$X_2-\updelta_L$};
    \filldraw (-2.55,0) circle (.05pt) node[anchor=south] {\scriptsize$\delta$};
    \filldraw (-1.9,0) circle (.05pt) node[anchor=south] {\scriptsize$L$};
    \filldraw (-1.2,0) circle (.05pt) node[anchor=south] {\tiny$M$};
    \filldraw (-.6,0) circle (.05pt) node[anchor=south] {\tiny$S$};
    \filldraw (.5,0) circle (.05pt) node[anchor=south] {\tiny$L$};
    \filldraw (1.3,0) circle (.05pt) node[anchor=south] {\tiny$\delta$};
    \node[rotate=90] at (1.8,1){\scriptsize$R_2-\delta$};
    %\filldraw (1.8,0) circle (.05pt) node[anchor=south] {\scriptsize$\updelta_L$};
    %\node[rotate=90] at (2.3,1){\scriptsize$X_2-\updelta_L$};
    \filldraw (2.3,0) circle (.05pt) node[anchor=south] {\tiny$\delta$};
    \node[rotate=90] at (2.85,1) {\scriptsize$R_3 - \delta$};
    \filldraw (3.35,0) circle (.05pt) node[anchor=south] {\tiny$\delta$};
    \end{tikzpicture}
    \caption{Case 6: all endpoints and masses identified.}
    \label{fig:5.3.6.3}
\end{figure}    
    
    %\begin{center}
    %    \includegraphics[scale=.4]{MainPaper/IntervalFlippingCase6.png}
    %\end{center}
    
    Specifically, we can let $\delta = R_1 - L$. The figures above show both the original configuration (with perimeter $a + b + c + d + e + f$) and the new configuration (with perimeter $a' + b + c + d' + e' + f'$). The final figure identifies all endpoints and corresponding masses. We can make the observation that, as endpoints, $a > d'$. This implies that the interval $[-a', -a]$ is narrower than the interval $[d', d]$ which contains the same mass. Therefore, we can conclude that $a' - a < d - d'$ and therefore that $a' + d' < a + d$. Additionally, we can see that $e' < e$ and $f' < f$. Taken together, these inequalities imply that after reconfiguration we have successfully lowered perimeter. 
    The rearrangment will leave us in a framework that is identical to Case 7.%After considering the new ordering of intervals, we can conclude that after reconfiguring we are in Case 7.
    
    \item Case 7: $ R_1 \leq L \leq R_2 \leq R_3$. In this case we can transpose $M$ and $R_1$ and, simultaneously, transpose $L$ and $R_2$. The process is shown in the picture below.
    
        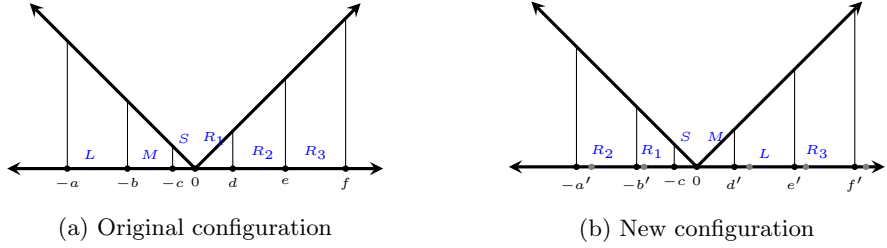
\begin{figure}[H]
    \centering
    \begin{subfigure}[h]{0.45\textwidth}
         \centering
         \begin{tikzpicture}
    %AXIS AND |x|:
    \draw[stealth-,very thick] (-2.2,2.2) -- (0,0);
    \draw[-stealth,very thick](0,0) -- (2.2,2.2);
    \draw[stealth-stealth,very thick] (-2.5,0) -- (2.5,0);
    %POINTS: (top->down = left->right)
    \filldraw (-1.7,0) circle (1pt) node[anchor=north] {\tiny $-a$};
    \filldraw (-.9,0) circle (1pt) node[anchor=north] {\tiny$-b$};
    \filldraw (-.3,0) circle (1pt) node[anchor=north] {\tiny$-c$};
    \filldraw (0,0) circle (1pt) node[anchor=north] {\tiny$0$};
    \filldraw (.5,0) circle (1pt) node[anchor=north] {\tiny$d$};
    \filldraw (1.2,0) circle (1pt) node[anchor=north] {\tiny$e$};
    \filldraw (2,0) circle (1pt) node[anchor=north] {\tiny$f$};
    %VERTICAL LINES:
    \draw (-1.7,0) -- (-1.7,1.7);
    \draw (-.9,0) -- (-.9,.9);
    \draw (-.3,0) -- (-.3,.3);
    \draw (.5,0) -- (.5,.5);
    \draw (1.2,0) -- (1.2,1.2);
    \draw (2,0) -- (2,2);
    %MASSES:
    \filldraw [blue](-1.4,0) circle (.05pt) node[anchor=south] {\tiny $L$};
    \filldraw [blue](-.6,0) circle (.05pt) node[anchor=south] {\tiny$M$};
    \node at (-.15,.4) {\textcolor{blue}{\tiny$S$}};
    \node at (.25,.4) {\textcolor{blue}{\tiny$R_1$}};
    \filldraw [blue](.9,0) circle (.05pt) node[anchor=south] {\tiny$R_2$};
    \filldraw [blue](1.6,0) circle (.05pt) node[anchor=south] {\tiny$R_3$};
    \end{tikzpicture}
    
    \caption{Original configuration}
    \label{fig:5.3.7.1}
    \end{subfigure}
    \hfill
    \begin{subfigure}[h]{0.45\textwidth}
    \centering
    
    \begin{tikzpicture}
     %AXIS AND |x|:
    \draw[stealth-,very thick] (-2.2,2.2) -- (0,0);
    \draw[-stealth,very thick](0,0) -- (2.2,2.2);
    \draw[stealth-stealth,very thick] (-2.5,0) -- (2.5,0);
    %POINTS: (top->down = left->right)
    \filldraw (-1.6,0) circle (1pt) node[anchor=north] {\tiny $-a'$};
    \filldraw [gray](-1.4,0) circle (1pt);
    \filldraw (-.8,0) circle (1pt) node[anchor=north] {\tiny$-b'$};
    \filldraw [gray](-0.7,0) circle (1pt);
    \filldraw (-.3,0) circle (1pt) node[anchor=north] {\tiny$-c$};
    \filldraw (0,0) circle (1pt) node[anchor=north] {\tiny$0$};
    \filldraw [gray](.7,0) circle (1pt);
    \filldraw (.5,0) circle (1pt) node[anchor=north] {\tiny$d'$};
    \filldraw [gray](1.45,0) circle (1pt);   
    \filldraw (1.3,0) circle (1pt) node[anchor=north] {\tiny$e'$};
    \filldraw [gray](2.25,0) circle (1pt);
    \filldraw (2.1,0) circle (1pt) node[anchor=north] {\tiny$f'$};
    %VERTICAL LINES:
    \draw (-1.6,0) -- (-1.6,1.6);
    \draw (-.8,0) -- (-.8,.8);
    \draw (-.3,0) -- (-.3,.3);
    \draw (.5,0) -- (.5,.5);
    \draw (1.3,0) -- (1.3,1.3);
    \draw (2.1,0) -- (2.1,2.1);
    %MASSES:
    \filldraw [blue](-1.25,0) circle (.05pt) node[anchor=south] {\tiny $R_2$};
    \filldraw [blue](-.6,0) circle (.05pt) node[anchor=south] {\tiny$R_1$};
    \node at (-.15,.4) {\textcolor{blue}{\tiny$S$}};
    \node at (.25,.4) {\textcolor{blue}{\tiny$M$}};
    \filldraw [blue](.9,0) circle (.05pt) node[anchor=south] {\tiny$L$};
    \filldraw [blue](1.6,0) circle (.05pt) node[anchor=south] {\tiny$R_3$};
    \end{tikzpicture}
    \caption{New configuration}
    \label{fig:5.3.7.2}
    \end{subfigure}
    \caption{Case 7: transpose $R_1$ with $M$, and also  $R_2$ with $L$.}
    \label{fig:5.3.7.0}
\end{figure}
    
\begin{figure}[H]
    \centering
    \begin{tikzpicture}
    %LINES:
    %x axis and |x|:
    \draw[stealth-,very thick] (-3.5,3.5) -- (0,0);
    \draw[-stealth,very thick](0,0) -- (3.5,3.5);
    \draw[stealth-stealth,very thick] (-4.5,0) -- (4.5,0); 
    
    %mass dividers: (LH -> RH)
    %\draw [dotted,thick](-3.3,0) -- (-3.3, 3.3);
    \draw (-3,0) -- (-3,3);
    \draw [dotted,thick] (-2.5,0) -- (-2.5,2.5);
    \draw (-1.9,0) -- (-1.9,1.9);
    \draw [dotted,thick] (-1.5,0) -- (-1.5,1.5);
    \draw (-1,0) -- (-1,1);
    \draw (1,0) -- (1,1);
    \draw [dotted,thick](1.6,0) -- (1.6,1.6);
    \draw  (2,0) -- (2,2);
    \draw [dotted,thick](2.6,0) -- (2.6,2.6);
    \draw (3.2,0) -- (3.2,3.2);
    \draw [dotted,thick](3.5,0) -- (3.5,3.5);
    %POINTS: 
    %\filldraw (-3.3,0) circle (1.5pt) node[anchor=north] {\scriptsize $-a'$};
    \filldraw (-3,0) circle (1.5pt) node[anchor=north] {\scriptsize $-a'$};
    \filldraw (-2.5,0) circle (1pt) node[anchor=north] {\scriptsize $-a$};
    \filldraw (-1.9,0) circle (1.5pt) node[anchor=north] {\scriptsize$-b'$};
    \filldraw (-1.5,0) circle (1.5pt) node[anchor=north] {\scriptsize$-b$};
    \filldraw (-1,0) circle (1.5pt) node[anchor=north] {\scriptsize $-c$};
    \filldraw (0,0) circle (1pt) node[anchor=north] {\scriptsize $0$};
    \filldraw (1,0) circle (1.5pt) node[anchor=north] {\scriptsize$d'$};
    \filldraw (1.6,0) circle (1pt) node[anchor=north] {\scriptsize$d$};
    \filldraw (2,0) circle (1.5pt) node[anchor=north] {\scriptsize$e'$};
    \filldraw (2.6,0) circle (1pt) node[anchor=north] {\scriptsize$e$};
    \filldraw (3.2,0) circle (1.5pt) node[anchor=north] {\scriptsize$f'$};
    \filldraw (3.5,0) circle (1.5pt) node[anchor=north] {\scriptsize$f$};
    
    %LABELS:
    %BELOW: turning a node label 90 degrees to fit in the column
    %\node[rotate=90] at (-3.1,1){\scriptsize$\updelta_X + \updelta_L$};
    %\node[rotate=90] at (-2.55,1){\scriptsize$X_2-\updelta_L$};
    %\filldraw (-2.55,0) circle (.05pt) node[anchor=south] {\scriptsize$\delta$};
    \node[rotate=90] at (-2.7,1) {\scriptsize$\delta_1 + \delta_2$};
    \node[rotate=90] at (-2.2,1) {\scriptsize$L - \delta_1$};
    \filldraw (-1.7,0) circle (.05pt) node[anchor=south] {\scriptsize$\delta_1$};
    \filldraw (-1.2,0) circle (.05pt) node[anchor=south] {\tiny$M$};
    \filldraw (-.6,0) circle (.05pt) node[anchor=south] {\tiny$S$};
    \filldraw (.5,0) circle (.05pt) node[anchor=south] {\tiny$M$};
    \filldraw (1.3,0) circle (.05pt) node[anchor=south] {\scriptsize$\delta_1$};
    \node[rotate=90] at (1.8,1){\scriptsize$L-\delta_1$};
    %\filldraw (1.8,0) circle (.05pt) node[anchor=south] {\scriptsize$\updelta_L$};
    %\node[rotate=90] at (2.3,1){\scriptsize$X_2-\updelta_L$};
    \node[rotate=90] at (2.3,1){\scriptsize$\delta_1 + \delta_2$};
    %\filldraw (2.3,0) circle (.05pt) node[anchor=south] {\tiny$\delta$};
    \node[rotate=90] at (2.85,1) {\scriptsize$R_3 - (\delta_1 + \delta_2)$};
    %\filldraw (3.35,0) circle (.05pt) node[anchor=south] {\tiny$\delta$};
    \node[rotate=90] at (3.33,1){\scriptsize$\delta_1 + \delta_2$};
    \end{tikzpicture}
    \caption{Case 7: all endpoints and masses identified.}
    \label{fig:5.3.7.3}
\end{figure}
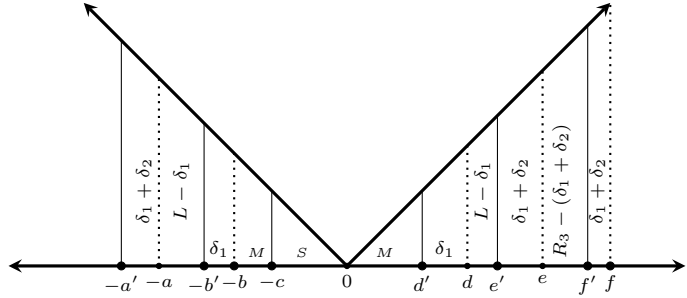

    %\begin{center}
    %    \includegraphics[scale=.4]{MainPaper/IntervalFlippingCase7.png}
    %\end{center}
    Specifically, we can let $\delta_1 = R_1 - M$ and let $\delta_2 = R_2 - L$. In the figures above, we see the original configuration (with perimeter $a + b + c + d + e + f$) as well as the new configuration (with perimeter $a' + b' c + d' + e' + f'$).
    
    Since $b > 0$ and $[-b, -c]$ has the same mass as $[0,d]$, can conclude that $b > d'$. This implies that $b' - b < d - d'$, and therefore that $b' + d' < b + d$. Similarly, we can conclude that $a > e'$, leading to $a' - a < e - e'$ and $a' + e' < a + e$. Finally, we can conclude that $f' < f$ automatically. Taking these inequalities together, we can conclude that our total perimeter has decreased after our reconfiguration. Applying this rearrangement returns us to a framework that is identical to Case 0.
    
    \item Case 8: $ R_1 \leq R_2 \leq L  \leq R_3$. In this case we can transpose $M$ and $R_1$, leaving the others in their current spots. The process is shown in the picture below.
    
         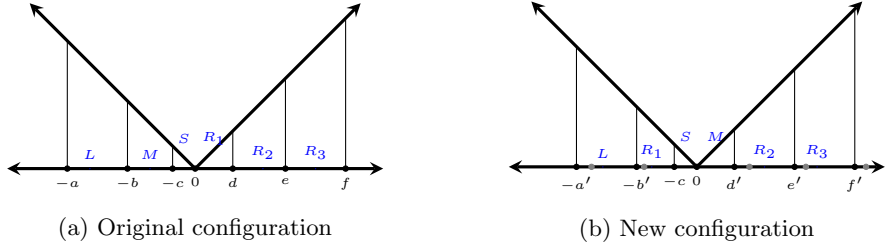
\begin{figure}[H]
    \centering
    \begin{subfigure}[h]{0.45\textwidth}
         \centering
         \begin{tikzpicture}
    %AXIS AND |x|:
    \draw[stealth-,very thick] (-2.2,2.2) -- (0,0);
    \draw[-stealth,very thick](0,0) -- (2.2,2.2);
    \draw[stealth-stealth,very thick] (-2.5,0) -- (2.5,0);
    %POINTS: (top->down = left->right)
    \filldraw (-1.7,0) circle (1pt) node[anchor=north] {\tiny $-a$};
    \filldraw (-.9,0) circle (1pt) node[anchor=north] {\tiny$-b$};
    \filldraw (-.3,0) circle (1pt) node[anchor=north] {\tiny$-c$};
    \filldraw (0,0) circle (1pt) node[anchor=north] {\tiny$0$};
    \filldraw (.5,0) circle (1pt) node[anchor=north] {\tiny$d$};
    \filldraw (1.2,0) circle (1pt) node[anchor=north] {\tiny$e$};
    \filldraw (2,0) circle (1pt) node[anchor=north] {\tiny$f$};
    %VERTICAL LINES:
    \draw (-1.7,0) -- (-1.7,1.7);
    \draw (-.9,0) -- (-.9,.9);
    \draw (-.3,0) -- (-.3,.3);
    \draw (.5,0) -- (.5,.5);
    \draw (1.2,0) -- (1.2,1.2);
    \draw (2,0) -- (2,2);
    %MASSES:
    \filldraw [blue](-1.4,0) circle (.05pt) node[anchor=south] {\tiny $L$};
    \filldraw [blue](-.6,0) circle (.05pt) node[anchor=south] {\tiny$M$};
    \node at (-.15,.4) {\textcolor{blue}{\tiny$S$}};
    \node at (.25,.4) {\textcolor{blue}{\tiny$R_1$}};
    \filldraw [blue](.9,0) circle (.05pt) node[anchor=south] {\tiny$R_2$};
    \filldraw [blue](1.6,0) circle (.05pt) node[anchor=south] {\tiny$R_3$};
    \end{tikzpicture}
    
    \caption{Original configuration}
    \label{fig:5.3.8.1}
    \end{subfigure}
    \hfill
    \begin{subfigure}[h]{0.45\textwidth}
    \centering
    
    \begin{tikzpicture}
     %AXIS AND |x|:
    \draw[stealth-,very thick] (-2.2,2.2) -- (0,0);
    \draw[-stealth,very thick](0,0) -- (2.2,2.2);
    \draw[stealth-stealth,very thick] (-2.5,0) -- (2.5,0);
    %POINTS: (top->down = left->right)
    \filldraw (-1.6,0) circle (1pt) node[anchor=north] {\tiny $-a'$};
    \filldraw [gray](-1.4,0) circle (1pt);
    \filldraw (-.8,0) circle (1pt) node[anchor=north] {\tiny$-b'$};
    \filldraw [gray](-0.7,0) circle (1pt);
    \filldraw (-.3,0) circle (1pt) node[anchor=north] {\tiny$-c$};
    \filldraw (0,0) circle (1pt) node[anchor=north] {\tiny$0$};
    \filldraw [gray](.7,0) circle (1pt);
    \filldraw (.5,0) circle (1pt) node[anchor=north] {\tiny$d'$};
    \filldraw [gray](1.45,0) circle (1pt);   
    \filldraw (1.3,0) circle (1pt) node[anchor=north] {\tiny$e'$};
    \filldraw [gray](2.25,0) circle (1pt);
    \filldraw (2.1,0) circle (1pt) node[anchor=north] {\tiny$f'$};
    %VERTICAL LINES:
    \draw (-1.6,0) -- (-1.6,1.6);
    \draw (-.8,0) -- (-.8,.8);
    \draw (-.3,0) -- (-.3,.3);
    \draw (.5,0) -- (.5,.5);
    \draw (1.3,0) -- (1.3,1.3);
    \draw (2.1,0) -- (2.1,2.1);
    %MASSES:
    \filldraw [blue](-1.25,0) circle (.05pt) node[anchor=south] {\tiny $L$};
    \filldraw [blue](-.6,0) circle (.05pt) node[anchor=south] {\tiny$R_1$};
    \node at (-.15,.4) {\textcolor{blue}{\tiny$S$}};
    \node at (.25,.4) {\textcolor{blue}{\tiny$M$}};
    \filldraw [blue](.9,0) circle (.05pt) node[anchor=south] {\tiny$R_2$};
    \filldraw [blue](1.6,0) circle (.05pt) node[anchor=south] {\tiny$R_3$};
    \end{tikzpicture}
    \caption{New configuration}
    \label{fig:5.3.8.2}
    \end{subfigure}
    \caption{Case 8: transpose $R_1$ with $M$.}
    \label{fig:5.3.8.0}
\end{figure}
    
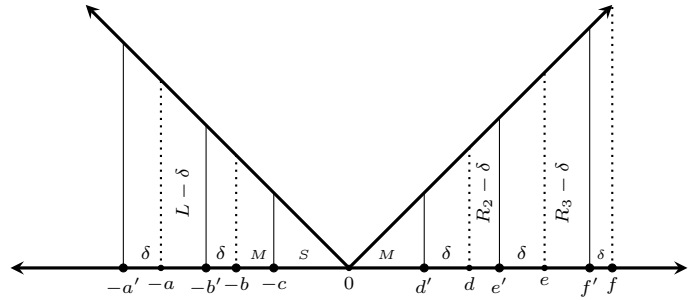
\begin{figure}[H]
    \centering
    \begin{tikzpicture}
    %LINES:
    %x axis and |x|:
    \draw[stealth-,very thick] (-3.5,3.5) -- (0,0);
    \draw[-stealth,very thick](0,0) -- (3.5,3.5);
    \draw[stealth-stealth,very thick] (-4.5,0) -- (4.5,0); 
    
    %mass dividers: (LH -> RH)
    %\draw [dotted,thick](-3.3,0) -- (-3.3, 3.3);
    \draw (-3,0) -- (-3,3);
    \draw [dotted,thick] (-2.5,0) -- (-2.5,2.5);
    \draw (-1.9,0) -- (-1.9,1.9);
    \draw [dotted,thick] (-1.5,0) -- (-1.5,1.5);
    \draw (-1,0) -- (-1,1);
    \draw (1,0) -- (1,1);
    \draw [dotted,thick](1.6,0) -- (1.6,1.6);
    \draw  (2,0) -- (2,2);
    \draw [dotted,thick](2.6,0) -- (2.6,2.6);
    \draw (3.2,0) -- (3.2,3.2);
    \draw [dotted,thick](3.5,0) -- (3.5,3.5);
    %POINTS: 
    %\filldraw (-3.3,0) circle (1.5pt) node[anchor=north] {\scriptsize $-a'$};
    \filldraw (-3,0) circle (1.5pt) node[anchor=north] {\scriptsize $-a'$};
    \filldraw (-2.5,0) circle (1pt) node[anchor=north] {\scriptsize $-a$};
    \filldraw (-1.9,0) circle (1.5pt) node[anchor=north] {\scriptsize$-b'$};
    \filldraw (-1.5,0) circle (1.5pt) node[anchor=north] {\scriptsize$-b$};
    \filldraw (-1,0) circle (1.5pt) node[anchor=north] {\scriptsize $-c$};
    \filldraw (0,0) circle (1pt) node[anchor=north] {\scriptsize $0$};
    \filldraw (1,0) circle (1.5pt) node[anchor=north] {\scriptsize$d'$};
    \filldraw (1.6,0) circle (1pt) node[anchor=north] {\scriptsize$d$};
    \filldraw (2,0) circle (1.5pt) node[anchor=north] {\scriptsize$e'$};
    \filldraw (2.6,0) circle (1pt) node[anchor=north] {\scriptsize$e$};
    \filldraw (3.2,0) circle (1.5pt) node[anchor=north] {\scriptsize$f'$};
    \filldraw (3.5,0) circle (1.5pt) node[anchor=north] {\scriptsize$f$};
    
    %LABELS:
    %BELOW: turning a node label 90 degrees to fit in the column
    %\node[rotate=90] at (-3.1,1){\scriptsize$\updelta_X + \updelta_L$};
    %\node[rotate=90] at (-2.55,1){\scriptsize$X_2-\updelta_L$};
    \filldraw (-2.7,0) circle (.05pt) node[anchor=south] {\scriptsize$\delta$};
    %\node[rotate=90] at (-2.7,1) {\scriptsize$\delta_1 + \delta_2$};
    \node[rotate=90] at (-2.2,1) {\scriptsize$L - \delta$};
    \filldraw (-1.7,0) circle (.05pt) node[anchor=south] {\scriptsize$\delta$};
    \filldraw (-1.2,0) circle (.05pt) node[anchor=south] {\tiny$M$};
    \filldraw (-.6,0) circle (.05pt) node[anchor=south] {\tiny$S$};
    \filldraw (.5,0) circle (.05pt) node[anchor=south] {\tiny$M$};
    \filldraw (1.3,0) circle (.05pt) node[anchor=south] {\scriptsize$\delta$};
    \node[rotate=90] at (1.8,1){\scriptsize$R_2-\delta$};
    \filldraw (2.3,0) circle (.05pt) node[anchor=south] {\scriptsize$\delta$};
    %\node[rotate=90] at (2.3,1){\scriptsize$X_2-\updelta_L$};
    %\node[rotate=90] at (2.3,1){\scriptsize$\delta_1 + \delta_2$};
    %\filldraw (2.3,0) circle (.05pt) node[anchor=south] {\tiny$\delta$};
    \node[rotate=90] at (2.85,1) {\scriptsize$R_3 - \delta$};
    \filldraw (3.35,0) circle (.05pt) node[anchor=south] {\tiny$\delta$};
    %\node[rotate=90] at (3.33,1){\scriptsize$\delta_1 + \delta_2$};
    \end{tikzpicture}
    \caption{Case 8: all endpoints and masses identified.}
    \label{fig:5.3.8.3}
\end{figure}       
    
    %\begin{center}
    %    \includegraphics[scale=.4]{MainPaper/IntervalFlippingCase8.png}
    %\end{center}
    Specifically, we can let $\delta = R_1 - M$. The figures above show both the original configuration (with perimeter $a + b + c + d + e + f$) as well as the new configuration (with perimeter $a' + b' + c + d' + e' + f'$). Since $S \geq 0$ and the intervals $[-b, -c]$ and $[0, d']$ have the same mass, we know that $b \geq d'$. This allows us to conclude that the interval $[-b', -b]$ is narrower than (or the same size as) the interval  $[d', d]$, as both contain the same mass. We conclude that $b' - b \leq d - d'$, and therefore that $b' + d' \leq b + d$. Using similar reasoning (and the fact that $L > R_2$), we can deduce that $a' + e' < a + e$. Finally, we can identify that $f' < f$. Taken together, we can conclude that the total perimeter is smaller after moving to the  new configuration. Applying this rearrangement returns us to a framework that is identical to Case 0.

    \item Case 9: $R_1 \leq R_2 \leq R_3 \leq L$. In this case we can transpose $L$ and $R_1$, leaving the others in their current spots. The associated picture and proof are identical to that of Case 8: however, after considering the new ordering of intervals, we can conclude that have reconfigured into Case 1 rather than Case 0.
    
%    \begin{center}
%        \includegraphics[scale=1]{MainPaper/fillerpicture.png}
%    \end{center}
    
%    \textcolor{red}{NOTE: PICTURE IS IDENTICAL TO CASE 8}
    
%    Specifically, we can let $\delta = R_1 - M$. The left picture shows the original configuration (with perimeter $a + b + c + d + e + f$), while the middle picture shows the new configuration (with perimeter $a' + b' + c + d' + e' + f'$). The picture on the right identifies all endpoints and corresponding masses. We can first make the observation that, since $S > 0$, the endpoint $b$ is further away from the origin than $d'$ is, and can conclude that $b' - b < d - d'$ and therefore that $b' + d' < b + d$. Using similar reasoning (and the fact that $L > R_2$), we can deduce that $a' - a < e - e'$. Finally, we can see directly that $f' < f$. Taking these inequalities together, we conclude that reconfiguring our intervals results in lower perimeter. After considering the new ordering of intervals, we can conclude that we are in Case 1. 
\end{itemize}

\end{proof}

\bibliographystyle{plain}
\bibliography{2021SCOPEBubble}

\end{document}